\documentclass[a4paper, 11pt,reqno]{amsart}
\usepackage[top = 1in, bottom = 1in, left=2.5cm, right=2.5cm]{geometry}

\usepackage{amssymb,amsmath,amsthm,graphicx,xcolor,mathtools,enumerate,mathrsfs}
\usepackage{tabularx}
\usepackage[shortlabels]{enumitem}
\usepackage[british]{babel}
\usepackage[utf8]{inputenc}
\usepackage{bbm}
\usepackage{tikz}
\usepackage{aligned-overset}
\usepackage{thm-restate}

\usepackage{url}

\allowdisplaybreaks

\usepackage[mathlines]{lineno}
\usepackage{etoolbox} 

\newcommand*\linenomathpatch[1]{%
	\expandafter\pretocmd\csname #1\endcsname {\linenomath}{}{}%
	\expandafter\pretocmd\csname #1*\endcsname{\linenomath}{}{}%
	\expandafter\apptocmd\csname end#1\endcsname {\endlinenomath}{}{}%
	\expandafter\apptocmd\csname end#1*\endcsname{\endlinenomath}{}{}%
}
\newcommand*\linenomathpatchAMS[1]{%
	\expandafter\pretocmd\csname #1\endcsname {\linenomathAMS}{}{}%
	\expandafter\pretocmd\csname #1*\endcsname{\linenomathAMS}{}{}%
	\expandafter\apptocmd\csname end#1\endcsname {\endlinenomath}{}{}%
	\expandafter\apptocmd\csname end#1*\endcsname{\endlinenomath}{}{}%
}

\expandafter\ifx\linenomath\linenomathWithnumbers
\let\linenomathAMS\linenomathWithnumbers
\patchcmd\linenomathAMS{\advance\postdisplaypenalty\linenopenalty}{}{}{}
\else
\let\linenomathAMS\linenomathNonumbers
\fi

\linenomathpatchAMS{gather}
\linenomathpatchAMS{multline}
\linenomathpatchAMS{align}
\linenomathpatchAMS{alignat}
\linenomathpatchAMS{flalign}
\linenomathpatch{equation}

\usepackage{hyperref}
\hypersetup{colorlinks, linkcolor={red!50!black}, citecolor={green!50!black}, urlcolor={blue!50!black}}

\usepackage{caption}
\usepackage{subcaption}

\usepackage{cleveref}
\theoremstyle{plain}
\newtheorem{theorem}{Theorem}[section]
\crefname{theorem}{Theorem}{Theorems}

\newtheorem{proposition}[theorem]{Proposition}
\crefname{proposition}{Proposition}{Propositions}

\newtheorem{corollary}[theorem]{Corollary}
\crefname{corollary}{Corollary}{Corollaries}

\newtheorem{lemma}[theorem]{Lemma}
\crefname{lemma}{Lemma}{Lemmas}

\crefname{conjecture}{Conjecture}{Conjectures}

\crefname{problem}{Problem}{Problem}

\newtheorem{claim}[theorem]{Claim}
\crefname{claim}{Claim}{Claims}

\newtheorem{observation}[theorem]{Observation}
\crefname{observation}{Observation}{Observations}

\crefname{setup}{Setup}{Setups}

\crefname{fact}{Fact}{Facts}

\crefname{algorithm}{Algorithm}{Algorithms}

\crefname{remark}{Remark}{Remarks}

\crefname{example}{Example}{Examples}

\theoremstyle{definition}
\newtheorem{definition}[theorem]{Definition}
\crefname{definition}{Definition}{Definitions}

\crefname{construction}{Construction}{Constructions}

\crefname{question}{Question}{Questions}

\numberwithin{equation}{section}

\crefname{section}{Section}{Sections}
\crefname{appendix}{Appendix}{Appendix}

\newcommand{\rf}[1]{\cref{#1} (\nameref*{#1})}
\newcommand{\fr}[1]{\nameref*{#1} (\cref{#1})}

\definecolor{DarkDesaturatedBlue}{HTML}{3A3556}
\definecolor{VividOrange}{HTML}{F15918}
\definecolor{PureOrange}{HTML}{FFBA00}
\definecolor{LightGrayishPink}{HTML}{EEC5D5}
\definecolor{VerySoftBlue}{HTML}{B5AFDB}

\newcommand{\vecb}{\mathbf}

\DeclareMathOperator{\MT}{\delta}

\newcommand{\es}{\emptyset}
\newcommand{\eps}{\varepsilon}
\renewcommand{\rho}{\varrho}
\newcommand{\sm}{\setminus}
\renewcommand{\subset}{\subseteq}

\newcommand{\NATS}{\mathbb{N}}
\newcommand{\NATSZ}{\mathbb{N}_0}
\newcommand{\INTS}{\mathbb{Z}}
\newcommand{\REALS}{\mathbb{R}}

\newcommand{\maxnorm}[1]{\lVert #1 \rVert_{\infty}}
\newcommand{\Maxnorm}[1]{\left\lVert #1 \right\rVert_{\infty}}
\newcommand{\onenorm}[1]{\lVert #1 \rVert_{1}}

\DeclareMathOperator{\Hom}{{Hom}}

\DeclareMathOperator{\ord}{\mathbf{ord}}

\newcommand{\Prob}{\mathbb{P}}
\newcommand{\Exp}{\mathbb{E}}

\DeclareMathOperator{\crit}{\chi_{\rm cr}}
\DeclareMathOperator{\cFcr}{\cF_{\rm {cr}}}

\newcommand{\fc}{flexi-chromatic}
\newcommand{\Fc}{Flexi-chromatic}

\newcommand{\hcf}{\text{gcd}_\chi}
\newcommand{\hcfc}{\text{gcd}_\text{c}}

\newcommand\restr[2]{{
  \left.\kern-\nulldelimiterspace 
  #1 
  \vphantom{\big|} 
  \right|_{#2} 
  }}

\def\le{\leqslant}
\def\leq{\leqslant}
\def\ge{\geqslant}
\def\geq{\geqslant}

\newcommand{\cB}{\mathcal{B}}

\newcommand{\cD}{\mathcal{D}}

\newcommand{\cF}{\mathcal{F}}

\newcommand{\cH}{\mathcal{H}}

\newcommand{\cJ}{\mathcal{J}}
\newcommand{\cK}{\mathcal{K}}
\newcommand{\cL}{\mathcal{L}}

\newcommand{\cP}{\mathcal{P}}

\newcommand{\cR}{\mathcal{R}}

\newcommand{\cV}{\mathcal{V}}
\newcommand{\cW}{\mathcal{W}}
\newcommand{\cX}{\mathcal{X}}

\newcommand{\tth}{^\text{th}}
\newcommand{\tX}{\tilde X}
\newcommand{\tcX}{\tilde{\mathcal{X}}}
\newcommand{\pri}{\pi}
\newcommand{\kk}{{\lceil \chi \rceil}}

\def\COMMENT#1{}

\newcommand{\DeltaRp}{\Delta_{R'}}

\newcommand{\NEW}[1]{#1}
\renewcommand{\NEW}[1]{{\color{magenta}#1}}

\newenvironment{proofclaim}[1][Proof of the claim]{\begin{proof}[#1]}{\end{proof}}

\title[Sufficient conditions for perfect mixed tilings]{Sufficient conditions for perfect mixed tilings}
\date{\today}

\author[E.~Hurley]{Eoin Hurley}
\address[E.~Hurley $\vert$ F.~Joos]{Heidelberg University,
	Institute for Computer Science,
	Im Neuenheimer Feld 205,
	69120 Heidelberg, Germany}
\email{(hurley$\vert$joos)@informatik.uni-heidelberg.de}

\author[F.~Joos]{Felix Joos}

\author[R.~Lang]{Richard Lang}
\address[R.~Lang]{Departament de Matemàtiques, Universitat Politècnica de Catalunya, 08034 Barcelona, Spain}
\email{richard.johannes.lang@upc.edu}

\thanks{The research leading to these results was partially supported by the Deutsche Forschungsgemeinschaft (DFG, German Research Foundation) -- 428212407 and 450397222.
}

\begin{document}

\begin{abstract}
	We develop a method to study sufficient conditions for perfect mixed tilings.
	Our framework allows the embedding of bounded degree graphs $H$ with components of sublinear order.
	As a corollary, we recover and extend the work of K\"uhn and Osthus regarding sufficient minimum degree conditions for perfect $F$-tilings (for an arbitrary fixed graph $F$)
	by replacing the $F$-tiling with the aforementioned graphs $H$.
	Moreover, we obtain analogous results for degree sequences and in the setting of uniformly dense graphs.
	Finally, we asymptotically resolve a conjecture of Koml\'os in a strong sense.
\end{abstract}

\maketitle
\thispagestyle{empty}
\vspace{-0.4cm}

\section{Introduction}\label{sec:introduction}
The search for (optimal) conditions that force  a particular graph to contain another graph as a subgraph is arguably among the most central topics in graph theory.
Consider a host graph~$G$ on $n$ vertices and a guest graph $H$ on at most $n$ vertices (throughout this article we aim to embed $H$ into~$G$).
In this setting, one of the most natural conditions to investigate is surely a minimum degree condition of the host graph.
We define the \emph{minimum degree threshold}  for containing~$H$ as
\begin{align*}
	\MT(n,H)=\min\{m\colon \delta(G)\geq m \text{ implies that $H$ is a subgraph of $G$}\}.
\end{align*}
A classic result of Dirac determines the minimum degree threshold for perfect matchings.
More generally, for a graph $F$, we say that $H$ is an $F$-tiling if $H$ is the disjoint union of copies of $F$;
each such copy of $F$ in $H$ is called a \emph{tile} of $H$.
Moreover, $H$ is a \emph{perfect $F$-tiling} (in $G$) if the order $v(H)$ of $H$ equals $v(G)$.
Hajnal and Szemer\'edi~\cite{HS70} extended Dirac's theorem by proving that $\MT(n,H) = (1-1/k)n$ for perfect $K_k$-tilings $H$ {for all $n$ divisible by $k$.}
We write $\chi(J)$ for the \emph{chromatic number} of a graph $J$ and suppose in the following that $\chi(F)=k\geq 2$.
Resolving the Alon--Yuster conjecture~\cite{AY96},
Koml\'os, S\'ark\"ozy and Szemer\'edi~\cite{KSS01} proved that $\MT(n,H) \leq (1-1/k)n+O(1)$ for perfect $F$-tilings $H$ (for any fixed $F$) {when $n$ is large enough and divisible by $v(F)$};
however, it was known that for some graphs, such as odd cycles of length at least five, this bound was not optimal~\cite{Abb98}.

Thus, the chromatic number of $F$ is not always the correct parameter to determine the asymptotic behaviour of $\MT(n,H)$ for perfect $F$-tilings $H$.
To plug this gap, Koml\'os~\cite{Kom00} introduced the critical chromatic number of a graph.
Let $\alpha(F)$ be the minimum proportion of vertices in a single colour class over all proper $k$-colourings of $F$.
The \emph{critical chromatic number} of~$F$ is
\begin{align}\label{equ:crit}
	\crit(F) = k-1 + \frac{\alpha(F)}{\frac{1}{k-1} (1- \alpha(F))} = {  \frac{k -1}{1-\alpha(F)}};
\end{align}
in other words, the critical chromatic number is $k-1$ plus the ratio of the size of the smallest colour class to the average of the other colour classes (in a proper $k$-colouring minimising the size of the smallest colour class).\footnote{To illustrate this parameter further, consider the complete $k$-partite graph $K_{a,(k-1)\ast b}$ whose parts have sizes $a,b,\dots,b$ with $a\leq b$.
	(In the literature, $K_{a,(k-1)\ast b}$ is often called a \emph{bottle graph}.)
	Then $\crit(K_{a,(k-1)\ast b}) = k-1 + \frac{a}{b}$.
	So by sliding~$\frac{a}{b}$ within $[0,1]$, the critical chromatic number transitions smoothly between $k-1$ and $k$.
	Conversely, for every graph $F$ with $\chi(F)=k$, there exist $a$, $b$ such that $\crit(F) = \crit(K_{a,(k-1)\ast b})$.}
Appropriately chosen complete $k$-partite graphs show that $\MT(n,H) \geq (1-1/\crit(F))n $ for all perfect $F$-tilings $H$.
On the other hand, Koml\'os~\cite{Kom00} proved that $\MT(n,H) \leq (1-1/\crit(F))n$ whenever $H$ is an $F$-tiling on $n-o(n)$ vertices.
Shokoufandeh and Zhao~\cite{SZ03} showed that the same conclusion is true whenever~$H$ is an $F$-tiling on $n-O(1)$ vertices, thereby confirming a conjecture of Koml\'os.
Finally, resolving a major question in this context,
K\"uhn and Osthus~\cite{KO09} determined $\MT(n,H)$ for perfect $F$-tilings~$H$ up to a constant depending only on $F$.

To state the K\"uhn--Osthus Tiling Theorem, we introduce the class of graphs $\cFcr$.
Define $\cD(F) \subset~\NATSZ$ to be the union of all integers {$||\phi^{-1}(i)|-|\phi^{-1}(j)||$ over all proper $k$-colourings $\phi$ of $F$ and all colours $1 \leq i < j \leq k$.}
We denote by $\hcf(F)$ the greatest common divisor of the integers in $\cD(F)\sm \{0\}$ with the convention that $\hcf(F)=\infty$ if $\cD(F)=\{0\}$.
We write $\hcfc(F)$ for the greatest common divisor of all the orders of the components of $F$ (note that $F$ may not be connected).
Finally, let
\begin{align}\label{eq:cFcr}
	F\in \cFcr \text{ if }\
	\begin{cases}
		\chi(F)=2 \text{ and } \hcf(F)=\hcfc(F)=1 \text{, or} \\
		\chi(F)\geq 3 \text{ and } \hcf(F)=1.
	\end{cases}
\end{align}

Given this, K\"uhn and Osthus~\cite{KO09} showed that $\MT(n,H) = (1-1/\crit(F))n + O(1)$ for all perfect $F$-tilings $H$ on $n$ vertices if $F \in \cFcr$ and otherwise $\MT(n,H) = (1-1/\chi(F))n + O(1)$.\footnote{{For some graphs $F$, the term $O(1)$ can even be dropped.}}
Hence, the function $\MT(n,H)$ is well-understood for all graphs $H$ on $n$ vertices that can be partitioned into isomorphic tiles each containing a fixed number of vertices.

\smallskip

The contribution of this paper is to extend the above results in several directions:
firstly, we do not require that the tiles of $H$ are isomorphic,
secondly, the order of the tiles may grow with $n$,
and thirdly, we develop a more general framework beyond minimum degree conditions.
The latter allows us, for example, to easily obtain similar results for host graphs with particular degree sequences and in the setting of uniformly dense graphs.
Along the way, we resolve a conjecture of Koml\'os asymptotically in a strong form.

\smallskip

We continue with four implications of our main result (\cref{thm:frameworks}).
To this end, let us introduce some further notation.
Given a family of graphs~$\cF$,
a graph $H$ is an \emph{$\cF$-tiling} if~$H$ is the disjoint union of copies of graphs in~$\cF$;
again, we call each of these copies a \emph{tile} of~$H$.
Let~$\cF(k,w)$ be the family of properly $k$-colourable graphs of order at most $w$
and let $\cFcr(k,w)=\cFcr\cap \cF(k,w)$.

Initial progress towards perfect mixed tilings was obtained by Ishigami~\cite{Ish02}, who proved an analogue of Koml\'os' theorem in this setting.
More precisely, it was shown that a graph $G$ with $\delta(G) \geq (1-1/\chi)n$ contains any $\cF$-tiling $H$ {on $n-o(n)$ vertices} with $\Delta(H)=O(1)$ and $\crit(H)=\chi$ where the tiles of $\cF$ are allowed to be of order $o(n)$.
Our first result generalises this to perfect tilings.
More precisely, we extend the K\"uhn--Osthus Tiling Theorem by replacing perfect $F$-tilings for $F\in \cFcr$ (where $F$ is fixed) by perfect $\cFcr(k,\sqrt{n}/\log n)$-tilings for any (fixed) {$k\geq 2$.}

\begin{theorem}\label{thm:extended-KO}
	For all $\chi,\Delta > 1$, $\mu >0$ and $n$ sufficiently large, the following holds.
	Let $H$ be an $\cFcr(\lceil \chi \rceil, \sqrt{n}/\log n)$-tiling on $n$ vertices with $\crit(H)\leq \chi$ and $\Delta(H) \leq \Delta$.
	Then $\MT(n,H) \leq (1-1/\chi + \mu)n$.
\end{theorem}

It is worth noting that the upper bound regarding the order of the tiles is (essentially) tight;
while it can be relaxed to $o(\sqrt{n})$ if all components of $H$ are pairwise isomorphic {(see the remark after the proof of \cref{thm:extended-KO} in \cref{sec:2})},
the theorem becomes false (even in the case when all components are isomorphic)
if tiles of order $10\sqrt{n}$ are permitted {(see \cref{sec:sizes-tiles}).}
We also point out that in the statement of \cref{thm:extended-KO},
we in fact only need to assume that a small fraction of (small) tiles belong to  $\cFcr(\lceil \chi \rceil, \sqrt{n}/\log n)$ and all others to $\cF(\lceil \chi \rceil, \sqrt{n}/\log n)$.

{Before we continue, let us make two further remarks.
	First, the original K\"uhn--Osthus Tiling Theorem is proved in terms of an (additive) constant error, which is a more precise statement.
	Our framework does not imply a version of this result, but it is conceivable that it can be obtained with additional work.
	We return to this problem in the conclusion.
	Second, the K\"uhn--Osthus Tiling Theorem also covers the case of arbitrary $F$, that is when $F \notin \cFcr$.
	Here, the necessary degree is determined by $\chi(F)$ as opposed to $\crit(F)$.
	Our main theorem can be used to recover such a statement with linear error term.
	We omit however the details, since this situation is already covered by the Bandwidth Theorem of Böttcher, Schacht and Taraz~\cite{BST09}. (See also \cite{LS23} for extensions beyond the minimum degree setting.)}

Our second result concerns mixed tilings which are near-perfect in the sense that the number of vertices they miss in the host graph is only a multiple of the maximum order of a tile.
Koml\'os~\cite{Kom00} conjectured the following for this setup:
Let $H$ be a graph on $n-f(w)$ vertices with components $F_1,\ldots, F_r$ where $f(\cdot)$ is a function of the maximum component order $w$.
Then
\begin{align}\label{eq:Kom_conj}
	\MT(n,H)\leq \sum_{i\in [r]}v(F_i)\left(1- \frac{1}{\crit(F_i)}\right) + f(w).
\end{align}
We remark that Koml\'os' conjecture, as it is stated, fails in general.\footnote{Suppose $H$ is the union of $4$-cycles on in total $n-C$ vertices and one $K_4$.
	{Note that a balanced complete tripartite graph does not contain a copy of $K_4$.
		Hence $\MT(n,H) \geq \frac{2}{3}n$.
		But~\eqref{eq:Kom_conj} claims that $\MT(n,H)\leq \frac{n-C}{2} + f(w)$.}
}
Moreover, he did not mention the function $f(\cdot)$, but he notes only a few lines earlier that it is even necessary when all $F_i$ are isomorphic; so we think it is likely that he intended to have $f(\cdot)$ included here as well as otherwise the conjecture is false for even more reasons.
However, when $\chi(F_i) = \lceil \crit(H) \rceil$ for all $i\in [r]$, one can rewrite~\eqref{eq:Kom_conj} (see \cref{sec:reformulating-komlos-conjecture}) to obtain\COMMENT{
	Let $n'= v(H)$, $v_i=v(F_i)$, $a_i=a(F_i)$, $\alpha_i=\alpha(F_i)$ and recall that $\alpha(F_i)=\frac{a_i}{v_i}$ and $\crit=\frac{k-1}{1-\alpha_i}$.
	$\sum_i v_i(1-\frac{1-\alpha_i}{k-1})
	=n' - \frac{1}{k-1}\sum_i v_i(1-\alpha_i)
	=n' - \frac{1}{k-1}(n'-\sum_i a_i)
	= n' - \frac{n'-a(H)}{k-1}
	= n' - n'\frac{1-\alpha(H)}{k-1}
	=n'(1- 1/\crit(H))$.
}
\begin{align}\label{eq:Kom_conj2}
	\MT(n,H)\leq \left(1- \frac{1}{\crit(H)}\right)n + \frac{f(w)}{\crit(H)}.
\end{align}

If true, this would in particular generalise the result of Shokoufandeh and Zhao.
Note that the work of Ishigami~\cite{Ish02} implies the conjecture in the special case when $w$ is linear in $n$.
Here, we confirm Koml\'os' conjecture asymptotically for all $w=o(n)$ in a strong form.
{We remark that it seems plausible that one can improve the dependencies of $C$ to $C= O(\kk)$ in \cref{thm:miss_constant_frac-simple}.}

\begin{theorem}\label{thm:miss_constant_frac-simple}
	For all $\chi, \Delta > 1$, $\mu>0$,
	there exist $C=C(\lceil\chi\rceil,\Delta), \xi=\xi(\lceil\chi\rceil,\Delta,\mu)>0$ such that for sufficiently large $n$ the following holds.
	Suppose that $w \leq \xi n$ and~$H$ is an $ \cF(\lceil \chi\rceil, w)$-tiling on at most $n-C w$ vertices with $\crit(H)\leq\chi$ and $\Delta(H)\leq \Delta$.
	Then $\MT(n,H) \leq (1- {1}/{\chi}+\mu)n$.
\end{theorem}

A third implication of our main result concerns host graphs with particular degree sequences.
In the tradition of Pósa's theorem~\cite{Pos62} for Hamilton cycles, Hyde, Liu and Treglown~\cite{HLT19} investigated degree sequences that force perfect tilings.
{Note that for a graph $F$, solving \eqref{equ:crit} for $\alpha$ gives $\alpha(F) = 1- (\lceil \crit(F)\rceil-1)/\crit(F)$.
	Motivated by this, we define $\alpha = \alpha(\chi) = 1 - (\lceil \chi \rceil-1)/\chi$ for given $\chi > 1$.}
For $\mu > 0$, we say that the degree sequence $d_1\leq \dots \leq d_n$ of a graph $G$ is $(\chi,\mu)$-\emph{strong} if
\begin{equation}\label{equ:deg-seq-def}
	d_i \geq \left(1- \frac{1}{\chi} - \alpha\right)n +  \chi\alpha \cdot i + \mu n\quad  \text{ for all } {1 \leq i \leq n/\chi.}
\end{equation}
In particular, if $\delta(G) \geq (1-1/\chi + \mu) n$, then $G$ has a $(\chi,\mu)$-\emph{strong} degree sequence.
Hyde, Liu and Treglown~\cite{HLT19} proved an extension of Koml\'os' theorem,
which states that a graph~$G$ on~$n$ vertices with a {$(\chi,\mu)$-strong} degree sequence contains any $F$-tiling on $n-o(n)$ vertices for fixed $F$.
In another paper, Hyde and Treglown~\cite{HT20} subsequently strengthened this to perfect $F$-tilings, {which constitutes a generalisation of the Kühn-Osthus Theorem with linear error term~\cite{KO09}.
	(It is an open question, whether the work of Hyde and Treglown can be improved to a constant error term.)}
We use the first result of Hyde, Liu and Treglown~\cite{HLT19} as a black box to give the following extension of the latter result~\cite{HT20} (and \cref{thm:extended-KO}).

\begin{theorem}\label{thm:extended-KO-deg-seq}
	For all $\chi,\Delta > 1$, $\mu >0$ and $n$ sufficiently large, the following holds.
	Let $H$ be an $\cFcr(\lceil \chi \rceil,\sqrt{n}/\log n)$-tiling on $n$ vertices with $\crit(H)\leq \chi$ and $\Delta(H) \leq \Delta$
	and let $G$ be a graph on $n$ vertices with a $(\chi,\mu)$-strong degree sequence.
	Then $H \subset G$.
\end{theorem}

An extension of Koml\'os' theorem concerns optimal degree conditions for embedding tilings~$H$ that cover a fixed proportion of the host graph $G$~\cite{Kom00}.
Similar questions, were later studied for mixed tilings by Ishigami~\cite{Ish02b} and in other settings by Piguet and Saumell~\cite{PS19}.
We remark that one can derive analogous (optimal) results of this type from \cref{thm:extended-KO-deg-seq} by adding isolated vertices to $H$ until it has the order of $G$.

Our fourth implication concerns host graphs that exhibit weak quasirandom properties.
For a set of vertices $U$ in a graph $G$, we denote by $e(U)$ the number of edges of $G$ which are contained in $G[U]$.
For $\rho,d >0$, we say that a graph $G$ on $n$ vertices is \emph{(uniformly)} $(\rho,d)$-\emph{dense} if $e(U) \geq d|U|^2/2-\rho n^2$ for every $U \subset V(G)$.
Note that $G$ could still have some isolated vertices, which presents an obstacle to embedding perfect tilings.
A natural way to overcome this is to require linear minimum degree in $G$.

Uniformly dense graphs have been investigated in the context of (powers of) Hamilton cycles and graphs of sublinear bandwidth by Staden and Treglown~\cite{ST20} as well as Ebsen, Maesaka, Reiher, Schacht and Schülke~\cite{EMR+20}.
However, in order to guarantee such \emph{connected} guest graphs one needs to impose further assumptions on the host graphs, which they call `robust inseparability' and which are stronger than linear minimum degree.
Here we show that for perfect tilings, a lower bound on the minimum degree suffices.

Unfortunately, we can not hope to find {perfect} $F$-tilings even for fixed $F \in \cFcr$
because a uniformly dense graph $G$ may be the union of disjoint cliques,
and if the orders of these cliques are not a multiple of $v(F)$, then $G$ trivially has no {perfect} $F$-tiling.
Therefore, as in the case $\kk=2$ in~\eqref{eq:cFcr}, we restrict ourselves to tiles $F$ with $\hcfc(F)=1$ (but arbitrary chromatic number).
On the other hand, it turns out that (unlike above) no requirement on $\hcf(F)$ is necessary
and we no longer require strict bounds on the critical chromatic number.
Formally, let $\cF_c$ be the family of graphs~$F$ with $\hcfc(F)=1$, and let {$\cF_c(k,w)=\cF_c\cap \cF(k,w)$.}

\begin{theorem}\label{thm:quasirandom}
	For all $k,\Delta > 1$ and $d,\mu > 0$, there is $\rho>0$ such that for sufficiently large~$n$ the following holds.
	Let $G$ be a $(\rho,d)$-dense graph on $n$ vertices with $\delta(G) \geq \mu n$.
	Let $H$ be an $\cF_c(k, \sqrt{n}/\log n)$-tiling on $n$ vertices with $\chi(H)\leq k$ and $\Delta(H) \leq \Delta$.
	Then $H \subset G$.
\end{theorem}

\smallskip
All the above results are derived from a more general framework that goes beyond local conditions (such as degrees and uniform density) and is formulated in terms of abstract properties of our guest and host graphs.
This also {allows us to} untangle different conditions that are needed in the proof and address them individually.
In the following,
we discuss these, often necessary, properties for the guest and host graphs separately.

\subsection{Guest graph properties}\label{sec:1.1}
Let us begin by motivating the properties of guest graphs that allow us to overcome certain divisibility obstructions in the host graph.
We provide two concrete examples to highlight the necessity of this soon-to-be-defined structure.
Consider a large complete bipartite graph $G=K_{a,b}$ with $a+b=n$, $6\mid n$ and {$a\approx b$.}
Clearly, if $a,b$ are odd, then $G$ does not contain a perfect $K_{2,4}$-tiling.
However, if~$H$ is the union of $\frac{n}{6}-1$ disjoint copies of $K_{2,4}$ and one $K_{3,3}$ with in total $n$ vertices,
then~$H$ is a subgraph of $G$.
Unfortunately, then~$H$ is no longer a subgraph of $G$ if $a,b$ are even.
Instructively, if instead of a $K_{3,3}$, the graph~$H$ is the union of $\frac{n}{6}-1$ disjoint copies of $K_{2,4}$ and two copies of $K_{1,2}$, then $H$ is a subgraph of~$G$ regardless of the parity of $a,b$.
This is because the union of the two copies of $K_{1,2}$ is what we soon call a proper \fc{} graph.

Similarly, suppose that $G$ is the disjoint union of $K_a$ and $K_b$ with $a+b=n$
and $6\mid n$ as well as {when} $a,b \not\equiv 0 \bmod 6$.
Again, $G$ cannot contain a perfect $K_{2,4}$-tiling.
On the other hand,
if~$H$ is the union of $\frac{n}{6}-1$ disjoint copies of $K_{2,4}$ and six isolated vertices,
then~$H$ is a subgraph of $G$ regardless of whether $a,b\equiv 0 \bmod 6$ or not.
This is because the six isolated vertices are  what we soon call a topological \fc{} graph.

We now formally define \fc{} graphs thereby capturing the correct properties of the examples above.
For a (not necessarily proper) $k$-colouring $\phi\colon V(W)\to [k]$ of a graph $W$,
we denote the tuple of colour class sizes by
$\ord(\phi)=\left(|\phi^{-1}(1)|,|\phi^{-1}(2)|,\dots,|\phi^{-1}(k)|\right).$
We say that~$\phi$ is \emph{proper} if no edges in~$W$ are monochromatic and \emph{topological} if all edges are monochromatic.
Roughly speaking, a graph is \fc{} with a central colouring $\phi$ if all colourings `close' to $\phi$ can be realised by this graph.

\begin{definition}[\Fc{} graphs]\label{def:wildcardS}
	A graph $W$ is \emph{$(k,s)$-\fc{}} with \emph{central} $k$-colouring $\phi\colon V(W) \to [k]$
	if for all $\vecb w\in\INTS^k$ with $\sum_{i \in [k]}\vecb w(i)=0$ and $\maxnorm{\vecb w} \leq s$, there {is a $k$-colouring} $\phi'$ of $W$ with $\ord(\phi)-\ord(\phi')=\vecb w$.
	We say that $W$ is \emph{properly}  {$(k,s)$-\fc{}} if $\phi, \phi'$ are proper and \emph{topologically} if $\phi, \phi'$ are topological.
\end{definition}

One key insight that follows from our proofs is that if $H$ contains enough disjoint \fc{} graphs with the correct parameters, then $\crit(H)$ determines the minimum degree threshold for containing~$H$, as is the case in Theorem~\ref{thm:extended-KO}.

The most flexible graphs with respect to colourings are surely isolated vertices and we exploit this, for instance, in the proof of Theorem~\ref{thm:miss_constant_frac-simple} because we simply add to $H$ isolated vertices until $H$ has $n$ vertices (call this new graph $H'$).
Then $H'$ contains \fc{} graphs and their strength and number depends on the number of isolated vertices added;
moreover, $\crit(H) \approx \crit(H')$.
Therefore, $\crit(H)$ determines the minimum degree threshold for containing~$H$ for almost spanning guest graphs~$H$, because $H'$ always contains \fc{} subgraphs.

For $k\geq 3$, examples of proper $(k,s)$-\fc{} graphs are (odd) cycles of length at least, say, $10ks$ or the disjoint union of at least $10ks$ cycles each of length at least $4$.
More generally, \cref{lem:hcf-gives-flexi-chrom-mixed} states that sufficiently large $\cFcr$-tilings are properly \fc{}.

In fact, whether a collection of graphs from $\cFcr$ forms a \fc{} graph quickly reveals itself to be a problem in additive combinatorics;
it boils down to the question of whether the set of all potential sums of a certain set of integers contains a long interval.
In this paper we adopt a result of Lev~\cite{Lev10} to our setting.
Questions of this type have recently been investigated by Conlon, Fox and Pham in~\cite{CFP21}.

\smallskip

Next we introduce some notation to refer to four graph classes $\cH$, $\cB$, $\cJ$, $\cW$ in order to conveniently state our results throughout the article.
The main graph class is $\cH$, which is the intersection of the other three graph classes.
Informally, $\cH$ consists of graphs that
admit an upper bound on their maximum degree and order of a tile (asserted by~$\cB$),
bounded critical chromatic number (asserted by~$\cJ$)
and contain a set of \fc{} subgraphs (asserted by~$\cW$).
For an $\cF$-tiling $H$ for some graph class $\cF$,
we refer to a \emph{subtiling} of $H$ as an induced subgraph of $H$ that contains each component either completely or not at all.
Observe that a subtiling may contain a tile only partly.
Let
\begin{itemize}  \setlength\itemsep{0.3em}
	\item $\cB(n,w,\Delta)$ be the family of graphs $H$ on at most $n$ vertices with $\Delta(H) \leq \Delta$ where each tile has order at most $w$;
	\item $\cJ(\chi)$ be the family of graphs $H$ with $\crit(H) \leq  \chi$;
	\item $\cW(t,\ell,\chi,w,w')$ be the family of graphs $H$ such that $H$ contains $t$ disjoint proper $(\lceil \chi\rceil,\lceil \chi\rceil w)$-\fc{} subtilings  and one {(disjoint of the former)} topological $(\ell,w)$-\fc{} subtiling each of order at most $w'$; and
	\item $\cH(n,w,\Delta; \chi;t,\ell,w') = \cB(n,w,\Delta) \cap \cJ(\chi) \cap \cW(t,\ell,\chi,w,w').$
\end{itemize}

\subsection{Host graph properties}\label{sec:host-graphs}

We now consider properties of the host graph.
By moving beyond degree conditions we can unravel  the consequences of the degree conditions that were implicitly used in earlier results and clearly articulate how these consequences interact with the properties of the guest graphs.

To this end, we introduce the concept of \emph{tiling frameworks}, which describe a set of necessary conditions for a host graph $G$ to contain all elements of {$\cH(n,w, \Delta; \chi;t,\ell,w')$.}
Our main result states that if~$G$ satisfies these properties in a robust way, then it contains the entirety of $\cH(n,w, \Delta; \chi;t,\ell,w')$.
The encoding of the aforementioned picture with this framework leverages ideas from the the third author and Sanhueza-Matamala~\cite{LS23}.
A related approach has been developed by Keevash and Mycroft~\cite{KM15} regarding perfect matchings in hypergraphs.

Let us start with a first necessary condition for a graph $G$ for containing a perfect tiling;
it clearly must contain an (almost) perfect fractional tiling of the same type.
We refer to this condition (and, at times, obstacle) as a \emph{space} constraint.
Let us introduce this notion in more detail.
We denote by $\Hom(F;G)$ the set of \emph{homomorphisms} from $F$ to~$G$, that is, the set of functions $\phi\colon  V(F) \to V(G)$, where $uv\in E(F)$ implies $\phi(u)\phi(v)\in E(G)$ for all $u,v\in V(F)$.
A \emph{fractional $F$-tiling} of $G$ is a function $\omega \colon \Hom(F;G)\rightarrow [0,1]$ such that for all $v\in V(G)$, we have $\sum_{\theta\in\Hom(F;G)} \omega(\theta)|\theta^{-1}(v)|\leq 1$.
The \emph{weight} of the tiling is $v(F)\sum_{\theta\in\Hom(F;G)}\omega(\theta) $.
A fractional tiling of $G$ is \emph{$(1-\rho)$-perfect} (or \emph{perfect} if $\rho=0$) if its weight is at least $(1-\rho)v(G)$.
Note that $\{0,1\}$-valued fractional tilings correspond to tilings.
Finally, consider a rational $\chi>1$.
Then there are unique coprime integers $a<b$ with $\chi = \lfloor \chi \rfloor + a/b$ (if $\chi\in \NATS$, set $a=0$, $b=1$).
{We write $B_\chi$ for the complete $\kk$-partite graph $B_{a,(\kk-1)\ast b}$, noting that $\crit(B_\chi) = \chi$.}

One can show that in order to find almost perfect tilings, it suffices that the host graph has a `robust' perfect fractional tiling (\cref{thm:Komlos-for-frameworks}) and is large.
However, for a perfect tiling we must ask for more properties of the host graph.

Another trivial necessary condition for a graph $G$ to contain a perfect $F$-tiling is that every vertex lies in a copy of $F$.
Suppose that $F$ contains the $k$-partite graph $K_{2,\ldots,2}$ as a subgraph.
Then it is certainly necessary that for every vertex $v\in V(G)$, there is a $k$-clique $K$ disjoint from~$v$ such that~$v$ has at least $k-1$ neighbours into $K$.
This example shows that for $G$ to contain a perfect $F$-tiling, in some cases it is necessary that every vertex is `linked' to a $k$-clique.

We define this formally in terms of hypergraphs.
A \emph{$k$-graph} $J$ is a \emph{$k$-uniform} hypergraph, meaning that every edge contains $k$ vertices.
We say that a vertex $v \in V(J)$ and an edge $f \in E(J)$ with $v \notin f$ are \emph{linked} in $J$ if there is an edge $e \in E(J)$ with $v \in e$ and  $|e \cap f| = k-1$.
Moreover, $J$ is \emph{linked} if every vertex is linked to an edge.

In order to model graph structures with hypergraphs,
we define the \emph{$k$-clique hypergraph}~$K_k(G)$ of a graph $G$ to be the $k$-graph with vertex set $V(G)$
where $\{v_1,\ldots,v_k\}$ is an edge if $\{v_1,\ldots,v_k\}$ induces a $k$-clique in~$G$.
Properties of~$K_k(G)$ encode {the} potential divisibility issues when one aims to embed a $k$-chromatic graph $H$ into $G$.
These properties are captured by the tight and loose component structure of the host graph, which we introduce now.

For a $k$-graph $J$, we define by $T(J)$ the \emph{dual graph} whose vertices are the edges of $J$, where $e,f \in E(J)$ are adjacent whenever $|e\cap f| = k-1$.
The \emph{tight components} of~$J$ are the subgraphs $T \subset J$, whose edges form the components of $T(J)$.
For loose connectivity, let $L(J)$ be the graph whose vertices are the tight components of $J$ and two tight components $T,T'$ form an edge whenever their vertex sets intersect.
The \emph{loose components} of $J$  are unions $L= T_1 \cup \dots \cup T_q$ of tight components $T_1, \dots, T_q$ which form a component in $L(J)$.
Given this, we say that $J$ is $(t,\ell)$-\emph{connected} if there is a (vertex) spanning subgraph $J' \subset J$ such that $J'$ has at most~$t$ tight components and at most~$\ell$ loose components.

We remark that each individual tight component and the combined loose components of~$K_k(G)$ can potentially form a divisibility obstruction to embedding a tiling $H$.
Each proper and topological \fc{} subtiling of $H$ helps us to overcome one of these obstacles.
(This we illustrated for $k=2$ by some examples in Section~\ref{sec:1.1}.)
A more involved construction shows that graphs satisfying the degree conditions of \cref{thm:miss_constant_frac-simple} can contain many tight components, {where each tight component comes with its own divisibility obstruction (see \cref{sec:constructions}).}

The following definition packages the structural properties we have just introduced into what we call a tiling framework.
\begin{definition}[Tiling framework]\label{def:tiling-framework}
	Let {$\chi \in \mathbb{Q}$ with $\chi > 1$,} $\rho\in [0,1]$ and $t,\ell \in \NATS$.
	A graph $G$ is a \emph{$(\chi,\rho;t,\ell)$-tiling framework} if
	\begin{enumerate}[\upshape (F1)]
		\item  \label{item:framework-Geometry} $G$ has a $(1-\rho)$-perfect fractional {$B_\chi$}-tiling, \hfill (space)
		\item \label{item:framework-divisibility} $K_{\lceil \chi \rceil}(G)$ is $(t,\ell)$-connected and
		\hfill (divisibility)
		\item \label{item:framework-linkage} $K_{\lceil \chi \rceil}(G)$ is linked. \hfill (linkage)
	\end{enumerate}
\end{definition}

We remark here that, in particular, the host graphs in {\cref{thm:extended-KO,thm:miss_constant_frac-simple,thm:extended-KO-deg-seq,thm:quasirandom}} are tiling frameworks for an appropriate choice of the parameters (see Lemmas~\ref{lem:deg-gives-framework} and~\ref{lem:deg-seq-mixed-framework}).

\subsection{Main result}\label{sec:main-result}
To state our main result, we need one final piece of notation.
For $\eps,d\in [0,1]$, a subgraph $G'$ of a graph $G$ is an \emph{$(\eps,d)$-approximation of $G$} if $\deg_{G'}(v) \geq \deg_{G}(v) - d v(G)$ for all $v \in V(G')$ and $v(G') \geq (1 - \eps)v(G)$.
For $\mu>0$, we say that $G$ $\mu$-\emph{robustly} admits a graph property~$\cP$
if every $(\mu,\mu)$-approximation $G' \subseteq G$ also admits~$\cP$.

Now we are ready to state our main result.
Informally, it states that if the host graph~$G$ is a robust tiling framework, then $G$ contains all graphs of $\cH$ (quantified appropriately).
Note that the number of \fc{} subtilings required of $H$ matches the number of tight and loose components of the host (clique) graph.

\begin{theorem}[Tiling framework theorem]\label{thm:frameworks}
	For all $\chi, \Delta>1$, $\mu>0$ with $\chi, \mu \in \mathbb{Q}$,
	there exist $\rho,\xi>0$ such {that} the following holds for all $w,w'$ and sufficiently large $n$
	with $w\leq \xi n$ and $w'\leq \sqrt\xi n$.
	Suppose $G$ is a $\mu$-robust $({\chi+\mu},\rho;t,\ell)$-tiling framework on $n$ vertices and
	$H \in \cH(n,w,\Delta; \chi; t,\ell,w')$ {for some $t,\ell\geq 1$.}
	Then $H\subseteq G$.
\end{theorem}

All our previously stated results follow fairly easily from \cref{thm:frameworks} (see Section~\ref{sec:2}).
We also  remark that under the assumptions of \cref{thm:frameworks},
that is, assuming $G$ is a $\mu$-robust $(\chi+\mu,\rho;t,\ell)$-tiling framework on $n$ vertices,
there are $t'(\mu),\ell'(\mu)$ such that $G$ is also a $\mu$-robust $(\chi+\mu,\rho;t'(\mu),\ell'(\mu))$-tiling framework
(see \cref{pro:backstop}).

\subsection{Organisation of the paper}
The paper is organised as follows.
{In the next section, we provide some more details on how different relevant objects are related.
	In \cref{sec:outline}, we give an overview of our proof.}
In \cref{sec:notation}, we introduce some basic notation.
We derive the result{s} for the host and guest graphs in \cref{sec:host-graph-properties,sec:guest-graph-properties}, respectively.
In \cref{sec:tools}, we introduce a few tools such as the Regularity Lemma and the Blow-Up Lemma that will help us to embed large graphs.
In \cref{sec:intermediate}, we show a few technical observations that will prepare us for the upcoming proofs.
In \cref{sec:komlos-framework,sec:mixed-komlos-framework}, we prove (mixed) versions of Koml\'os' theorem for frameworks.
Finally, \cref{sec:proof-framework-theorem} is dedicated to the proof of \cref{thm:frameworks}.
Further arguments used in the proof of \cref{thm:frameworks} are found in \cref{sec:lemma-for-Exceptional-Vertices}--\ref{sec:lemma-for-H}.
Finally, we conclude in \cref{sec:conclusion} with some remarks.

\section{More details on the bigger picture}\label{sec:2}

In this section we deduce Theorems~\ref{thm:extended-KO}--\ref{thm:quasirandom} assuming some observations about tiling frameworks and \fc{} graphs.
This may also help the reader to better capture the concepts introduced in the introduction.

In order to simplify the statements of our results we use the following notation.
We write \emph{$\alpha\ll\beta$} to mean that there is a non-decreasing function~$\alpha_0\colon(0,1]\rightarrow(0,1]$ such that for any $\beta\in(0,1]$, the subsequent statement holds for~$\alpha\in(0,\alpha_0(\beta)]$.
Hierarchies with more constants are defined similarly and should be read from right to left.
Constants in hierarchies are always rational numbers in $(0,1]$.
Moreover, if we write $1/a$ in such a case {then we implicitly mean} that $a\in \NATS$ with the only exception of the letter $\chi$ which we always assume to be a rational number larger than $1$.

We start with a closer look at guest graphs and now substantiate the claim that isolated vertices make \fc{} graphs.
For any $k,w\in \NATS$, the collection of $kw$ isolated vertices are both topologically and properly $(k,w)$-\fc{}, indeed one can take the central colouring to be the balanced $k$-colouring. This has the following useful consequence.
\begin{observation}\label{obs:isolated-vertices-flexible}
	Let $n,w,t,\ell,\Delta \in \NATS$ and $\chi>1$ with $n>(t+\ell)k^2w$ where $k:=\kk$.
	Suppose~$H \in \cB(n-(t+\ell)k^2w,w,\Delta) \cap \cJ(\chi)$. Let $H'$ be obtained from $H$ by adding $(t+\ell)k^2w$ isolated vertices.
	Then $H' \in \cH(n,w,\Delta; \chi;t,\ell,k^2w)$.
\end{observation}

The next lemma extends the fact that long cycles are \fc{} graphs when dealing with at least three colours.
It states that if $\crit(H)\approx \chi(H)-1$, where this approximation has to be the better the larger the maximum degree of $H$ is,
then $H$ is \fc.
{(Note that the property of being topologically {$(\ell,\cdot)$-\fc{}}-chromatic is trivial when $\ell =1$.)}
We defer its proof to \cref{sec:guest-graph-properties}.

\begin{lemma}\label{lem:miss_constant_frac}
	Let $1/n \ll \xi \ll \mu,1/\chi,1/\Delta,1/t$ with $\chi + \mu \leq \frac{\Delta+1}{\Delta}(k -1)$ where $k := \kk$.
	Let $\ell=t$ if $k=2$ and $\ell=1$ if $k \geq 3$.
	Suppose $H\in \cB(n,w,\Delta) \cap  \cJ(\chi)$ where $w \leq \xi n$.
	Then  $H \in \cH(n,w,\Delta;\chi;t,\ell,\sqrt \xi n)$.
\end{lemma}

Recall that the K\"uhn--Osthus Tiling Theorem states that for all (fixed) graphs $F\in \cFcr$,
the minimum degree threshold for containing a perfect $F$-tiling is determined by the critical chromatic number.
The following lemma, also proved in \cref{sec:guest-graph-properties}, states that in fact any $\cFcr$-tiling admits many \fc{} subtilings provided the components are not too large.
Moreover, for $F$-tilings with fixed $F \in \cFcr$ we obtain slightly better quantifications.

\begin{lemma}\label{lem:hcf-gives-flexi-chrom-mixed}
	Let $n,w,w',t,\Delta\in \NATS$, $\chi > 1$ and $k = \kk$.
	Let $\ell=t$ if $k=2$ and $\ell=1$ if $k \geq 3$.
	Let $H \in \cB(n,w,\Delta) \cap  \cJ(\chi)$ and suppose further that one of the following holds
	\begin{enumerate}[\upshape (a)]
		\item \label{itm:Fcr_for_mixed}$H$ is an {$\cFcr(k,w)$-tiling} and $w'\geq (k+\ell)^{10} w^2\log w$,
		\item \label{itm:Fcr_for_non-mixed} $H$ is an $F$-tiling for some $F\in \cFcr(k,w)$ and $w'\geq (k+\ell)^{10} w^2$.
	\end{enumerate}
	Then $H \in \cH(n,w,\Delta; \chi;t,\ell,w')$ provided that $v(H)\geq(t+\ell)w'$.
\end{lemma}

We show a similar lemma for  $\cF_c$-tilings.

\begin{lemma}\label{prop:quasirandom-wildcards}
	Let $1/n \ll \xi \ll 1/\ell, 1/k,1/\Delta$.
	Let $H$ be an $\cF_c(k, \sqrt{n}/\log n)$-tiling on $n$ vertices with $\chi(H)\leq k$ and $\Delta(H) \leq \Delta$.
	Then $H \in \cH(n, \xi n,\Delta;k+1;\ell,\ell,\sqrt \xi n)$.
\end{lemma}

Let us now turn to host graphs.
As we indicated already above, if the host graph has a sufficiently large minimum degree or uniform density, then it must be  a (robust) tiling framework.
The same is true for lower bounds on the degree sequence.
The following two lemmas formalise this claim.
Theorem \cref{thm:frameworks} in conjunction with the above this allows us to derive \cref{thm:extended-KO,thm:extended-KO-deg-seq,thm:quasirandom}  from our main theorem.

\begin{lemma}\label{lem:deg-gives-framework}
	Let $1/n  \ll 1/t\ll \mu,1/\chi$ with $\chi+\mu/4 \leq k$ where $k := \kk$.
	Let $\ell=t$ if $k=2$ and $\ell=1$ if $k\geq 3$.
	Suppose $G$ is a graph on $n$ vertices with $\delta(G) \geq (1-1/\chi  + \mu)n$.
	Then $G$ is a $\mu/4$-robust $(\chi+\mu/4,{0};t,\ell)$-tiling framework.
\end{lemma}

\begin{lemma}\label{lem:deg-seq-mixed-framework}
	Let $1/n \ll \rho \ll 1/t \ll \mu,1/\chi$ with $\chi+\mu/4 \leq k$ where $k := \kk$.
	Let $\ell=t$ if $k=2$ and $\ell=1$ if $k\geq 3$.
	Suppose~$G$ is a graph on $n$ vertices with $(\chi,\mu)$-strong degree sequence.
	Then $G$ is a $\mu/4$-robust $(\chi+\mu/4,\rho;t,\ell)$-tiling framework.
\end{lemma}

\begin{lemma}\label{pro:quasirandom-framework}
	Let $1/n \ll \rho, 1/\ell \ll \mu' \ll \mu,d,1/k$.
	Suppose $G$ is a $(\rho, d)$-dense graph on $n$ vertices with $\delta(G) \geq \mu n$.
	Then $G$ is a $\mu'$-robust $(k,0;\ell,\ell)$-tiling framework.
\end{lemma}

\begin{proof}[Proof of \cref{thm:extended-KO}]
	Suppose we are given $\chi',\Delta,\mu'$ and $H$ as input.
	Let $G$ be a graph on $n$ vertices with $\delta(G) \geq (1-1/\chi' + \mu')n$.
	To prove that $\MT(n,H) \leq (1-1/\chi' + \mu')n$, it suffices to show that $H \subset G$.

	In order to apply \cref{lem:deg-gives-framework}, we need to calibrate our parameters slightly.
	We define $\chi$ and~$\mu$ as follows:
	If $\lceil \chi' \rceil - \mu'/16 \leq \chi'\leq \lceil \chi' \rceil$, set $\chi = \chi' + \mu'/4$ and $\mu = \mu'/4$.
	Otherwise, set $\chi = \chi'$ and $\mu = \mu'/4$.
	Note that in both cases, we have $\chi + \mu/4 \leq \lceil \chi \rceil$.
	Moreover, $H$ is still a $\cFcr(\lceil \chi \rceil, \sqrt{n}/\log n)$-tiling with $\crit(H)\leq \chi$, and we have $\delta(G) \geq (1-1/\chi + \mu)n$.
	For the rest of the proof, we use $\chi,\mu$ instead of $\chi',\mu'$.

	Let $k=\kk$.
	Clearly, $G$ satisfies the assumptions of \cref{lem:deg-gives-framework}.
	It follows that $G$ is a $\mu/4$-robust $(\chi+\mu/4,\rho;t,\ell)$-tiling framework with $t = k^{4k}/(\chi + 1 -k +\mu/2)^k$
	and $\ell=t$ if $k=2$ and $\ell=1$ if $k\geq 3$ as well as sufficiently small $\rho>0$.

	Concerning the guest graph, consider a partition of $H$ into collections of tiles such that each collection contains between $\xi n/2$ and $\xi n$ vertices
	for sufficiently small $\xi>0$.
	One of these collections $H'$ satisfies $\crit(H') \leq \chi$.
	Let $\xi'\in [\xi/2,\xi]$ be such that $\xi' n = v(H')$.
	Then $H' \in \cB(\xi' n,w,\Delta) \cap  \cJ(\chi)$ where $w := \sqrt{n} / \log n$.
	Let $w' = v(H')/(t+\ell) = \xi' n / (t+\ell)$, and note that $w' \geq (\lceil \chi \rceil + \ell)^{10} n / \log n \geq (\lceil \chi \rceil + \ell)^{10} w^2 \log w$ as $n$ is sufficiently large.
	It follows by \cref{lem:hcf-gives-flexi-chrom-mixed}\ref{itm:Fcr_for_mixed} that $H' \in \cH(\xi' n,w,\Delta; \chi; t,\ell, w')$,
	which in turn implies that $H \in \cH(n,w,\Delta; \chi; t,\ell,w')$.

	To finish, note that $w \leq \xi'^2 n$ and $w' \leq \xi' n$.
	So we can apply \cref{thm:frameworks} with $\mu/4,\xi'^2$ playing the role of~$\mu,\xi$ to conclude that $H \subset G$.
\end{proof}
We remark that, one can use \cref{lem:hcf-gives-flexi-chrom-mixed}\ref{itm:Fcr_for_non-mixed} in the same way to show a version of \cref{thm:extended-KO} for $F$-tilings with fixed $F \in \cFcr(\kk,w)$ of order $w \leq \xi'\sqrt{n}$.
The {proof} of \cref{thm:extended-KO-deg-seq} follows analogously.
\cref{thm:miss_constant_frac-simple} is proved in almost the same way, by adding $C$ isolated vertices to $H$ and then using \cref{obs:isolated-vertices-flexible}.
We provide the details for {the} sake of completeness.

\begin{proof}[Proof of \cref{thm:miss_constant_frac-simple}]
	Given $\chi,\mu,\Delta$, we can assume that $\chi + \mu/4 \leq \kk$ without loss of generality.
	(Otherwise use a preprocessing argument as in the proof of \cref{thm:extended-KO}.)

	Set $k=\kk$, $C= (\Delta k)^{8k}$, $t = k^{4k}/(\chi+1 -k +\mu/2)^k$ and $\xi=\xi(\lceil\chi\rceil,\Delta,\mu)>0$ sufficiently small.
	Let $\ell=t$ if $k=2$ and $\ell=1$ if $k\geq 3$.
	Suppose that $G$ is a graph on $n$ vertices with $\delta(G) \geq (1-1/\chi + \mu)n$.
	Consider $H$ with $w\leq \xi n$ as in the statement of the theorem.
	We have to show that $H \subset G$.

	To this end, we apply \cref{lem:deg-gives-framework} to conclude that $G$ is a $\mu/4$-robust $(\chi+\mu/4,\rho;t,\ell)$-tiling framework with $t = k^{4k}/(\chi+1 -k +\mu/2)^k$.
	We then distinguish between two cases.
	First, suppose that $\chi +\mu/2 \geq  \frac{\Delta+1}{\Delta}(k-1)$.
	We bound
	\begin{align*}
		t \leq \frac{k^{4k}}{(\chi-k+1+ \mu/2 )^k}
		\leq \frac{k^{4k}}{(k-1)^k} \Delta^k
		\leq (\Delta k)^{4k}.
	\end{align*}
	By adding $Cw \geq (t+\ell)k^2w$ isolated vertices to $H$ and using \cref{obs:isolated-vertices-flexible}, we obtain a graph $H' \in \cH(n,\xi n,\Delta;\chi;t,\ell,k^2w)$.
	Since $k^2 w \leq  k^2 \xi n\leq  \sqrt{\xi} n$, we can apply \cref{thm:frameworks} to obtain $H' \subset G$ and so $H\subset G$.

	Now suppose that $\chi + \mu/2 \leq \frac{\Delta+1}{\Delta}(k-1)$.
	Then by \cref{lem:miss_constant_frac} applied with $n-Cw,2\xi, \mu/2$ playing the role of $n,\xi, \mu$, we see that  $H \in \cH(n-Cw,2\xi n,\Delta;\chi;t,\ell,\sqrt{2\xi} n)$.
	Thus we can apply \cref{thm:frameworks} again to obtain $H' \subset G$ and so $H\subset G$.
\end{proof}

Finally, we show our result on uniformly dense graphs.
\begin{proof}[Proof of \cref{thm:quasirandom}]
	We choose $\ell,\rho, \mu'$ such that $1/n\ll \rho \ll 1/\ell \ll \mu'\ll d, \mu,1/k,1/\Delta$.
	By \cref{pro:quasirandom-framework} applied with $k+2$ playing the role of $k$,
	the graph $G$ is a $\mu'$-robust $(k + 1 + \mu',0;\ell,\ell)$-tiling framework.
	(We remark that due to the quantification of \cref{pro:quasirandom-framework}, we could even ask for a $(k+2,0;\ell,\ell)$-tiling framework.)
	Moreover, $H \in \cH(n,\xi n,\Delta;k+1;\ell,\ell,\sqrt \xi n)$ by \cref{prop:quasirandom-wildcards} {for sufficiently small $\xi$.}
	Combining this with \cref{thm:frameworks} gives the desired result.
\end{proof}

\section{Proof outline}\label{sec:outline}
Here we sketch the main ideas of our approach.
Suppose the assumptions made in \cref{thm:frameworks};
in particular, suppose $G$ is a $\mu$-robust $(\chi+\mu,\rho;t,\ell)$-tiling framework on $n$ vertices and $H$ has $\ell$ loose and $t$ tight \fc{} subtilings.
Our approach simultaneously works at the two ends of the problem;
on the one hand, we elaborate on the structure of $G$ and on the other hand, we prepare $H$ for an embedding into $G$.

For the host graph $G$,
we apply the Regularity Lemma to $G$ and obtain a vertex partition $V_0,V_1,\ldots, V_r$ of $G$ and a reduced graph $R$ with vertex set $[r]$ that captures the global structure of $G$ (see \fr{lem:lemma-for-G-true-form}).
Importantly, it can be shown that $R$ inherits the property of being a $(\chi+\mu,\rho;t,\ell)$-tiling framework.
We then eliminate the `exceptional vertices' $V_0$ of $G$ by suitably embedding a few components of $H$ to cover all of them.
At this point, we require the linkage property of frameworks.

Let $G'$ and $H'$ be the remaining graphs.
Let $R^\ast$ be the blow-up of $R$, where each vertex $i$ of $R$ has been replaced with the vertex set $V_i$ (adjusted to $G'$) and the edges with complete bipartite graphs.
(We can assume that $G' \subset R^\ast$.)
The well-known Blow-Up Lemma states that in order to show that $H' \subset G'$, it suffices to show that $H' \subset R^\ast$.

We show that $H' \subset R^\ast$ in three steps (see \fr{lem:lemma-for-H});
that is, as all vertices in a single cluster of $R^\ast$ have the same neighbourhood, we only need to decide for each vertex of~$H'$ to which cluster it is assigned.
The first step equips us with a certain flexibility used in a later stage.
For each $k$-clique in $R$, we embed a few components of $H$ in the corresponding $k$-partite subgraph of~$R^\ast$.
This has the following advantage.
For any two hyperedges $e,f\in K_k(R)$ with $|e\cap f|=k-1$, $i\in e\sm f$ and $j\in f\sm e$, we can easily modify our embedding by moving a few vertices from $V_i$ to $V_j$.
Roughly speaking, the same modifications can be done between clusters $V_i,V_j$ that are further apart in $K_k(R)$ as long as $i$ and $j$ are covered by the edges of the same tight component of $K_k(R)$.

In the second step, we utilise a framework version of Koml\'os' theorem (\cref{thm:mixed-Komlos-for-frameworks}) to obtain an almost embedding of $H'$ into $G'$ (ignoring the reduced graph) and thus also into $R^\ast$.
This is somewhat counter-intuitive, since we aim to utilise the Blow-Up Lemma which will embed $H'$ back into $G'$. But it efficiently deals with the issue that the clusters $V_i$ might at this point be quite unbalanced.
Next we ignore that we actually found an embedding and only memorise its induced allocation of almost all the vertices in $V(H')$ to the clusters of~$R^\ast$.
Note that we did not use any property of $R$ here, but that~$G'$ is still a robust tiling framework (due to robustness of $G$) and the space property of the framework.

In the last step, we allocate  the remaining few vertices of $G'$.
Here we work again on the level of $R$.
As we will see, every tight and loose component of $K_k(R)$ constitutes an obstacle for embedding the rest of~$G'$.
However, as $R$ is a $(\chi+\mu,\rho;t,\ell)$-tiling framework, its divisibility property ensures that {the} number of these components matches the number of \fc{} subtilings contained in $H$.
Using these subtilings and exploiting the flexibility within tight components created by our first step, we can then finish our embedding of $H'$ into $R^\ast$ as desired.
The details of the last two steps can be found in \cref{sec:proof-framework-theorem}.

Some of our ideas have been inspired by the recent work of the third author and Sanhueza-Matamala~\cite{LS23}.

\section{Notation}\label{sec:notation}

We introduce and recapitulate some standard notation.
Given $n \in \NATS$, we write $[n]=\{1,\dots,n\}$.
For $x,y,z \in \REALS$, we write $x= y \pm z$ to mean $y-z \leq x \leq y+z$.
Logarithms are natural unless indicated otherwise.
Vectors will be represented by bold, lower-case Latin letters, such as $\vecb w \in \REALS^n$, while the elements of $\vecb w$ will be denoted $\vecb w(i)\in \REALS$ for $i \in [n]$.
{We treat vector always as column vectors.}
It will often be convenient to speak of vectors $\vecb w \in \INTS^n$, or $\vecb u \in \NATSZ^n$, and we will do so, bearing in mind that the vector structure comes from the ambient vector space $\REALS^n$.
For a set $S$, we write $\vecb w \in \REALS^S$ for the vector {indexed} by the elements of $S$.
For a set $T \subset S$ and a vector $\vecb w \in \REALS^{S}$, we obtain the restricted vector $\restr{\vecb w}{T}$ by setting all entries that do not correspond to an element of $T$ to be zero.
Further we define $\onenorm{\vecb w}=\sum_{i\in [n]}|\vecb w(i)|$ and  $\maxnorm{\vecb w}=\max_{i\in[n]}|\vecb w(i)|$ as usual.
For $c \in \REALS$, we write $\vecb w \geq c$ whenever $\vecb w$ is pointwise at least $c$.

Given a graph $G$ and a subset $U\subset V(G)$, we define $G[U]$ to be the subgraph of $G$ induced by the vertices in $U$ and $G-U$ to be the subgraph of $G$ induced by the vertices in $V(G)\sm U$.
For a collection $\cV'=\{V_i'\}_{i \in [r]}$ of subsets of $V(G)$, we abbreviate $G[\cV']=G[\bigcup_{i \in [r]}V_i']$.
The number of vertices and edges in a graph are denoted $v(G)$ and $e(G)$, respectively.
The \emph{$m$-blow-up} $R^\ast$ of a graph $R$, is obtained by replacing each vertex of $R$ with a set $X_i$ of order $m$ and the edges are replaced by complete bipartite graphs.
For a graph $F$,  we denote by $\alpha(F)$ the minimum proportion of vertices in a single colour class over all proper $\chi(F)$-colourings of $F$.
Finally, we denote by $\gcd(S)$ the greatest common divisor of a set of integers $S \subset \INTS$.

\section{Host graph properties}\label{sec:host-graph-properties}
The goal of this section is to prove \cref{lem:deg-gives-framework,lem:deg-seq-mixed-framework,pro:quasirandom-framework}.
We begin with a few general observations on how minimum degree conditions affect the number of tight and loose components.
Given a $k$-graph $H$, we define the \emph{shadow graph} $\partial H$ of $H$ to be the $(k-1)$-graph with vertex set $V(H)$ and edge set consisting of all sets of $k-1$ vertices in $V(H)$ that are contained in some edge of $H$.
In our argumentation, we use the following consequence of the Kruskal-Katona
Theorem to bound from below the number of edges in $\partial H$ (see \cite[Theorem~2.14]{FranklTokushige2018}).

\begin{theorem}\label{thm:KK}
	For $x \geq 1, k \geq 2$, if $H$ is a $k$-graph with $e(H) \geq \binom{x}{k}$, then $e(\partial H) \geq \binom{x}{k-1}$.
\end{theorem}

The following result will also be useful.

\begin{lemma}\label{lem:co-degree}
	For $k \geq 2$, let  $G$ be a graph on $n\geq k	$ vertices with $\delta(G)= \delta n >
		\frac{k-2}{k-1}n$.
	Let $\delta' = \delta - \frac{k-2}{k-1}$.
	Then every $(k-1)$-set $S\subset V(G)$ satisfies $|\bigcap_{v\in
			S}N_G(v)|>(k-1)\delta'n$.
\end{lemma}
\begin{proof}
	Let $S$ be as in the statement and write $T=\bigcap_{v\in S}N_G(v)$
	and $t=|T|$.
	Every vertex in $V(G)\sm (S\cup T)$ has at most $k-2$ neighbours
	in $S$.
	So double-counting the edges between $S$ and $V(G)\sm
		S$ gives
	\begin{align*}
		t(k-1)  + (k-2)(n-k+1-t) \geq \delta n (k-1) - \binom{|S|}{2}.
	\end{align*}
	Rearranging and using that $|S| = k-1$ gives
	\begin{align*}
		t\geq \delta n (k-1)- (k-2)(n-k+1) - \binom{k-1}{2} >
		(k-1)n\left(\delta - \frac{k-2}{k-1}\right),
	\end{align*}
	which completes the proof.
\end{proof}

\subsection{Minimum Degree Conditions}
Here we show \cref{lem:deg-gives-framework}.
To this end, we need the fractional analogue of Koml\'os' theorem, which is implicit in the work of Piguet and Saumell~\cite{PS19}.
See also~\cite[Lemma 6.6]{Lan23} for a full proof.

\begin{lemma}\label{lem:degree_frac_bottle_tiling}
	Let $\chi\geq 1$ and $F$ be a graph with $\crit(F)\leq \chi$.
	Let $G$ be a graph on $n$ vertices with $\delta(G) \geq  (1- {1}/{\chi})n$.
	Then $G$ has a perfect fractional $F$-tiling.
\end{lemma}

The following result shows that under certain minimum degree conditions, the tight components are substantially large.

\begin{proposition}\label{prop:decomposition}
	For $k \geq 2$ and $\delta >\frac{k-2}{k-1}$, let $G$ be a
	graph on $n$ vertices with $\delta(G) \geq \delta n$ and define $\delta'
		= \delta-\frac{k-2}{k-1}$.
	Then for any tight component $T$ of $K_k(G)$, we have $e(T)> {\binom{\delta'n}{k} }$.
	\begin{proof}
		Observe that by \cref{lem:co-degree} each $(k-1)$-set $S\subset
			V(G)$ that induces a clique in $G$
		is contained in at least $(k-1)\delta'n$ edges in $K_k(G)$.
		Let $T$ be a tight component of $K_k(G)$ and $x \geq 0$ be such that $\binom{xn}{k}=e(T)$.
		Recall that if two edges $e,e' \in E(T)$ intersect  in $k-1$ vertices, then they are in the same tight component.
		Therefore, each $(k-1)$-edge $f\in E(\partial T)$ is contained in at
		least $\delta'n$ edges in $E(T)$, and each $k$-edge of $T$ contains $k$
		distinct $(k-1)$-edges of $\partial T$. Hence $e(T)\geq e(\partial
			T)\delta'n/k$.
		By combining these facts and applying Lemma~\ref{thm:KK}, we obtain
		\[
			\frac{\delta'n}{k} \binom{xn}{k-1} \leq \frac{\delta'n}{k} e(\partial
			T) \leq e(T) = \binom{xn}{k}.
		\]
		From this we conclude that $\delta'n \leq xn-k+1$ and hence $x > \delta'$.
	\end{proof}
\end{proposition}

Since the tight components of a graph partition its edge set, we immediately obtain the following consequence of \cref{prop:decomposition}.

\begin{corollary}\label{lem:reachability-components}
	For $k\geq 2$ and $\delta >\frac{k-2}{k-1}$,
	let $G$ be a graph on $n$ vertices with $\delta(G) \geq \delta n$ and
	define $\delta' = \delta-\frac{k-2}{k-1}$.
	Suppose $n$ is large in terms of $\delta'$.
	Then there are at most $2\delta'^{-k}$  tight components in $K_k(G)$.
\end{corollary}\COMMENT{\begin{proof}
		We can partition the edge set of $K_k(G)$ into tight components
		$T_1,\dots,T_m$.
		As $K_k(G)$ has at most ${n\choose k}$ edges and by
		\cref{prop:decomposition} each $T_j$ contains at least ${\delta'n\choose
					k}$ edges, we obtain  $m\leq (1/{\delta'})^k/2$.
	\end{proof}}

Finally, a simple observation that yields a sufficient condition for being linked.
\begin{observation}\label{obs:min_degree_second_neighbours}
	For $k \geq 2$, let $G$ be a graph on $n$ vertices with $\delta(G)= \delta n >
		\frac{k-2}{k-1}n+1$.
	Then every vertex of $G$ is linked in $K_k(G)$.
\end{observation}
\begin{proof}
	Fix an arbitrary vertex $v_1\in V(G)$.
	For $i\in [k-1]$, choose iteratively vertices $v_{i+1}\in \bigcap_{j\in
			[i]}N(v_j)$, which {exist} by \cref{lem:co-degree}.
	This yields a $k$-clique with vertices $v_1,\ldots,v_k$.
	Now we apply \cref{lem:co-degree} and obtain that $|\bigcap_{j\in
			\{2,\ldots,k\}}N(v_j)| {> k-1\geq 1}$.
	Thus $v_1$ is linked in $K_k(G)$.
\end{proof}
We are now ready to prove that minimum degree conditions imply that a
graph is a tiling framework.

\begin{proof}[Proof of \cref{lem:deg-gives-framework}]
	Set $t = k^{4k}/(\chi+1-k + \mu/2)^k$.
	To show that $G$ is a $\mu/4$-robust $(\chi,\rho;t,\ell)$-tiling framework, we verify the properties detailed in \cref{def:tiling-framework} for an arbitrary $(\mu/4,\mu/4)$-approximation $G' \subset G$.
	Let $\chi'=\chi+\mu/4$, $\mu'=\mu/4$ and $n' = v(G')$.
	It is {straightforward} to check that $\delta({G'}) \geq (1-1/\chi'  + \mu')n'$.
	Furthermore, we have  $\lceil \chi' \rceil \leq k$ by assumption.

	Now for the properties of \cref{def:tiling-framework}.
	To begin, note that property~\ref{item:framework-Geometry} follows from \cref{lem:degree_frac_bottle_tiling} applied to
	$G'$ with $B_{\chi'}$ in place of $F$.

	Now we establish property~\ref{item:framework-divisibility} by
	considering tight and loose components in turn.
	A quick computation gives
	\urldef\myurlx\url{https://www.wolframalpha.com/input/?i=%28x-1%29%2Fx+-+%28k-2%29%2F%28k-1%29}
	\begin{align*}
		\frac{\delta(G')}{n'} - \frac{k-2}{k-1} & \geq 1-\frac{1}{\chi'}- \frac{k-2}{k-1} + \mu'
		= \frac{\chi'+1-k}{(k-1)\chi'} + \mu' \geq \frac{1}{k^2} (\chi'+1-k + \mu').
	\end{align*}\COMMENT{\myurlx}
	Thus by \cref{lem:reachability-components}, there are at most $2(k^2/(\chi'+1-k + \mu'))^k$ tight components in~$G'$.

	We now consider loose components.
	For $k=2$, tight components and loose components are equivalent
	and, for $k\geq 3$, all vertices have degree at least $n'/2$.
	Thus the graph is connected.
	Note that each edge of $G'$ is contained in a $k$-clique, and hence its endvertices are contained in the same loose component.
	Together this implies that $G'$ has just one loose component.
	So, in the case $k=2$, property~\ref{item:framework-divisibility} is satisfied with
	$\ell=t$ and, in the case $k\geq 3$,
	we have \ref{item:framework-divisibility} is
	satisfied with $\ell=1$.

	Finally we check property~\ref{item:framework-linkage}.
	This follows from \cref{obs:min_degree_second_neighbours} as
	$1-1/\chi' \geq (k-1)/(k-2) $ and $\mu' n'/2>1$.
\end{proof}

\subsection{Degree sequences}
\label{sec:degree-sequences}
In this section, we prove \cref{lem:deg-seq-mixed-framework}, which states that strong degree sequences guarantee tiling frameworks.
\urldef\myurls\url{%https://www.wolframalpha.com/input/?i=%281-2m-a%29%2F%281-m-a%29+for+m%3D%281-a%29%2F%28k-1%29}
We use the following result of Hyde, Liu and Treglown~\cite{HLT19}.
\begin{theorem}\label{thm:deg-seq-komlos}
	Let  $1/n \ll \rho \ll \mu,1/\chi,1/f$.
	Let $F \in \cFcr$ be a graph on $f$ vertices and $\crit(F) \leq \chi$.
	Let $G$ be a graph on $n$ vertices with {a $(\chi,\mu)$-strong} degree sequence.
	Then $G$ contains a $(1-\rho)$-perfect $F$-tiling.
\end{theorem}

We also require the following proposition.
\begin{proposition}\label{pro:deg-seq-linked-edges}
	Let $1/n \ll \mu, 1/\chi$ and $k=\kk$.
	Suppose $G$ is a graph on $n$ vertices with a $(\chi,\mu)$-strong degree sequence.
	Then $K_k(G)$ is linked.
	Moreover, if $k\geq 3$, then $K_k(G)$ has only one loose component.
\end{proposition}
\begin{proof}
	We start by showing that $K_k(G)$ is linked.
	Note that the case $k=2$ is trivial as every vertex has degree at least $2$.
	So let $k\geq 3$.
	Label the vertices of $G$ by $v_1,\dots,v_n$ in ascending order of their degrees.
	Let $\alpha = 1-\frac{k-1}{\chi}$.
	By {assumption}, we have
	\begin{equation}
		\deg (v_i) \geq \left(1- \frac{1}{\chi} - \alpha\right)n +  \chi\alpha \cdot i + \mu n\text{ for all } i \in [n/\chi].
	\end{equation}
	Let $v=v_i$ be an arbitrary vertex of $G$, and set $G'=G[N(v)]$ with $n' = v(G')$.
	Since $G$ has a $(\chi,\mu)$-strong degree sequence, we have $n' \geq (1-\alpha - 1/\chi + \mu) n$.
	Say that a vertex $u \in V(G)$ is \emph{big} if $\deg(u) \geq \left(1- 1/\chi \right)n$.
	Denote by $B$ the set of big vertices.
	Since $G$ has a $(\chi,\mu)$-strong degree sequence, we have $|B| \geq (1-1/\chi)n$.
	Set $B' = B\cap N(v)$.
	It follows that $|B'| \geq (1 - 1/\chi)n - (n-n') = n' - n/\chi$.
	Similarly, $\deg_{G'}(u) \geq \deg_G(u) - (n-n') \geq n' - n/\chi$  for each $u \in B'$.
	We therefore have
	\begin{align*}
		\frac{d_{G'}(u)}{n'}, \frac{|B'|}{n'} & \geq \frac{n' - n/\chi}{n'}
		=1-\frac{n/\chi}{n'}
		\geq 1-\frac{1/\chi}{1- \alpha - 1/\chi + \mu}
		\\&\geq 1-\frac{1/\chi }{1- \alpha - 1/\chi} + \frac{k}{n}
		= \frac{1  -\alpha -2/\chi}{1  -\alpha - 1/\chi} + \frac{k}{n}
		= \frac{k-3}{k-2} + \frac{k}{n}
	\end{align*}
	where we used $n' \geq (1- \alpha - 1/\chi + \mu)n$ and $1/n \ll \mu,1/k$.\COMMENT{\myurls}
	So in particular, $v$ has a big neighbour.
	In fact, we can find a $(k-1)$-clique $K \subset G'$ with at least $k-2$ big vertices, by  greedily selecting  first the big vertices and then one more arbitrary vertex.
	At least $(k-2)(1-1/\chi)n+(1-\alpha - 1/\chi +\mu)n \geq (k-2)n + k$ edges are leaving $K$ in $G$.
	Hence there is some $w \neq v$ such that $V(K) \cup \{w\}$ forms the desired linked edge of $v$ in $K_k(G)$.

	Now for the loose components.
	We assume that $k\geq3$, since there is nothing to show otherwise.
	So in particular, any two big vertices share a neighbour in $G$.
	By the above argument, every edge that contains at least one big vertex can be extended to a $k$-clique.
	Together, this implies that all big vertices and their neighbours are in the same loose component.
	We have also shown above that each vertex has a big neighbour (as $v$ was arbitrary).
	It follows that $K_k(G)$ has one loose component.
\end{proof}

Finally, we use the following technical lemma, which provides a backstop to the number of tight components in robust frameworks.
The short proof can be found in \cref{sec:intermediate}.

\begin{proposition}\label{pro:backstop}
	Let $1/n \ll \mu' \ll \mu,1/\chi$ with $k = \lceil \chi\rceil$ and $\rho >0$.
	Suppose $G$ is a $\mu$-robust $(\chi,\rho;n^k,\ell)$-tiling framework on $n$ vertices.
	Then $G$ is a $\mu$-robust $(\chi,\rho;t,\ell)$-tiling framework on~$n$ vertices where $t \leq 1/\mu'^k$.
\end{proposition}

\begin{proof}[Proof of \cref{lem:deg-seq-mixed-framework}]
	We show that $G$ is a $\mu/4$-robust $(\chi+\mu/4,\rho;n^k,\ell')$-tiling framework with $\ell' = n^k$ if $k=2$ and $\ell'=1$ if $k \geq 3$.
	It then follows by \cref{pro:backstop} that $G$ is a $\mu/4$-robust $(\chi+\mu/4,\rho;t,\ell)$-tiling framework.
	(Note that when $k=2$, we may replace $\ell'$ by $\ell=t$ since loose and tight components are the same in graphs.)

	To show the above claim, let $G'$ be an arbitrary $(\mu/4,\mu/4)$-approximation of $G$ on $n' \geq (1-\mu/4)n$ vertices.
	Further let $\mu'=\mu/4$ and $\chi'=\chi+\mu/4$.
	It is easy to see that $G'$ still has a $(\chi',\mu')$-strong degree sequence.
	To show that $G'$ is a $(\chi',\rho;t,{\ell'})$-tiling framework, we verify the properties detailed in \cref{def:tiling-framework}.
	We apply \cref{thm:deg-seq-komlos} with $B_{\chi'}$ playing the role of $F$ to obtain property~\ref{item:framework-Geometry}.
	Next, we apply \cref{pro:deg-seq-linked-edges} to show that $K_k(G')$ is linked, which gives property~\ref{item:framework-linkage}.
	Moreover, by the second part of \cref{pro:deg-seq-linked-edges}, if $k\geq 3$ then $K_k(G')$ has precisely one loose component.
	We can trivially observe that $K_k(G')$ has at most $\binom{n'}{k}\leq n^k$ tight and loose components, thus applying \cref{pro:backstop} gives the result {(recall that $\chi+\mu/4\leq k$)}.
\end{proof}

\subsection{Universally dense graphs}\label{sec:quasirandom-host}
Here we show \cref{pro:quasirandom-framework}.
The result easily follows from the next two observations, whose proofs are sketched in \cref{sec:uni-dense-explanation}.
We remark that \cref{prop:dense-inseperable-matching} was originally stated with $G$ satisfying an inseparability condition.
However, the (short) proof only requires the weaker assumption of a linear minimum degree.

\begin{restatable}[{Corollary 5.27~\cite{LS23}}]{proposition}{cordensemanywellconnectedcliques}\label{cor:dense-many-well-connected-cliques}
	Let $1/n\ll\rho \leq \eta \ll \mu, d,1/k$.
	Let $G$ be a $(\rho, d)$-dense graph on $n$ vertices with $\delta(G) \geq \mu n$.
	Then $G[S]$ has at least $\eta n^{k}$ $k$-cliques for every $S \subset V(G)$ of size at least $\eta^{1/(2k)} n$.
\end{restatable}

\begin{restatable}[{Proposition 5.28~\cite{LS23}}]{proposition}{propdenseinseperablematching}\label{prop:dense-inseperable-matching}
	Let $1/n \ll \mu' \ll \rho \ll \eta \ll {\mu, d, 1/k}$.
	Let $G$ be a $(\rho, d)$-dense graph on $n$ vertices with $\delta(G) \geq \mu n$.
	Let $G' \subset G$ be a $(\mu',\mu')$-approximation of $G$.
	Then $G'$ contains a perfect fractional $K_k$-tiling.
\end{restatable}

\begin{proof}[Proof of \cref{pro:quasirandom-framework}]
	Let $G' \subset G$ be an arbitrary $(\mu',\mu')$-approximation of $G$ on $n'$ vertices.
	By \cref{prop:dense-inseperable-matching}, $G'$ has a perfect fractional $K_k$-tiling.
	{This implies that $G'$ has a perfect fractional $B_\chi$-tiling.
	Indeed, whenever a $k$-clique $K$ of $G'$ receives weight $x$, we can consider $k$ homomorphisms from $B_\chi$ to $K$, where each homomorphism maps the small part of $B_\chi$ to a distinct vertex of $K$.
	Giving each of these homomorphisms weight $x/k$ results in a fractional $B_\chi$-tiling of $K$ of weight $x$.
	Repeating this for all $k$-cliques of $G'$ turns the perfect fractional $K_k$-tiling of $G'$ into a perfect fractional $B_\chi$-tiling of $G'$.}

	Consider any vertex $v \in V(G')$.
	By assumption, we have $|N_{G'}(v)| \geq \mu n- \mu' n \geq {\mu n'/2}$.
	By \cref{cor:dense-many-well-connected-cliques}, $v$ has at least $\eta n'^{k} \geq \mu' n^k$ linked edges in $K_k(G')$.
	Finally, we can trivially observe that $K_k(G')$ has at most $\binom{n'}{k}\leq n^k$ tight and loose components.
	This shows that $G$ is a $\mu'$-robust $(k,0;n^k,n^k)$-tiling framework.
	Thus applying \cref{pro:backstop} gives the desired result {(recall that there are at most as many loose components as tight components)}.
\end{proof}

\section{Guest graph properties}\label{sec:guest-graph-properties}

This section is dedicated to the proofs of \cref{lem:hcf-gives-flexi-chrom-mixed,lem:miss_constant_frac,prop:quasirandom-wildcards}.

\subsection{Graphs in $\cFcr$ are flexible}\label{sec:hcf-gives-flexi-chrom}
Recall that in the statement of \cref{lem:hcf-gives-flexi-chrom-mixed}, we are given a graph $H$ consisting of (possibly different) tiles $F \in \cFcr$, and we have to show that $H$ contains \fc~subtilings.
This problem can be solved with the help of additive combinatorics.
An \emph{arithmetic progression} is a set of integers $\{x_i\}_{i\in\INTS}$ such that $x_{i+1}=x_i+d$  for all $i\in \INTS$ and some common \emph{difference} $d\in\NATS$.
Note that if a set of integers $D$ has greatest common divisor $\gcd (D) > 1$, then $D$ is contained in an arithmetic progression of difference at least $2$.
(The converse does not hold as can be seen by considering the set of all odd integers).

We derive \cref{lem:hcf-gives-flexi-chrom-mixed} from the following two propositions.
For a (multi)subset $X \subset \INTS$, define the \emph{set of subset sums} as $\mathbf{\Sigma}X = \{\sum_{x\in X'} x \colon X'\subset X\}$.

\begin{proposition}\label{lem:transversal_contains_interval}
	Suppose that $s,w\geq 3$ are integers and $A_1,\dots,A_s\subset [{1},w]$ are non-empty integer sets with $\gcd(A_i)=1$ for all $i\in[s]$.
	Suppose further that one of the following holds,
	\begin{enumerate}[\upshape (a)]
		\item \label{itm:transversal_contains_interval-mixed}	$s \geq {4}0 w\log w$,
		\item \label{itm:transversal_contains_interval-fixed}	$s\geq {2}0 w$ and $A_1 = \dots = A_s$.
	\end{enumerate}
	Then there exists a multiset $X=\{x_i\}_{i \in [s]}$ with $x_i\in A_i$
	such that $\mathbf{\Sigma}X$ contains an interval of length $w$.
\end{proposition}

The proof of \cref{lem:transversal_contains_interval} is deferred to the end of this section.
We also need the following result, which we prove under  more general circumstances in \cref{sec:wildcards} (see \cref{prop:build_wildcards}).
\begin{proposition}\label{prop:build_wildcards-simple}
	Let $k,w \in \NATS$.
	Let $H_1,\dots,H_{k^2}$ be graphs such that $H_i$ has a proper (topological) $k$-colouring $\theta_i$ for each $i\in[k^2]$ satisfying the following property.
	For all $y \in [w]$, there exists a proper (topological) $k$-colouring $\theta'_i$ of $H_i$ such that $\ord(\theta'_i)(j)=\ord(\theta_i)(j)$ for $j\in[k-2]$ and $\ord(\theta'_i)(k)=\ord(\theta_i)(k)+y$.
	Then $\bigcup_{i \in [k^2]}H_i$ is properly (topologically) $(k,w/k)$-\fc{}.
\end{proposition}

Now we are ready to show the first main result of this section.

\begin{proof}[Proof of \cref{lem:hcf-gives-flexi-chrom-mixed}]
	Suppose without loss of generality that $v(H)=n$.
	Indeed by assumption $v(H)\leq n $ and so this can be achieved by decreasing $n$.
	We have to show that $H$ contains $t$ disjoint proper $(k,kw)$-\fc{} subtilings and one topological $(\ell,w)$-\fc{} subtiling each of order at most $w'$.
	We will use \cref{lem:transversal_contains_interval} to build coloured subtilings that can be recoloured to give a controlled increase and decrease in the number of vertices coloured $k$ and~$k-1$ respectively.
	Applying \cref{prop:build_wildcards-simple} then gives the result.

	We begin by considering case \ref{itm:Fcr_for_mixed}.
	Let $w^\ast = k^2w$ and $s = k^{8} w \log w$.
	Note that $s \geq 2^4 k^2 w k^2\log w \geq 10 {(w^\ast+1) \log (w^\ast+1)}$ as $k \geq 2$.
	We begin by showing that $H$ contains $t$ disjoint proper $(k,kw)$-\fc{} subgraphs.
	Our plan is to apply \cref{prop:build_wildcards-simple}.
	Hence, we have to find subtilings $H_1,\dots, H_{k^2}$ satisfying its assumptions.
	To this end, denote the tiles of $H$ by $F_1,\dots,F_q$ and let $A_i=\cD(F_i){\setminus \{0\}}$ for each $i\in [q]$.
	{We remark that $A_i\subseteq [1,w]\subseteq [1,w^\ast]$.}
	Note that by definition of $\cFcr$ we have $\gcd(A_i)=1$.
	Consider tiles $F_1,\dots, F_{s} \subset H$, which is possible as $n \geq s w$ and each $F_i \in \cFcr(k,w)$ has order at most $w$.
	We apply \cref{lem:transversal_contains_interval}\ref{itm:transversal_contains_interval-mixed} to $A_1,\dots,A_{s}$ with {$w^\ast+1$} playing the role of $w$ to obtain a multiset $X=\{x_i\}_{i \in [s]}$ with $x_i\in A_i$ such that $\mathbf{\Sigma}X$ contains an interval $\{z, \dots, z+w^\ast\}$ for some $z \in \NATSZ$.

	For each $i\in[s]$, fix a proper $k$-colouring $\phi_i$ of $F_i$ such that  $x_i = |{\phi}_i^{-1}(k-1)| - |{\phi}_i^{-1}(k)|$.
	By definition of $\mathbf{\Sigma}X$, there is a set $I\subset [s]$ with $\sum_{i\in I}x_i = z$.
	We obtain a proper $k$-colouring $\theta_i$ from each $\phi_i$ by permuting colour classes $k$ and $k-1$ if and only if $i \in I$.
	Let $H_1=\bigcup_{i \in [s]} F_i$ and $\theta = \bigcup_{i \in [s]} \theta_i$.
	Since $n \geq k^2 s w$, we can choose (pairwise disjoint) further subgraphs $H_2,\dots,H_{k^2} \subset H$ with the same property.
	Let $W_1 =\bigcup_{i \in [k^2]} H_i$.

	We aim to apply \cref{prop:build_wildcards-simple} to $H_1,\dots,H_{k^2}$ with $w^\ast$ playing the role of $w$.
	To see that,~$H_1$ say, satisfies the conditions consider an arbitrary $y \in [w^\ast]$.
	We have to show that there exists a second proper $k$-colouring $\theta'$ of $H_1$ such that $\ord(\theta')(j)=\ord(\theta)(j)$ for each $j\in[k-2]$ and $\ord(\theta')(k)=\ord(\theta)(k)+y$.
	By definition of $\mathbf{\Sigma}X$, there exists $I' \subset [s]$ such that $\sum_{i\in I'}x_i = z+y$.
	Let $\theta'$ be the colouring of $H_1$ with restriction $\theta'_i$ to $F_i$ where $\theta'_i$ is $\theta_i$ with colour classes $k-1$ and $k$ permuted if $i$ is in the symmetric difference of $I$ and $I'$ and $\theta'_i=\theta_i$ otherwise.
	Then by construction
	\[\ord(\theta')(k)-\ord(\theta)(k) =  \sum_{i\in I'\sm I}x_i -  \sum_{i\in I\sm I'}x_i  = \sum_{i\in I'}x_i - \sum_{i\in I}x_i = z+y - z = y.\]
	Moreover, we have $\ord(\theta)(j)=\ord(\theta')(j)$ for each $j \in [k-2]$.
	Hence, \cref{prop:build_wildcards-simple} applied with $w^\ast=k^2w$ playing the role of $w$ reveals that $W_1 =\bigcup_{i \in [k^2]} H_i$ is proper $(k,kw)$-\fc{}.
	Moreover, we have $v(W_1) \leq  k^2sw \leq w'$, since {$W_1$} consists of $k^2 s$ tiles of $H$, which have order at most $w$ each.
	We then select $W_2,\dots,W_t \subset H$ pairwise disjoint further proper $(k,kw)$-\fc{} subgraphs in the same way, which is possible since $n \geq {(t+\ell)w'}$.

	Finally, we have to show that $H$ contains a topological $(\ell,w)$-\fc{} subtiling $W$ (disjoint of $W_1,\dots,W_t$).
	If $\ell=1$ this is trivial (any graph is $(1,w)$-\fc{} for any $w\in \NATS$), otherwise we simply repeat the same argument as above, which is possible by definition of $\cFcr$ for $k=2$ and since $w'\geq \ell^{10} w^2\log w$.
	Observe that we only swapped certain colour classes in certain tiles; this preserves proper and topological colourings.
	\COMMENT{We replace all proper colourings by topological colourings and $k$ by $\ell$ in the above arguments. In particular we replace $\cD(F_i)$ by $\cD_\text{top}(F_i)$.}

	For \ref{itm:Fcr_for_non-mixed} the argument is largely the same as above.
	Since $F_1 = \dots =F_q$, we have $A_1 = \dots =A_q$ where $A_i=\cD(F_i)$.
	Set $s = k^{8} w$ and (again) $w^\ast = k^2 w$.
	As $s \geq {20} w^\ast$, we can conclude by \cref{lem:transversal_contains_interval}\ref{itm:transversal_contains_interval-fixed} that $\mathbf{\Sigma}X$ contains the desired interval.
	The rest of the proof is identical.
\end{proof}

It remains to show \cref{lem:transversal_contains_interval}.
To this end, let us introduce some further notation.
Let $A_1,\dots,A_{q}$ be subsets of $\INTS$.
We define the \emph{sumset} $A_1 + \dots + A_{q}$ as $\{a_1+ \dots + a_{q}\colon  a_i\in A_i \text{ for } i\in[{q}]\}$.
Similarly, let $A_1-A_2 = \{a_1 -a_2 \colon  a_1\in A_1 \text{ and } a_2\in A_2\}$.
We require the following result.
\begin{lemma}[Lev~\cite{Lev10}]\label{lem:sumset_contains_interval}
	Suppose that ${q}\geq 1$, ${m}\geq 3$ are integers and $A_1,\dots,A_{q}\subset [0,w]$ are integer sets of order at least {$m$} such that none of them are contained in {an} arithmetic progression  with common difference at least $2$.
	If ${q}\geq 2\lceil (w-1)/({m}-2) \rceil$, then the sumset $A_1+ \dots +A_{q}$ contains an interval of length ${q}({m}-1)$.
\end{lemma}

For $x\in \NATS$, define $\pri(x)$ to be the set of prime factors of $x$ and for $X\subset \NATS$ define $\pri(X)=\bigcap_{x\in X}\pri(x)$.
Thus $\gcd(X)=1$ if and only if $\pri(X)=\es$ since $1$ is not prime.

\begin{proof}[Proof of \cref{lem:transversal_contains_interval}]
	The idea is to construct a set $X$ such that $\mathbf{\Sigma} X$ contains $B_1+\dots+B_q$ where $B_1,\dots,B_q\subset\INTS$ satisfy \cref{lem:sumset_contains_interval}.

	Let us start by proving~\ref{itm:transversal_contains_interval-mixed}.
	We begin constructing the set $X=\{x_i\}_{i \in [s]}$ as follows.
	Suppose first that there are pairwise distinct $x_1 \in A_1,$ $x_2 \in A_2$, $x_3 \in A_3$.
	(We deal with the other case below.)
	We then take a set $\{x_3,x_4,\dots,x_t\}$ with $x_{i+3}\in A_{i+3}$ for $i \in [t-3]$ which is maximal among all sets satisfying $\pri(\{x_{3},\dots,x_j\})\not\subset \pri(x_{j+1})$ for all  $j\in \{3,\dots,t\}$.
	Such a set has $\pri(\{x_{3},\dots,x_t\})= \es$ by maximality and as $\gcd(A_{t+1})=1$.
	Note that, $A_{t+1}$ exists, because $t\leq 10\log w$.
	The latter bound is obtained as follows.
	We have $|\pri(\{x_{3},\dots,x_j\})|\leq |\pri(\{x_{3},\dots,x_{j-1}\})| -1 $ for all $j\in \{4,\dots,t\}$.
	It follows that $t - 3 \leq |\pi(x_3)| \leq \log_2 x_3 \leq 4 \log w$ and so $t\leq 10\log w$.
	Further it is clear that $\{x_3,\dots, x_t\}$ are distinct.
	Denote $B'_1 = \{x_2,\dots,x_t\}$ and let $B_1= \{x_1\} \cup (\{x_1\} + B'_1)$.
	Let $d$ be the difference of an arbitrary arithmetic progression containing $B_1$.
	We claim that $d = 1$.
	This follows since $d$ by definition divides $\gcd(B_1-B_1)$ and we have $B_1' \subset (B_1-B_1)\sm \{0\}$ and $\gcd  (B_1')=1$ by construction.
	{Observe that $B_1\subseteq [1,2w]$.}

	Now suppose that $x_1 \in A_1$, $x_2 \in A_2$, $x_3 \in A_3$ cannot be chosen distinctly.
	Then either we have $|A_i| = 1$ and thus  $ A_i = \{ 1 \}$ for some $i\in [3]$ (say $i=1$), {because $\gcd(A_1)=1$ for all $i\in[s]$}, or we have $|A_i| = 2 $ for all $i\in [3]$ and so $A_1=A_2=A_3 = \{a_1,a_2\}$.
	In the former case we let $B_1 = \{x_2,x_2+x_3,x_2 + x_3+ 1\}$ and in the latter case we let $B_1 = \{a_1,a_1+a_1,a_1+a_2\} $.
	It is easy to check that the common difference of any arithmetic progression containing $B_1$ is also $1$ as in the other case.

	Next we define, in the same way, a set $B_2$ using only the sets $A_i$ that were not used above.
	We continue this way until sets $B_1,\dots,B_{q+1}$ are defined, where $B_{q+1}$ acts as a leftover set.
	So ${B_j \subset [1,2w]}$ is not contained in an arithmetic progression with common difference at least $2$ and $|B_j| \geq 3$ for each $j \in [q]$.
	Moreover, $|B_j| \leq 10 \log w$ for each $j \in [q+1]$.
	The {latter} implies that $q \geq s/(10 \log w) -1 \geq {4}w-1\geq {2(2w-1)}$ by assumption of~\ref{itm:transversal_contains_interval-mixed}.
	Hence, we can apply \cref{lem:sumset_contains_interval} with ${2w}$ playing the role of $w$, ${m}=3$ and $q$ playing the role of $s$ to reveal that $B_1+ \dots +B_q$ contains an interval of length $2q \geq w$.
	Finally, note that $\mathbf{\Sigma} X$ contains $B_1+\dots+B_q$ by construction.
	This shows~\ref{itm:transversal_contains_interval-mixed}.

	The proof of~\ref{itm:transversal_contains_interval-fixed} follows the same pattern.
	However, since $A_1=\dots=A_s  $, we can choose $B_1=\dots=B_q$ with $q=\lfloor \frac{s}{{m}} \rfloor$\COMMENT{$B_1 = x_1 \cup \{x_1 + A_1\}$ for some arbitrary $x_1\in A_1$. Thus $(B_1-B_1)\sm \{0\}$ contains $A_1$} and ${m}=|A_1|+1$.
	By assumption of~\ref{itm:transversal_contains_interval-fixed}, we have $q \geq {10} (w-1) / {m} \geq 2\lceil (w-1)/({m}-2) \rceil$.
	If $|A_1|=1$, then $A_1=\{1\}$ and we are done trivially, so we can apply \cref{lem:sumset_contains_interval} with $w$, ${m}=|A_1|+1\geq 3$ and $q$ playing the role of $s$ to see that $B_1+\dots+B_q$ contains an interval of length $q({m}-1) \geq s/4{\geq w}$.
\end{proof}

\subsection{Graphs of low critical chromatic number are flexible}
\label{sec:miss_constant_frac}

In this section we show \cref{lem:miss_constant_frac}.
We require the following simple result, which states that graphs of bounded maximum degree and sufficiently low critical chromatic number allow for flexible proper colourings.
Recall that for a graph $G$ with chromatic number $k$, we define $\alpha(G)$ to be the minimum proportion of vertices in a single colour class over all proper $k$-colourings of $G$.

\begin{proposition}\label{prop:max_degree_gives_wildcards}
	Let $F$ be a graph on $w$ vertices with  $\Delta(F) = \Delta$,	$\alpha(F) < \frac{1}{\Delta + 1}$ and $\chi(F)=k \geq 3$.
	Then there exists a proper $k$-colouring $\phi$ of $F$ with the following properties.
	For all $y\in \NATSZ$ with
	$y < \frac{w}{k}(\frac{1}{\Delta+1}-\alpha(F))$, there exists a proper $k$-colouring $\phi'$ of $F$ with $\ord(\phi')(i)=\ord(\phi)(i)$ for $i\in [k-2]$ and $\ord(\phi')(k)=\ord(\phi)(k)+y$.
\end{proposition}
\begin{proof}
	Let $\phi$ be a proper $k$-colouring of $F$ such that $\ord(\phi)(k)=\alpha(F)w$.
	Observe that the number of vertices that are adjacent to a vertex in $\phi^{-1}(k)$ is at most $\Delta\alpha(F)w$.
	Note that $\Delta\alpha(F)w + \alpha(F) w = (\Delta + 1) \alpha(F)w < w$ by assumption.
	Thus some vertices in $V(F) \sm \phi^{-1}(k)$  are not adjacent to any vertex in $\phi^{-1}(k)$.
	We recolour one of these vertices with the colour $k$.
	As long as the order of the colour class $k$ is less than $w/(\Delta+1)$, we can find another vertex to add to it.
	Thus we can greedily add a total of at least $w(\frac{1}{\Delta+1}-\alpha(F))$ vertices to the colour class $k$.
	By the pigeonhole principle at least ${\frac{w}{k-1}(\frac{1}{\Delta+1}-\alpha(F))\geq \frac{w}{k}(\frac{1}{\Delta+1}-\alpha(F))}$ of these vertices came from a single colour class.
	Without loss of generality let this colour class be $k-1$ (this can be achieved by permuting colours of $\phi$), and let $X$ be the set of vertices  that were previously (but are no longer) coloured with $k-1$.

	Now suppose we are given $y\in \NATSZ$ with
	$y < \frac{w}{k}(\frac{1}{\Delta+1}-\alpha(F))$. Then we build $\phi'$ from $\phi$ by recolouring $y$ vertices in $X$ with $k$ instead of $k-1$.
	It is clear that $\phi'$ is the desired proper $k$-colouring of $F$.
\end{proof}

\begin{proposition}\label{prop:max_degree_bipartite_isolated}
	Suppose $F$ is a bipartite graph with $\Delta(F) \leq \Delta$.
	If $\alpha(F)\leq \frac{1}{\Delta + 1}$, then $F$ contains $v(F)(1-(\Delta+1)\alpha(F))$ isolated vertices.
\end{proposition}
\begin{proof}
	Let $\phi$ be a proper $k$-colouring of $F$ such that $\ord(\phi)(1)=\alpha(F)v(F)$.
	Observe that the number of vertices that are adjacent to a vertex in $\phi^{-1}(1)$ is at most $\Delta\alpha(F)v(F)$.
	However every vertex not coloured with $1$ is coloured with $2$.
	Thus if a vertex is not coloured $1$, and is not adjacent to a vertex of colour $1$, it is an isolated vertex.
	We calculate that therefore all but $(\Delta+1)\alpha(F)v(F)$  vertices are isolated.
\end{proof}

\begin{proof}[Proof of \cref{lem:miss_constant_frac}]
	{By assumption, we have $\chi \leq \frac{\Delta+1}{\Delta}(k -1) (1+ \frac{\mu}{\chi})^{-1}$.
		Solving \eqref{equ:crit} for $\alpha(H)$, we obtain
		\begin{align*}
			\alpha(H) 
			& = 1 - \frac{k-1}{{\crit(H)}} \leq 1-\frac{k-1}{\chi} \leq 1-\frac{\Delta}{\Delta+1} \left(1+\frac{\mu}{\chi}\right) \\
			&\leq 1-\frac{\Delta }{\Delta + 1} - {\frac{\Delta}{\Delta+1}}\frac{\mu}{\chi} 
			= \frac{1}{\Delta+1} - {\frac{\Delta}{\Delta+1}}\frac{\mu}{\chi}.
		\end{align*}
	Hence $\alpha(H) + {\frac{\Delta}{\Delta+1}}\frac{\mu}{{\chi}} \leq \frac{1}{\Delta + 1}$.}
	First suppose $k=2$, then by \cref{prop:max_degree_bipartite_isolated}, $H$ contains $v(H)(1-(\Delta+1)\alpha(H)) \geq \mu n/{2}$ isolated vertices.
	Without loss of generality, we assume that $H$ has $n$ vertices (adding isolated vertices if necessary.)
	As $n$ is sufficiently large we can build (disjointly) $t$ proper $(k,kw)$-\fc{} subtilings and one topological $(\ell,w)$-\fc{} subtiling as in \cref{obs:isolated-vertices-flexible}.

	Now suppose $k\geq 3$.
	We immediately construct the $t$ proper $(k,kw)$-\fc{} subtilings of order at most $w'=\sqrt \xi n$.
	{Observe that $\frac{1}{\Delta+1}-\alpha(H)\geq \frac{\Delta}{\Delta+1}\frac{\mu}{\chi}\geq \frac{\mu}{2(\Delta+1)}$}.
	As we can choose $\xi$ small in terms of ${\mu}$, $k$ and $t$, we may find~$k^2 t$ disjoint subtilings $L$ of $H$ each with $\frac{v(L)}{k}(\frac{1}{\Delta+1}-\alpha(L))\geq k^2w$.
	Indeed, we can find these subtilings greedily by observing that removing a subtiling of order $\sqrt\xi n$ can only increase $\alpha(H)$ and thus~$\crit(H)$ by a function of $\xi$.
	By \cref{prop:max_degree_gives_wildcards}, {each set of $k^2$ subtilings} satisfies the assumptions of
	\cref{prop:build_wildcards-simple} and so the union of said subtilings is $(k,kw)$-\fc{}.
	Using all $k^2t$ of the reserved tilings, we obtain $t$ proper $(k,kw)$-\fc{} subtilings of $H$ each of order at most $w'$.
	{Recall again that (trivially) any graph is a topological $(1,w)$-\fc{} graph.}
	It follows that $H\in \cH(n,w,\Delta;\chi;t,1,\sqrt\xi n)$, as desired.
\end{proof}

\subsection{Graphs in $\cF_c$ are flexible}\label{sec:Fc-gives-flexi-chrom}
Here we show \cref{prop:quasirandom-wildcards}.
We require the following result, which states that $k$-colourable graphs are properly $(k+1,s)$-\fc{} for suitably $s$.
Its proof is deferred to \cref{sec:wildcards}.

\begin{proposition}\label{lem:k+1_wildcard}
	Let $1/n\leq \xi \ll 1/k$.
	Let $w\leq \xi n$ and $H$ be an {$\cF(k,w)$-tiling}
	on $n$ vertices with $\crit(H) \leq k$.
	Then $H$ is properly $(k+1, n/(k+1)^5)$-\fc{}.
\end{proposition}

{We conclude the section with the proof of \cref{prop:quasirandom-wildcards}.}

\begin{proof}[Proof of \cref{prop:quasirandom-wildcards}]
	Consider $H' \subset H$ of order $\sqrt{\xi} n$.
	Then $H' \in \cB(\sqrt \xi n,w,\Delta) \cap  \cJ(k)$ where $w := \sqrt{n} / \log n$.
	Let $w' = v(H')/\ell = \sqrt \xi n / \ell$, and note that $w' \geq  \ell^{10} n / \log n \geq  \ell^{10} w^2 \log w$.
	We can therefore find a topological $(\ell,w)$-\fc~subtiling $W \subset H'$ with $v(W) \leq w'$ as in the proof of \cref{lem:hcf-gives-flexi-chrom-mixed}.
	We omit the details.
	Moreover, using \cref{lem:k+1_wildcard}, it is easy to see that $H$ has $\ell$ pairwise disjoint (and disjoint of $W$) properly $(k+1,w)$-\fc~subtilings each of size at most $w'$.
	It follows that $H \in \cH(n, \sqrt{n}/\log n,\Delta;k+1;\ell,\ell,\sqrt   \xi n)$.
\end{proof}

\section{Tools for embedding spanning subgraphs}\label{sec:tools}

This section contains a summary of a few standard tools regarding embeddings of large sparse substructures in dense graphs,
which are needed later in our proofs.

\subsection{Regular and super-regular pairs}

The Regularity Lemma guarantees that every large enough graph $G$ admits a reduced graph $R$ of bounded complexity, which encodes the approximate structure of $G$.
Before we can make this precise, we introduce the notion of regular pairs.

Let $G$ be a graph and~$A$ and~$B$ be non-empty, disjoint subsets of $V(G)$.
We write $e_G(A,B)$ for the number of edges in~$G$ with one vertex in~$A$ and one in~$B$ and define the \emph{density} of the pair $(A,B)$ to be $d_{G}(A,B)=e_G(A,B)/(|A||B|)$.
The pair $(A,B)$ is \emph{$\eps$-regular} in~$G$ if $|d_{G}(A',B') - d_{G}(A,B)| \leq \eps$ for all $A'\subseteq A$ with $|A'|\ge\eps|A|$ and $B'\subseteq B$ with $|B'|\ge\eps |B|$.
We say that $(A,B)$ is \emph{$(\eps,d)$-regular} if it is $\eps$-regular and has density at least~$d$.
A vertex $b$ in~$B$ (and analogously in $A$) has \emph{typical degree} in an $(\eps,d)$-regular pair $(A,B)$ if $b$ has at least $(d-\eps)|A|$ neighbours (in $G$); and we say $b$ has \emph{atypical degree} otherwise.
We say a pair $(A,B)$ is \emph{$(\eps,d)$-super-regular} if it is $(\eps,d)$-regular
and contains no vertices of atypical degree (in $A$ or $B$).

Let~$G$ be a graph and $\cV=\{V_i\}_{i \in [r]}$ be a partition of $V(G)$ of \emph{size} $r$.
We refer to the elements of $\cV$ as \emph{clusters}.
We say that the partition~$\cV$ is \emph{$\kappa$-balanced} (or just \emph{balanced} when $\kappa=1$) if there exists $m\in\mathbb N$ such that  $m\le|X_i|\le\kappa m$ for all $i \in [r]$.
Let $R$ be a graph on vertex set~$[r]$.
We say that $(G,\cV)$ is an \emph{$R$-partition} if no cluster of~$\cV$ is empty, {$G$ has no edge with both ends in the same cluster,} and whenever there are edges of~$G$ between~$V_i$ and~$V_j$, the pair~$ij$ is an edge of~$R$.
(Equivalently, we could define $\cV$ by $V_i=\phi^{-1}(i)$ for $i\in [r]$, where   $\phi\colon V(G)\rightarrow V(R)$ is a surjective homomorphism.)
{Moreover, $(G,\cV)$ is \emph{$\kappa$-balanced} if $\cV$ is $\kappa$-balanced.}
We say that $(G,\cV)$ is an \emph{$(\eps,d)$-regular} $R$-partition if for each $ij\in E(R)$, the pair $(V_i,V_j)$ is \emph{$(\eps,d)$-regular}.
Moreover, $(G,\cV)$ is \emph{$(\eps,d)$-super-regular on~$R$} if for each $ij\in E(R)$, the pair $(V_i,V_j)$ is \emph{$(\eps,d)$-super-regular}.

It is a well-known fact that every dense regular pair contains a large super-regular subpair.
The following proposition of Kühn, Osthus and Taraz~\cite[Proposition 8]{KOT05} generalises this insight to bounded degree subgraphs of the reduced graph.

\begin{proposition}[Super-regularising $R'$] \label{pro:super-regularising-R'}
	{For $1/m  \ll \eps \ll d, 1/{\Delta_{R'}}$ and $r \geq 2$.}
	Let $R$ be a graph on vertex set $[r]$ and let $R'\subset R$ be a spanning subgraph with $\Delta(R')\leq \DeltaRp$.
	Let~$G$ be a graph with a balanced vertex partition $\cV=\{V_i\}_{i \in [r]}$ {with $|V_i|=m$ for each $i \in [r]$}.
	Suppose that $(G,\cV)$ is an $(\eps,d)$-regular $R$-partition.
	Then there is a balanced family $\cV'=\{V_i'\}_{i \in [r]}$ of subsets $V_i' \subset V_i$ of size $\lfloor (1 - \sqrt{\eps})m \rfloor$ such that  $(G[\cV'],\cV')$ is a $(2\eps,d/2)$-regular $R$-partition, which is $(2\eps,d/2)$-super-regular on~$R'$.
\end{proposition}

The next proposition by Böttcher, Schacht and Taraz~\cite[Proposition 8]{BST08} states that regular and super-regular pairs are robust regarding small alterations on their vertex sets.

\begin{proposition}[Robust regular and super-regular pairs]\label{proposition:robust-regular}
	Let $(A,B)$ be an $(\eps, d)$-regular pair, and let $(A', B')$ be a pair such that $|A \triangle A'| \leq \alpha |A'|$ and $|B \triangle B'| \leq \beta |B'|$ for some $0 \leq \alpha, \beta \leq 1$.
	Then $(A', B')$ is an $(\eps', d')$-regular pair where
	\[ \eps' = 2\eps  + 6 (\alpha^{1/2} + \beta^{1/2}) \quad \text{and} \quad d' = d - 2 (\alpha + \beta). \]
	If, moreover, $(A,B)$ is $(\eps, d)$-super-regular, each vertex in $A'$ has at least $(d-\eps')|B'|$ neighbours in $B'$ and each vertex in $B'$ has at least $(d-\eps') |A'|$ neighbours in $A'$, then $(A', B')$ is $(\eps', d')$-super-regular.
\end{proposition}

\subsection{The Regularity Lemma and the Blow-Up Lemma}
In the following we introduce our two main tools for embedding large subgraphs.
Szemer\'edi's Regularity Lemma~\cite{Sze76} allows us to partition the vertex set of a graph into clusters of vertices such that most pairs of clusters are regular.
We will use the following degree form~\cite[Theorem 1.10]{KS96}, which we restate here in terms of $R$-partitions and approximations {as defined in \cref{sec:main-result}.}

{Recall that for $\eps,d\in [0,1]$, a subgraph $G'$ of a graph $G$ is an {$(\eps,d)$-approximation of $G$} if $\deg_{G'}(v) \geq \deg_{G}(v) - d v(G)$ for all $v \in V(G')$ and $v(G') \geq (1 - \eps)v(G)$.}

\begin{lemma}[Regularity Lemma]\label{lem:regularity}
	Let $1/n \ll 1/r_1 \ll 1/r_0, \eps, d$.
	Let $G$ be a graph on $n$ vertices.
	Then there is $r_0 \leq r \leq r_1$, a subgraph $G' \subset G$ with a balanced partition $\cV$ of size $r$ and a graph $R$ on vertex set $[r]$ such that
	\begin{enumerate}[\upshape (S1)]
		\item $G'$ is an $(\eps,d+\eps)$-approximation of $G$ and
		\item $(G',\cV)$ is an $(\eps,d)$-regular $R$-partition.
	\end{enumerate}
\end{lemma}

Next, we state the Blow-Up Lemma, which allows us to embed spanning subgraphs into regular pairs.
{The original Blow-Up Lemma (\cref{lem:simple-blow-up}) was proved by Koml\'os, S\'ark\"ozy and Szemer\'edi~\cite{KSS97}.
For our purposes, it is practical to work with the following variant due} to Allen, B{\"o}ttcher, H{\`{a}}n, Kohayakawa and Person~\cite[Lemma 7.1]{ABH+16}.

Let~$G$, $H$, $R$ be graphs, where $R$ has vertex set $[r]$.
Let $\cV=\{V_i\}_{i \in [r]}$ be a partition of $V(G)$ and $\cX=\{X_i\}_{i \in [r]}$ be a partition of $V(H)$, and let $\tcX=\{\tX_i\}_{i \in [r]}$ be a family of subsets of $V(H)$.
The partitions~$\cX$ and~$\cV$ are \emph{size-compatible} if $|X_i|=|V_i|$ for all $i\in [r]$.
The family $\tcX=\{\tX_i\}_{i \in [r]}$  is an \emph{$(\alpha,R)$-buffer} for $(H,\cX)$ if  $\tX_i\subset X_i$, $|\tX_i|\ge\alpha|X_i|$ for each $i\in [r]$ and for each $x\in\tX_i$ and $xy,yz\in E(H)$ with $y\in X_j$ and $z\in X_k$, we have $ij,jk\in E(R)$.

\begin{lemma}[Blow-Up Lemma]\label{lem:blow-up}
	Let $1/n \ll 1/r \ll \eps \ll  1/\Delta,1/\DeltaRp,1/\kappa,\alpha,d$.
	Let $R$ be a graph on vertex set $[r]$, and let $R'\subset R$ be a spanning subgraph with $\Delta(R')\leq \DeltaRp$.
	Let~$G$ and~$H$ be graphs on $n$ vertices with $\kappa$-balanced size-compatible vertex partitions $\cV=\{V_i\}_{i \in [r]}$ and $\cX=\{X_i\}_{i \in [r]}$, respectively.
	Let $\tcX=\{\tX_i\}_{i \in [r]}$ be a family of subsets of $V(H)$.
	Suppose that
	\begin{enumerate}[\upshape (B1)]
		\item $\Delta(H)\leq \Delta$, $(H,\cX)$ is an $R$-partition, and $\tcX$ is an $(\alpha,R')$-buffer for $(H,\cX)$ and
		\item \label{item:blowup-regsuperreg} $(G,\cV)$ is an $(\eps,d)$-regular $R$-partition, which is $(\eps,d)$-super-regular on~$R'$.
	\end{enumerate}
	Then there is an embedding $\psi\colon V(H)\to V(G)$ with $\psi(X_i) = V_i$ for all $i \in [r]$.
\end{lemma}

{We also require the following version of the Blow-Up Lemma with image restrictions~\cite[Theorem 1 and Remark 13]{KSS97}, which we restate in the terminology of this paper.
\begin{lemma}[Blow-Up Lemma with image restrictions]\label{lem:simple-blow-up}
		Let $1/n \ll \eps,\zeta \ll  1/r,1/\Delta,d,c$.
		Let $R$ be a graph on vertex set $[r]$.
		Let~$G$ and~$H$ be graphs on $n$ vertices with balanced size-compatible vertex partitions $\cV=\{V_i\}_{i \in [r]}$ and $\cX=\{X_i\}_{i \in [r]}$, respectively.
		Suppose that
		\begin{enumerate}[\upshape (B1)]
			\item $\Delta(H)\leq \Delta$ and $(H,\cX)$ is an $R$-partition,
			\item $(G,\cV)$ is an $(\eps,d)$-super-regular $R$-partition and
			\item for each $i \in [r]$, there is a set $X_i^\ast \subset X_i$ with $|X_i^\ast| \leq \zeta n$ such that for every $x \in X_i^\ast$, there is a set $I_x \subset V_i$ with $|I_x| \geq c |V_i|$.
		\end{enumerate}
		Then there is an embedding $\psi\colon V(H)\to V(G)$ with $\psi(X_i) = V_i$ for all $i \in [r]$.
		Moreover, $\psi(x) \in I_x$ for every $x \in X_i^\ast$ and $i \in [r]$.
\end{lemma}}

\section{Intermediate results}\label{sec:intermediate}
In the following, we show a series of results that act as building blocks for the rest of the paper.

\subsection{Properties of robust tiling frameworks}\label{sec:robust-inheritance}
In this section, we show three propositions about tiling frameworks.
Recall that a vertex $v$ and an edge  $f \not \ni v$ in a $k$-graph $\cJ$ are {linked} if there is an edge $e \in E(\cJ)$ with $v \in e$ and  $|e \cap f| = k-1$.
We also say that a vertex is linked if there exists an edge to which it is linked.
Moreover, in a tiling framework (\cref{def:tiling-framework}), every vertex is linked to some edge of the clique hypergraph.
The following two propositions are adapted from the work of Lang and Sanhueza-Matamala~\cite[Propositions 3.6 and 7.8]{LS23}.
For sake of completeness, their proofs are included in \cref{sec:app-regularity}.
The first result states that in a robust tiling framework, there are in fact many edges linked to each vertex.
\begin{restatable}[Supersaturated linkage]{proposition}{propsupsatlink}\label{prop:supersaturated-linkage}
	{Let $\ell, t \in \NATS$ and $\rho >0$, $1/n  \ll \mu' \ll 1/\chi, \mu$ with $k=\lceil \chi \rceil$.}
	Let $G$ be a $\mu$-robust $(\chi,\rho;t,\ell)$-tiling framework on $n$ vertices.
	Then every vertex of~$G$ is linked to at least $\mu' n^k$ edges in $K_k(G)$.
\end{restatable}
The second result states that the reduced graph of a robust tiling framework is again a tiling framework.
\begin{restatable}[Reduced frameworks]{proposition}{propinheritframework} \label{prop:inherit-framework}
	Let $\ell, t \in \NATS$ and $1/n \ll \rho \ll 1/r  \ll \mu, 1/\chi$.
	Let $G$ be a $4\mu$-robust $(\chi,\rho;t,\ell)$-tiling framework on $n$ vertices.
	Let $(G,\cV)$ be a balanced $R$-partition, where $R$ has vertex set $[r]$.
	Then $R$ is a $\mu$-robust $(\chi,\rho;t,\ell)$-tiling framework.
\end{restatable}

We conclude this section, by showing \cref{pro:backstop}.
For the proof, we require the following result proved implicitly by Lang and Sanhueza-Matamala~\cite[Proposition 3.3]{LS20}.
Its proof can be found in \cref{sec:app-regularity}.
\begin{restatable}{proposition}{propdensetightcomponent}\label{prop:dense-tight-component}
	Let $1/n \ll \eps , \mu,1/k$ and $k \geq 2$. Let $G$ be a $k$-graph on $n$ vertices with at least $(1+\eps )\mu \binom{n}{k}$ edges.
	Then $G$ contains a tight component $T$ with $e(\partial T) \geq \mu^{k-1} \binom{n}{k-1}$.
\end{restatable}

\begin{proof}[Proof of \cref{pro:backstop}]
	Let $J=K_k(G)$.
	By \cref{prop:supersaturated-linkage}, every vertex $v \in V(G)$ can be associated with a subgraph $J_v \subset J$ with at least $\mu' n^k \geq 2\mu' \binom{n}{k}$ edges, each of which is linked to $v$.
	By \cref{prop:dense-tight-component} applied with $\eps=1$, there is a tight component $T_v \subset J_v$ satisfying $e(\partial T_v) \geq {\mu '}^{k-1} \binom{n}{k-1}$.
	Write $T'_v \subset J$ for the tight component of $J$ that contains~$T_v$.
	Let $J'= \bigcup_{v \in V(G)} T'_v$ and denote its tight components by $T_1,\dots,T_t$.
	By {the} definition of linked edges, $J'$ spans the vertices of $J$.
	Moreover, by definition of tight connectivity,
	we have $t \mu'^{k-1} \binom{n}{k-1}\leq \sum_{i \in [t]} e(\partial(T_i)) = e({\partial} J') \leq \binom{n}{k-1}$, which gives {$t \leq \mu'^{1-k}$.}
	
	It remains to show that $J'$ is $(t,\ell)$-connected {or can be transformed into something that is.}
	Let $L_1,\dots,L_{\ell'}$ denote the loose components of $J'$, and note that $\ell'\leq t$.
	(This is because every tight component is contained in a single loose component.)
	If $\ell' \leq \ell$, we are done.
	So suppose that $\ell' > \ell$.
	By assumption, $J$ has at most $\ell$ loose components.
	Take a maximal set of loose components $L'_1, \dots ,L'_s \subset J'$ which are in the same loose component of $J$.
	Note that $s\geq 2 $ because $\ell'>\ell$.
	Take a tight component $T \subset J$ such that $V(T)$ intersects, but is not contained in, $V(L'_s)$.
	{We can find such a component $T$ because $s \geq 2$ and $L'_1, \dots ,L'_s$ are in the same loose component of $J$.
		(It may be the case that $T$ is just a single edge.)
		Since $J'$ is spanning, it follows that $V(T)$ intersects $V(L'_j)$ for some $j\in[s-1]$.}
	Adding $T$ to $J'$ therefore decreases the number of $J'$'s loose components.
	We continue this process (updating the tight and loose components) until $\ell' \leq \ell$.
	We remark that by the end, the number of tight components in $J'$ will have at most doubled.
	
	In conclusion, 	$J' \subset J$ is a (vertex) spanning subgraph which has at most $2t \leq 2\mu'^{1-k} \leq 1/\mu'^k$ tight components and at most $\ell$ loose components.
	Hence $J$ is $(2t,\ell)$-connected, and we may finish with $2t$ playing the role of $t$.
\end{proof}

\subsection{Observations on the critical chromatic number}\label{sec:observation-on-crit-chromatic-number}

Here, we discuss a few facts about the critical chromatic number.
The following lemma in particular implies that if the host graph~$G$ has a  fractional $B_\chi$-tiling, then it also has a fractional $B_{\chi'}$-tiling for $\chi'\leq \chi$.
This also justifies why our frameworks are stated in terms of $B_\chi$-tilings.
\begin{lemma}\label{lem:bottle_frac_H_tiling}
	Let $\chi > 1$ and $F$ be a graph with $\crit(F)\leq \chi$.
	Let $G$ be a graph with a fractional $B_\chi$-tiling of weight $s$.
	Then $G$ has a fractional $F$-tiling of weight $s$.
\end{lemma}
\begin{proof}
	Observe that $\delta(B_\chi) \geq (1-1/\crit(F))v(B_\chi)$.
	Therefore, by \cref{lem:degree_frac_bottle_tiling},
	$B_\chi$ has a perfect fractional $F$-tiling, say $\omega \colon \Hom(F;B_\chi)\rightarrow [0,1]$.
	Let $\omega' \colon \Hom(B_\chi;G) \rightarrow [0,1]$ be a fractional $B_\chi$-tiling of weight $s$.
	Define $\omega''\colon \Hom(F;G)\rightarrow [0,1]$ by setting for each $\theta''\in \Hom(F;G)$,
	\[
	\omega''(\theta'')= \sum_{(\theta,\theta')} \omega'(\theta')\omega(\theta)
	\]
	where the sum is taken over all pairs $\theta\in \Hom(F;B_\chi)$ and $\theta'\in \Hom(B_{\chi};G)$ such that $\theta'\circ \theta=\theta''$.
	A straightforward calculation shows that $\omega''$ is a fractional $F$-tiling of weight $s$.\COMMENT{Indeed we can check that for all $u\in V(G)$ we have
		\begin{align*}
			\sum_{\theta''\in\Hom(F;G)}\omega''(\theta'')|\theta''^{-1}(u)| & =   \sum_{v\in V(B_\chi)}\sum_{\theta\in \Hom(F;B_\chi)}\sum_{\substack{\theta' \in \Hom(B_\chi;G) \\
					\theta'(v)=u}}\omega'(\theta')\omega(\theta)|\theta^{-1}(v)|
			\\
			& =\sum_{v\in V(B_\chi)}\sum_{\substack{\theta' \in \Hom(B_\chi;G)                                   \\
					\theta'(v)=u}}\omega'(\theta')\sum_{\theta\in \Hom(F;B_\chi)}\omega(\theta)|\theta^{-1}(v)|
			\\
			& = \sum_{v\in V(B_\chi)}\sum_{\substack{\theta' \in \Hom(B_\chi;G)                                  \\
					\theta'(v)=u}}\omega'(\theta')
			\\
			& = \sum_{\theta' \in \Hom(B_\chi;G)} w'(\theta')|\theta'^{-1}(u)|.
		\end{align*}
		In the first line we used the fact that if $\theta'\circ\theta = \theta''$ then $|\theta''^{-1}(u)|=\sum_{v\in \theta'^{-1}(u)}|\theta^{-1}(v)|$.
	}
\end{proof}

We obtain the following simple consequence.

\begin{corollary}\label{cor:fractional-tiling-weight}
	Let $1 < \chi' \leq \chi$ {with $\chi,\chi' \in \mathbb{Q}$.}
	Let $G$ be a graph with a fractional $B_{\chi}$-tiling of weight $s$.
	Then $G$ has a fractional $B_{\chi'}$-tiling of weight at least $s$.
\end{corollary}

{Finally, a more specific remark.
\begin{observation}\label{obs:framework-monotone}
	For all $\chi>1$, $t,\ell\geq 1$, $\mu \geq \mu' > 0$ with $\chi + \mu, \chi + \mu' \in \mathbb{Q}$ and $\rho \in [0,1]$ the following holds.
	Suppose $G$ is a $\mu$-robust $({\chi+\mu},\rho;t,\ell)$-tiling framework.
	Then $G$ is also a $\mu'$-robust $({\chi+\mu'},\rho;t,\ell)$-tiling framework.
\end{observation}}
\begin{proof}
	{Consider a $(\mu',\mu')$-approximation $G' \subseteq G$.
	Since $G'$ is also a $(\mu,\mu)$-approximation of $G$, it follows that $G'$ is a $({\chi+\mu},\rho;t,\ell)$-tiling framework.
	So $G'$ has a $(1-\rho)$-perfect fractional {$B_{\chi+\mu}$}-tiling, $K_{\lceil {\chi+\mu} \rceil}(G')$ is $(t,\ell)$-connected and $K_{\lceil {\chi+\mu}  \rceil}(G')$ is linked.}
	
	{By \cref{cor:fractional-tiling-weight}, $G'$ has also a $(1-\rho)$-perfect fractional {$B_{\chi+\mu'}$}-tiling.
	Moreover, $K_{\lceil {\chi+\mu'} \rceil}(G')$ is $(t,\ell)$-connected and $K_{\lceil {\chi+\mu'} \rceil}(G')$ is linked.
	For $\lceil {\chi+\mu'} \rceil = \lceil {\chi+\mu} \rceil$, this is trivial.
	For $\lceil {\chi+\mu'} \rceil < \lceil {\chi+\mu} \rceil$, this follows because every edge of $K_{\lceil {\chi+\mu'} \rceil}(G')$ is contained in an edge of $K_{\lceil {\chi+\mu} \rceil}(G')$.
	Hence $G'$ is also a $({\chi+\mu'},\rho;t,\ell)$-tiling framework.}
\end{proof}

\subsection{Bounded degree covers}
In this section, we show a basic result about covering vertices with the homomorphic images of a graph $F$.
More precisely, we say that a graph $F'$ is a \emph{homomorphic image} of $F$ if there exists a surjective homomorphism from $F$ to $F'$.
Let $R$ be a graph.
We say that a spanning subgraph $R' \subset R$ is a {\emph{$(1-\rho)$-homomorphism $F$-cover} of $R$} if $R'$ is the union of homomorphic images of $F$ such that all but at most $\rho r$ vertices {of $R$} belong to a vertex set of a homomorphic image of $F$.
When $\rho=0$, we just say \emph{homomorphism $F$-cover} instead.
The next lemma states that graphs with almost perfect fractional $F$-tilings admit approximate homomorphism $F$-covers of bounded degree.

\begin{lemma}[Bounded degree cover]\label{lem:bounded-degree-cover}
	Let $1/r \leq \rho < 1/{f^6} \leq 1/2^6$.
	Let $F$ be a graph on $f$ vertices, and let $R$ be a graph on~$r$ vertices with a $(1-\rho)$-perfect fractional $F$-tiling.
	Then $R$ contains a $(1-\rho f^6)$-homomorphism $F$-cover $R'$ with $\Delta(R') \leq f^4$.
\end{lemma}

We derive  \cref{lem:bounded-degree-cover} from the following `exact' result.

\begin{lemma}\label{lem:bounded-degree-cover-tight}
	Let $F$ be a  graph on $f \geq 2$ vertices.
	Let $R$ be a graph with a perfect fractional $F$-tiling.
	Then $R$ contains a {homomorphism} $F$-cover $R'$ with $\Delta(R') \leq f^4$.
\end{lemma}

\cref{lem:bounded-degree-cover-tight} is a consequence of the following result, due to Allen, Garbe, Lang, Mycroft and Sanhueza-Matamala~\cite{GLL+21}.
{A proof of this can be found in \cite[Lemma 7.7]{LS23}.
	Related work has been recently obtained by Planken and Ueckerdt~\cite{PU23}.}
A \emph{perfect fractional matching} in a $k$-graph $G$ is a function $\omega \colon E(G) \to [0,1]$ such that $\sum_{e\ni v} \omega(e) = 1$ for every vertex $v \in V(G)$.
A \emph{cover} of a $k$-graph $J$ is a spanning subgraph $J' \subset J$ in which every vertex is covered by at least one edge.
The \emph{maximum degree}~$\Delta(J)$ of~$J$ is the maximum number of edges containing a vertex of $J$.

\begin{lemma}\label{lem:bounded-degree-cover-matching}
	Let $J$ be a $k$-graph with a perfect fractional matching.
	Then $J$ admits a cover of maximum degree at most $k^2+1$.
\end{lemma}

\begin{proof}[Proof of \cref{lem:bounded-degree-cover-tight}]
	Let $R^\ast$ be an $f$-blow-up of $R$ and let $\eta$ be the obvious homomorphism from~$R^\ast$ to $R$.
	We can associate to any $\theta\in \Hom(F;R)$, the set $\Phi_\theta$ of all injective $\theta^\ast \colon\Hom(F;R^\ast)$ such that $\eta \circ \theta^\ast = \theta$.
	It is clear that $\frac{f}{|\Phi_\theta|}\sum_{\theta^\ast\in \Phi_\theta}|{\theta^\ast}^{-1}(u)| = |\theta^{-1}(v)|$ for all $v\in V(R)$ and $u\in \eta^{-1}(v)$.
	
	{Let $\omega \colon \Hom(F;R) \to [0,1]$ be a perfect fractional $F$-tiling of $R$.
		Define $\omega^\ast \colon \Hom(F;R^\ast) \to [0,1]$ by assigning weight $\omega(\theta)/|\Phi_\theta|$ to every $\theta^\ast \in \Phi_\theta$.
		Then  $\omega^\ast$ is a perfect fractional $F$-tiling of~$R^\ast$ such that all homomorphisms of positive weight are injective.}
	We construct an $f$-graph $J$ on vertex set $V(R^\ast)$ by adding an edge $e$ to $J$, whenever $R^\ast[e]$ contains a copy of $F$.
	Then $\omega^\ast$ can be used to a construct a perfect fractional matching of $J$.
	By \cref{lem:bounded-degree-cover-matching}, it follows that $J$ has a cover $J' \subset J$ of maximum degree at most $f^2+1$.
	{This in turn translates to a homomorphism $F$-cover $R' \subset R$ with $\Delta(R')  \leq f\cdot {(f-1) \cdot} (f^2+1) \leq f^4$, where the factor $f$ comes from the fact that $R^\ast$ is an $f$-blow-up of $R$ and the factor $(f-1)$ comes from the fact that, for any vertex $v$, each homomorphic image of $F$ contributes at most $f-1$ vertices to $N_{R'}(v)$.}
\end{proof}

\begin{proof}[Proof of \cref{lem:bounded-degree-cover}]
	Let $\rho'$ with $\rho  \leq \rho'\leq f\rho$ such that $\rho'r\in \NATS$.
	We obtain a graph $R_0$ by adding  $\rho' r (f-1)$ vertices to $R$ that are adjacent to all other vertices (including each other).
	Take a $(1-\rho')$-perfect fractional $F$-tiling $\omega$ of $R$ and suppose the weight given to $v\in V(R)$ is $1-c_v \in [0,1]$.
	Let $\Phi_v = \{ \phi \in \Hom(F; {R_{v}}) \colon |\phi^{-1}(v)|=1\}$ where ${R_v} = R_0 - (V(R) \sm \{v\})$.
	
	We define a perfect fractional $F$-tiling $\omega'$ of $R_0$ in two steps.
	First, we let $\omega'(\phi)=\omega(\phi)$ for~$\phi$ in $\Hom(F;R)$ and $\omega'(\phi)=\frac{c_v}{|\Phi_v|}$ for all $\phi \in \Phi_v$ and $v\in V(R)$.
	{So every vertex in $R$ receives weight $1$ by definition of $\Phi_v$.
		Moreover, all vertices in $R_0 - V(R)$ receive the same weight, which we denote by $c \in [0,1]$.
		In a second step, take a perfect fractional $F$-tiling of the complete graph $R_0-V(R)$.
		Rescaling with $1-c$ and adding it to $\omega'$ gives the desired perfect fractional $F$-tiling of $R_0$.}

To finish, we apply \cref{lem:bounded-degree-cover-tight} to obtain a homomorphism $F$-cover $R'_0 $ of $R_0$ with $\Delta(R'_0) \leq f^4$.
Let $R'$ be obtained from $R_0'$ by deleting $V(R_0) \sm V(R)$.
Since every added vertex was incident with at most $f^4$ homomorphic images of $F$ in $R'_0$ and $v(F)=f$, it follows that $R'$ is a $(1-\rho' f^6)$-homomorphism $F$-cover  with $\Delta(R') \leq f^4$.
\end{proof}

\section{Koml\'os' theorem for frameworks}\label{sec:komlos-framework}
Koml\'os' theorem allows us to find an almost perfect $F$-tiling in a graph  $G$ of sufficiently large minimum degree when $F$ is of fixed order.
Here, we extend this result to the setting where~$G$ is a tiling framework.
{In this and the following sections, we consider graphs $B_\chi$ (as in defined \cref{sec:host-graphs}) where $\chi=(k-1) + a/b$ is determined via given parameters $a,b,k$.
We thereby tacitly assume that $a$ and $b$ are always coprime and $a \leq b$, so that $B_\chi = K_{a,(k-1)\ast b}$.}

\begin{theorem}[Koml\'os' theorem for frameworks] \label{thm:Komlos-for-frameworks}
Let $\ell,t\in\NATS$ and $1/n \ll \rho \ll \rho' \ll \mu, 1/k,1/a,1/b$ with $\chi = (k-1) + a/b$.
Let $G$ be a $\mu$-robust $(\chi,\rho;t,\ell)$-tiling framework on $n$ vertices.
Then $G$ contains a $(1-\rho')$-perfect $B_\chi$-tiling.
\end{theorem}

For the proof of \cref{thm:Komlos-for-frameworks}, we require the following lemma which establishes the regularity setup that we typically work with.

\begin{lemma}[Lemma for $G$ -- simple form]\label{lem:lemma-for-G-simple}
Let $\ell,t\in\NATS$ and $1/n \ll \rho \ll 1/r_1 \ll 1/r_0 \ll \eps \ll d  \ll \mu, 1/k,1/a,1/b$ with $\chi = (k-1) + a/b$.
Let $G$ be a $\mu$-robust $(\chi, \rho ;t,\ell)$-tiling framework on $n$ vertices.

Then there is $r_0 \leq r \leq r_1$, a subgraph $G' \subset G$ with a balanced partition $\cV$ of size $r$ of $V(G')$, a graph $R$ on vertex set $[r]$ and a homomorphism $B_\chi$-cover $R' \subset R$ {of $R$} with $\Delta(R') \leq (kb)^4$ such that
\begin{enumerate}[\upshape (G1)]
	\item  \label{itm:lemma-for-G-simple-1} $G'$ is an $(\eps,10d)$-approximation of $G$,
	\item  \label{itm:lemma-for-G-simple-2} $(G',\cV)$ is an $(\eps,d)$-regular $R$-partition, which is $(\eps,d)$-super-regular on~$R'$ and
	\item \label{itm:lemma-for-G-simple-3} $R$ is a $(\chi,\rho;t,\ell)$-tiling framework.
\end{enumerate}
\end{lemma}

\begin{proof}
The strategy is to apply \rf{lem:regularity} to $G$ to obtain a reduced graph $R$ which is also a tiling framework by \rf{prop:inherit-framework}.
Then we apply \rf{lem:bounded-degree-cover} to find an approximate homomorphism $B$-cover $R'$ of $R$ where $B:=K_{a,(k-1) \ast b}$,	 whose corresponding regular pairs we turn super-regular using \rf{pro:super-regularising-R'}.

Turning to the details, {let us set $\eps^\ast = (\eps/2)^2$, $d^\ast = 2d$.}
Given the assumptions of \cref{lem:lemma-for-G-simple}, we apply \cref{lem:regularity} to $G$ to obtain, for some $2r_0 \leq {{r^\ast}} \leq r_1$, a subgraph $G' \subset G$ with a balanced partition $\cV=\{V_i\}_{i \in [{{r^\ast}}]}$ of $V(G')$ and a graph $R$ on vertex set $[{{r^\ast}}]$ such that
\begin{enumerate}[\upshape (S1)]
	\item $G'$ is an $({{\eps^\ast}},{{{d^\ast}}}+{{\eps^\ast}})$-approximation of $G$ and
	\item $(G',\cV)$ is an $({{\eps^\ast}},d)$-regular $R$-partition.
\end{enumerate}
Since ${{\eps^\ast}},{{{d^\ast}}} \ll \mu$, it follows that $G'$ is still a $\mu/2$-robust $(\chi,\rho;t,\ell)$-tiling framework.
Hence,  $R$ is a $\mu/8$-robust $(\chi,\rho;t,\ell)$-tiling framework by \cref{prop:inherit-framework}.
Note that $\crit(B) =  \chi$ and $\Delta(B) \leq kb$.
Since $R$ is a $\mu/8$-robust $(\chi,\rho;t,\ell)$-tiling framework, it contains a $(1-\rho)$-perfect fractional $B$-tiling.
{Using}  \cref{lem:bounded-degree-cover}, we obtain a  $(1- {{\eps^\ast}})$-homomorphism $B$-cover $R' \subset R$ {of $R$} with $\Delta(R') \leq (kb)^4$.

Next, we find a family $\cV'$  of subsets  in the partition $\cV$, which has suitable super-regular properties.
Without loss of generality, we may assume that $R'$ has vertex set $[{{{r}}}]$ for some ${{{r}}} \geq (1-{{\eps^\ast}}){{r^\ast}} \geq r_0$.
{(Here we used that ${{r^\ast}} \geq 2r_0$.)}
{So in particular, $R$ is a homomorphism $B$-cover $R' \subset  R[[{{{r}}}]]$ of $R[[{{{r}}}]]$.}
Let $ G'' = G'[\bigcup_{i \in [{{{r}}}]} V_i]$, and note that $( G'', {\{V_i\}_{i \in [{{{r}}}]}})$ is an $({{\eps^\ast}},{{{d^\ast}}})$-regular {$R'[[{{{r}}}]]$-partition.}
{Denote by $m=|V_1|$ the common cluster size.}
We apply \cref{pro:super-regularising-R'} with $G'',(kb)^4,{R[[{{{r}}}]],R'},{{{r}}}$ playing the role of $G, \DeltaRp,R,{R'},r$.
This gives a balanced family $\cV'=\{V_i'\}_{i \in [{{{r}}}]}$ of subsets $V_i' \subset V_i$ of size $m' =\lfloor (1 - {\sqrt{{{\eps^\ast}}}})m \rfloor$ with the following property.
For $G'''=G''[\cV']$, we have that $(G''',\cV')$ is a $({2{{\eps^\ast}},{{{d^\ast}}}/2})$-regular $R[[{{{r}}}]]$-partition, which is $({2{{\eps^\ast}},{{{d^\ast}}}/2})$-super-regular on~$R'$.

There are three reasons vertices in $G$ may not be in $G'''$. They are removed because $G'$ is an $({{\eps^\ast}},{{{d^\ast}}}+{{\eps^\ast}})$-approximation of $G$ (at most ${{\eps^\ast}} n$ vertices), they are missed by $R'$ (at most ${{\eps^\ast}} mr$ vertices), or they are removed when super-regularising (at most $\sqrt{{{\eps^\ast}}}m{{{r}}}$ vertices).
{Recall that $\eps^\ast = (\eps/2)^2$ and $d^\ast = 2d$.}
Using $m{{{r}}} \leq m{{r^\ast}} \leq n$, it follows that $v(G''') \geq n - 2{{\eps^\ast}} n - \sqrt{{{\eps^\ast}}} n \geq (1-2\sqrt{{{\eps^\ast}}}) n \geq {(1-\eps)n}$.
Since ${{{\eps^\ast}}} \ll {{{d^\ast}}}$,  the graph $G'''$ is an {$(\eps,10d)$-approximation of $G$.}
{Hence we may finish with $G'''$, $\cV'$ and $R[[{{{r}}}]]$ playing the role of $G',\cV$ and $R$, respectively.}
Since $v(R')\geq (1-{\eps^\ast})v(R)$ and $R$ is a $\mu/8$-robust {$(\chi,\rho;t,\ell)$-tiling framework}, it follows that $R[[{{{r}}}]]$ is a $(\chi,\rho;t,\ell)$-tiling framework as well.
\end{proof}

Now we are ready to prove \cref{thm:Komlos-for-frameworks}.

\begin{proof}[Proof of \cref{thm:Komlos-for-frameworks}]
In addition to the constants from the statement, we introduce $\eps,d>0$ and $r_0,r_1 \in \NATS$ such that
\begin{align*}
	1/n \ll \rho \ll 1/r_1 \ll 1/r_0 \ll  \eps  \ll  d \ll \rho' \ll\mu, 1/k,1/a,1/b.
\end{align*}
Let $B = B_\chi$.
By \rf{lem:lemma-for-G-simple}, there is $r_0 \leq r \leq r_1$, a subgraph $G' \subset G$ with a balanced partition $\cV=\{V_i\}_{i \in [r]}$ of $V(G')$ and a graph $R$ on vertex set $[r]$ that has a homomorphism $B$-cover $R' \subset R$ with $\Delta(R') \leq (kb)^4$ such that~\ref{itm:lemma-for-G-simple-1}--\ref{itm:lemma-for-G-simple-3} hold.
{We set $m = |V_1|$ to be the (uniform) size of the clusters.}

Let $R^\ast$ be the $m$-blow-up of $R$, where each vertex $i$ of $R$ is replaced with a set $X_i$ of {order $m$} and the edges are replaced by complete bipartite graphs.
We define $\cX= \{X_i\}_{i \in [r]}$.
Our plan is to find an almost perfect $B$-tiling $H=H_1 \cup H_2 \subset R^\ast$, where $H_1$ accounts for a buffer and $H_2$ contains most of the vertices.
We then use the \nameref{lem:blow-up} to embed $H$ into~$G'$.

We start by constructing  $H_1\subset R^\ast$.
Let $\alpha = \frac{\rho'}{4(kb)^8}$.
Since $R' \subset R$ is a homomorphism $B$-cover, there is for every $i \in V(R)$, a homomorphism $\phi_i\in \Hom(B;R')$ with $ i \in \phi_i(B)$.
Let $H_1 = \bigcup_{i \in [r]} H_{1,i}$ where each $H_{1,i}$ consists of $\alpha m$ disjoint copies of $B$ whose vertices are placed in the clusters {of $R^\ast$} dictated by $\phi_i$.
{Since $\Delta(R') \leq (kb)^4$, $B$ has diameter $2$ {and each homomorphism might place up to $b$ vertices in a cluster}, it follows that $H_1$ has at most {$b(kb)^8\alpha m \leq (\rho'/2) m$} vertices in each cluster of $\cX$.}
We obtain a new balanced partition $\cX' = \{X_i'\}_{i \in [r]}$ with $X_i' \subset X_i$ by deleting from each $X_i$ exactly $(\rho'/2) m$ vertices, including all vertices of $H_1$.
Let $m'=|X_1'| = (1-\rho'/2)m$ denote the cluster sizes of $\cX'$.

Next, we find an almost perfect $B$-tiling $H_2$ of $R^\ast[\cX']$.
Since $R$ is a $(\chi,\rho;t,\ell)$-tiling framework,
there is a $(1-\rho)$-perfect fractional $B$-tiling $\omega\colon  \Hom(B;R) \to  {[0,1]}$.
Thus, for each $\phi \in \Hom(B;R)$, we may embed $\lfloor \omega (\phi) m' \rfloor$ disjoint copies of $B$ into the blown-up image of $\phi$ in  $R^\ast[\cX']$.
As there are certainly at most $r^{kb}$ homomorphisms from $B$ to $R$, we have that $\sum_{\phi\in\Hom(B;R)}\lfloor \omega (\phi) m' \rfloor \geq \sum_{\phi\in\Hom(B;R)} \omega (\phi) m'  -  r^{kb} \geq (1-\rho'/4) m'r$, where the last relation follows from our choice of parameters.
Thus, we have a $(1-\rho'/4)$-perfect $B$-tiling $H_2 \subset R^\ast[\cX']$.
Recall the choice of {$m'=(1-\rho'/2)m$,} the fact that $G'$ is an $(\eps,10d)$-approximation of $G$ and that $\eps$ and $d$ are much smaller than $\rho'$.
Together, these imply that $v(H_2) \geq (1- \rho'/4) m'r \geq (1-3\rho'/4)mr \geq (1-\rho')n$.

To finish, set $H = H_1 \cup H_2$, and note that $H \subset R^\ast$.
We can ensure that $\cX$ is a partition of $H$  by adding isolated dummy vertices to $H$ to obtain size compatibility.
(These dummy vertices are deleted again after embedding $H$ into $G$.)
Moreover, by setting $\tilde X_i = X_i \cap V(H_1)$, we obtain a family of subsets $\tcX=\{\tX_i\}_{i\in[r]}$ such that $\tcX$ is an $(\alpha,R')$-buffer for $(H,\cX)$.
Hence, we can apply \rf{lem:blow-up} with $G',{v(G'),} (kb)^4, \alpha ,1$ playing the role of $G,{n,}\Delta_{R'},\alpha,\kappa$ to obtain the desired embedding $\psi\colon V(H)\to V(G)$.
\end{proof}

Having proved \cref{thm:Komlos-for-frameworks}, we can upgrade \cref{lem:lemma-for-G-simple} so that
$R'$ becomes a perfect $B_\chi$-tiling instead of a homomorphism $B_\chi$-cover of bounded degree.

\begin{lemma}[Lemma for $G$]\label{lem:lemma-for-G-true-form}
Let $\ell,t \in \NATS$ and $1/n \ll \rho \ll 1/r_1 \ll 1/r_0 \ll \eps \ll d  \ll \mu, 1/k,1/a,1/b$ with  $\chi = (k-1) + a/b$.
Let $G$ be a $\mu$-robust $(\chi, \rho ;t,\ell)$-tiling framework on $n$ vertices.

Then there is $r_0 \leq r \leq r_1$, a subgraph $G' \subset G$ with a balanced partition $\cV$ of size $r$ of $V(G')$ and a graph $R$ on vertex set $[r]$ with a perfect $B_\chi$-tiling $R' \subset R$ such that
\begin{enumerate}[\upshape (G1)]
	\item \label{itm:lemma-for-G-cycle-approx} $G'$ is an $(\eps,10d)$-approximation of $G$,
	\item \label{itm:lemma-for-G-cycle-reg} $(G',\cV)$ is an $(\eps,d)$-regular $R$-partition, which is $(\eps,d)$-super-regular on~$R'$ and
	\item \label{itm:lemma-for-G-cycle-frame} $R$ is a $(\chi,\rho;t,\ell)$-tiling framework.
\end{enumerate}
\end{lemma}

The proof of \cref{lem:lemma-for-G-true-form} follows almost line by line the proof of \cref{lem:lemma-for-G-simple}.
The only difference is that we apply \cref{thm:Komlos-for-frameworks} instead of \cref{lem:bounded-degree-cover} to obtain the $B_\chi$-tiling~$R'$.
We omit the details.

\section{Mixed Koml\'os' theorem for frameworks}\label{sec:mixed-komlos-framework}
In the following, we extend \cref{thm:Komlos-for-frameworks} to  mixed tilings with components of order linear in the order of $G$.

\begin{theorem}[Koml\'os' theorem for frameworks]\label{thm:mixed-Komlos-for-frameworks}
Let $\ell, t \in \NATS$ and $1/n \ll \xi \ll \rho \ll \rho',   1/k,1/a,1/b ,\mu,1/\Delta$ with $\chi = (k-1) + a/b$.
Let $G$ be a $\mu$-robust $(\chi,\rho;t,\ell)$-tiling framework on $n$ vertices and $H \in \cB((1-\rho')n,\xi n,\Delta)  \cap \cJ(\chi)$.
Then $H\subseteq G$.
\end{theorem}

We require the following allocation result.
Recall that the $m$-blow-up $R^\ast$ of a graph $R$ is obtained by replacing each vertex of $R$ with a set of order $m$ and the edges with complete bipartite graphs.

\begin{lemma}[Lemma for $H$ -- simple form]\label{lem:lemma-for-H-for-mixed-Komlos}
Let $1/n \ll \xi \ll 1/ r\ll\rho, 1/k, 1/a,1/b,1/\Delta$ with $\chi = (k-1) + a/b$.
Suppose $R$ is a $B_\chi$-tiling on $r$ vertices, $R^\ast$ is the $m$-blow-up of $R$ on $n = mr$ vertices,
and $H \in \cB((1- \rho )n,\xi n,\Delta)  \cap \cJ(\chi)$.
Then $H \subseteq R^\ast$.
\end{lemma}

{To prove \cref{lem:lemma-for-H-for-mixed-Komlos}, we} need the following simple proposition on partitioning vectors.
The proof is a standard application of concentration bounds, and hence we omit it.
\begin{proposition}\label{prop:balanced_allocation}
{Let $1/n,\xi \ll 1/k,1/s$.}
Consider $\vecb x_1 \dots , \vecb x_p \in \NATSZ^k$ and $\vecb  x = {\sum_{i \in [p]} \vecb x_i}$ such that $\onenorm{\vecb x}=n$ and $\maxnorm{ \vecb{x}_i}  \leq   \xi n$ for $i\in [p]$.
Then there is a partition $S_1,\dots,S_s$ of $\{\vecb x_1,\dots,\vecb x_p\}$ such that $\maxnorm{\sum_{\vecb x_i \in S_j}\vecb x_i - \vecb x /s} \leq n /s^2  $ for all $j \in [s]$.
\end{proposition}
\COMMENT{
We will need the following concentration bound.
\begin{lemma}[Hoeffding's Inequality]\label{lem:hoeffding}
	Suppose $X_1,\dots,X_p$ are independent random variables with $X_i$ taking a value in $[0 ,c_i]\subset \REALS$. Let $X=\sum_{i\in [p]} X_i$.
	Then
	\[
	\Prob[|X-\Exp[X]|>t]\leq 2\exp\left[-\frac{2 t^2}{\sum_{i\in [p]} c_i^2}\right].
	\]
\end{lemma}
\begin{proof}
	{We assign each $\vecb x_i$ independently and uniformly at random to one of the parts $S_1,\dots,S_s$.
		Fix an $S_j$ and an index $q \in [k]$.
		Let $X_i$ be the random variable whose value is $\vecb x_i(q)$ if $\vecb x_i$ is assigned to $S_j$.
		Then  $X = \sum_i X_i$ satisfies the conditions of \cref{lem:hoeffding} with $c_i := \vecb x_i(q)$.
		Note that by convexity, $\sum_{i\in [p]} c_i^2 \leq \tfrac{n}{\xi n} (\xi n)^2 =\xi n^2$.
		Hence
		\begin{align*}
			\Prob \left[\left|\sum_{\vecb x_i \in S_j}\vecb x_i(q) -   \vecb x (q) /s \right| > n/s \right] & = \Prob[|X - \exp[X]| > n/s ]
			\\ & \leq 2\exp\left[-\frac{2 (n/s )^2}{\xi n^2}\right]
			\leq 2\exp\left[-\frac{2}{s\xi}\right].
		\end{align*}
		Union bounding over all $S_j$'s and $q$ then shows that the partition fails to satisfy the condition with probability at most $ks  2\exp\left[-2/(s\xi)\right] < 1$.}
\end{proof}
}

\begin{proof}[Proof of \cref{lem:lemma-for-H-for-mixed-Komlos}]
The rough trajectory of the proof is as follows.
We find a proper colouring of $H$ that has approximately the same proportions as a proper colouring of $B:=B_\chi$.
We then apply \cref{prop:balanced_allocation} to partition the components of $H$ into $s:= \frac{r}{v(B)}$ subtilings, each of which fits into a blow-up of $B$.

Let $R$ be a $B$-tiling with vertex set $[r]$, and let $R^\ast$ be the $m$-blow-up of $R$ on  $n = mr$ vertices.
Let  $H \in \cB((1- \rho )n,\xi n,\Delta)  \cap \cJ(\chi)$ with $v(H) \geq n/2$, where the latter can be ensured by adding isolated vertices.
Our goal is to show that $H \subset R^\ast$.

Let $\phi$ be a proper $k$-colouring of $H$ that minimises the number of vertices of colour $1$.
So in particular $|\phi^{-1}(1)|/v(H) \leq \alpha(B)$ as $\crit(H) \leq (k-1) + a/b = \crit(B)$.
Denote the components of $H$ by $F_1,\dots, F_p$, let $\vecb x_q = \ord(\restr{\phi}{V(F_q)})$ for $q\in [p]$ and $\vecb x = \sum_{q\in[p]}\vecb x_q =  \ord (\phi)$.
Note that $\onenorm{\vecb x}=v(H)$ and $\maxnorm{\vecb x_q}\leq \xi n $ for $q\in[p]$.
Next, we find a second proper $k$-colouring $\phi'$ of $H$ with proportions similar to a proper $k$-colouring of $B$.
\begin{claim}\label{cla:komlos-permutation}
	There is a proper colouring  $\phi'$ of $H$
	such that $\vecb x' :=   {\ord} (\phi')$ satisfies
	\begin{enumerate}[\upshape (1)]
		\item \label{itm:komlos-permutation-a}
		$\vecb x'(1) = \alpha(B)\onenorm{\vecb x'} \pm \xi n$ and
		\item \label{itm:komlos-permutation-b}
		$\vecb x'(i) = \vecb x'(j)  \pm \xi n$  for all $i,j \in [k] \sm \{1\}$.
	\end{enumerate}
\end{claim}
\begin{proofclaim}
	We find $\phi'$ by permuting the colour classes (under $\phi$) of the components $F_q$ of $H$.
	
	We first ensure that $\vecb x'(1)$ has the correct order and then balance $\vecb x'(2),\dots,\vecb x'(k)$.
	Note that $\alpha(B) \leq 1/k$.
	Observe that as long as $\vecb x(1) < \onenorm{\vecb x}/k $, there exists $q \in [p]$ such that $\vecb x_q(1) < \vecb x_q(j) \leq \vecb x_q(1)+\xi n$ for some $j\in [k]\sm \{1\}$.
	Thus by repeatedly finding such components~{$F_q$} and permuting the colour classes $1$ and $j$, we can obtain a new proper colouring  $\tilde \phi$ for which $\vecb {\tilde x} :=  {\ord}(\tilde \phi)$ satisfies $\vecb {\tilde x}(1) = \alpha(B)\onenorm{\vecb {\tilde x}} \pm \xi n$.
	
	Next, we obtain a proper colouring $\phi'$ from $\tilde \phi$ by reorganising the colour classes $2,\dots,k$.
	More precisely, choose colour classes $2,\dots,k$ of $\phi'$ such that $\max_{i,j\in[k]\sm \{1\}}|\vecb x'(i) - \vecb x'(j)|= \max_{i\in[k]\sm \{1\}}\vecb x'(i)-\min_{j\in[k]\sm \{1\}}\vecb x'(j)$ is minimised, among all proper colourings $\phi'$ in which~$\restr{\phi'}{V(F_q)}$ is a permutation the colour classes $2,\dots, k$ of ${\tilde \phi} {\big|}_{V(F_q)}$ for each $q\in [p]$.
	Let $\vecb x'_q=\ord(\restr{\phi'}{V(F_q)})$ for $q\in[p]$.
	It follows that $\phi'$ has the desired properties.\COMMENT{If there exist $i,j\in[k]\sm \{1\}$ such that
		$\vecb x'(i) - \vecb x'(j) > \xi n$, then there exists a component $F_q$ with $q\in[p]$ such that
		$\vecb x'_q(j) +\xi n \geq \vecb x'_q(i) > \vecb x'_q(j)$.
		Permuting colour classes $i$ and $j$ on the component $F_q$ reduces $\vecb x'(i) - \vecb x'(j)$.
		Note that $\max_{i,j\in[k]\sm \{1\}}|\vecb x'(i) - \vecb x'(j)|= \max_{i\in[k]\sm \{1\}}\vecb x'(i)-\min_{j\in[k]\sm \{1\}}\vecb x'(j)$ and so such a permutation as described above can not increase $\max_{i,j\in[k]\sm \{1\}}|\vecb x'(i) - \vecb x'(j)|$, while at most $k/2$ such permutations are required to decrease it.
		Thus by the minimality of $\phi'$, we have that
		$\vecb x'(i) = \vecb x'(j)  \pm \xi n$  for all $i,j \in [k] \sm \{1\}$
		as desired.}
\end{proofclaim}

Suppose that $R$ consists of copies of $B$ called $B_1,\dots, B_s$.
Note that  $s = r/v(B) \geq 1/\rho^3$.
We now apply \cref{prop:balanced_allocation} to $\vecb x'_1,\dots,\vecb x'_p$ with $v(H) \geq n/2$ playing the role of $n$.
Thus we obtain a partition of $H$ into subtilings $S_1,\dots,S_s$ such that
${\maxnorm{\sum_{\vecb x'_q \in S_j}\vecb x'_q - \vecb x' /s}} \leq v(H) /s^2  $ for each $j\in[s]$.
By definition of $\phi'$, we have $\onenorm{\vecb x'} = \onenorm{\vecb x} = v(H)$.
Moreover,  $v(H) \leq (1-\rho)v(R^\ast)$ where $R^\ast$ is the $m$-blow-up of $R$ on  $n = mr$ vertices.
Putting this together, we obtain that for each $j\in[s]$,
\begin{align}\label{equ:lemma-for-H-for-mixed-Komlos-1}
	{\sum_{\vecb x'_q \in S_j}\vecb x'_q(1)} & \leq \frac{\vecb x'(1)}{s} + \frac{v(H)}{s^2}
	\leq\left(\alpha(B)v(H) + \xi n + \frac{v(H)}{s}\right)   \frac{1}{s}\nonumber
	\\&\leq \left(1 + \frac{\xi n}{\alpha(B)v(H)} + \frac{1}{\alpha(B) s} \right) \frac{ \alpha(B) v(H)}{s} \nonumber
	\\&\leq \left(1 + 2\rho^2 \right)   \frac{ \alpha(B) v(H)}{s} \leq \alpha(B) \frac{v(R^\ast)}{s},
\end{align}
where we used \cref{cla:komlos-permutation}\ref{itm:komlos-permutation-a} in the second inequality, {and in the penultimate inequality that $s \geq 1/\rho^3$, $1/n \ll 1/\xi \ll 1/k,1/a,1/b$ and $v(H) \geq n/2$, which in particular implies that $\alpha(B) = 1/(1+b(k-1)/a)$ is much smaller than $s \geq 1/\rho^{3}$.}
Analogously, by exploiting \cref{cla:komlos-permutation}\ref{itm:komlos-permutation-b}, for each $j\in[s]$ and $i\in[k]\sm \{1\}$, we obtain
\begin{align}\label{equ:lemma-for-H-for-mixed-Komlos-2}
	{\sum_{\vecb x'_q \in S_j}\vecb x'_q(i)}
	& \leq  	{\frac{\vecb x'(i)}{s}} + \frac{v(H)}{s^2}
	\leq  \frac{1}{s}\left(\frac{v(H)- |\phi'^{-1}(1)|}{k-1} + (k-1)\xi n + \frac{v(H)}{s}\right)  \nonumber
	\\&\leq \left(1 + 2\rho^2 \right) \frac{1-\alpha(B)}{k-1} \frac{v(H)}{s} \leq \frac{1-\alpha(B)}{k-1} \frac{v(R^\ast)}{s}.
\end{align}

Fix an arbitrary $j \in [s]$, and denote by $H_j$ the union of the components of $H$ that belong to $S_j$.
Recall that $R$ is a $B$-tiling on vertex set $[r]$ and $R^\ast$ is the $m$-blow-up of $R$ on  ${n =mr}$ vertices.
To finish, it suffices to show that $H_j \subset B_j^\ast$ where $B_j^\ast$ is the $m$-blow-up of $B_j$.
The smallest part of each $B_j^\ast$ contains $\alpha(B)v(R^\ast)/s$ vertices and the other parts each contain $(1-\alpha(B))v(R^\ast)/((k-1)s)$ vertices.
Thus by~\eqref{equ:lemma-for-H-for-mixed-Komlos-1},~\eqref{equ:lemma-for-H-for-mixed-Komlos-2} and as $B_j^\ast$ is a blown-up clique, we can embed $H_j$ into $B_j^\ast$ by assigning the $i\tth$ colour class of $H_j$ under $\phi'$ to the $i\tth$ part of $B_j^\ast$.
Hence $H \subset R^\ast$ as desired.
\end{proof}

\begin{proof}[Proof of \cref{thm:mixed-Komlos-for-frameworks}]
Choose $\alpha,d,\eps, r_0,r_1$ such that
\begin{align*}
	1/n \ll  \xi \ll \rho \ll 1/r_1 \ll 1/r_0 \ll  \eps   \ll d \ll \alpha \ll \rho', 1/k,1/a,1/b,\mu,1/\Delta.
\end{align*}
Let $n' = (1-\rho')n$ and $w \leq \xi n$.
Let $G$ be a $\mu$-robust $(\chi,\rho;t,\ell)$-tiling framework on $n$ vertices and $H \in \cB(n',w,\Delta)  \cap \cJ(\chi)$.

By \rf{lem:lemma-for-G-true-form}, there is $r_0 \leq r \leq r_1$, a subgraph $G' \subset G$ with a balanced partition $\cV=\{V_i\}_{i \in [r]}$ of $V(G')$ and a graph $R$ on vertex set $[r]$ with a perfect  $B_\chi$-tiling $R' \subset R$ such that
\begin{enumerate}[\upshape (G1)]
	\item  $G'$ is an $(\eps,10d)$-approximation of $G$ and
	\item   $(G',\cV)$ is an $(\eps,d)$-regular $R$-partition, which is $(\eps,d)$-super-regular on~$R'$.
\end{enumerate}
Let $m = |V_1|$ denote the common cluster size.
Let $R^\ast$ be the $m$-blow-up of $R'$, where each vertex $i \in V(R')$ is replaced with a set $X_i$ of order $m$.
Note that $v(R^\ast) = v(G')  \geq (1-2\sqrt\eps) n \geq (1-\rho'/2)n'$.

By \rf{lem:lemma-for-H-for-mixed-Komlos} applied with $v(R^\ast),\rho'/2$ playing the role of $n,\rho$, it follows that $H \subset R^\ast$.
By adding additional isolated vertices to $H$, we can assume that $V(H) = V(R^\ast)$.
Hence, we have an $R$-partition $(H,\cX)$ where $\cX=\{X_i\}_{i\in[r]}$.
It is furthermore clear that  $\tcX=\{X_i\}_{i\in[r]}$ are subsets of $V(H)$ such that $\tcX$ is an $(\alpha,R')$-buffer for $(H,\cX)$.
Hence, we can apply \rf{lem:blow-up} to $G'$ and $H$ with $\Delta_{R'} \leq a+(k-1)b$, $\eps,d,\alpha$ and $\kappa = 1$ to obtain an embedding $\psi\colon V(H)\to V(G)$ as desired.
\end{proof}

\section{Proof of \rf{thm:frameworks}}\label{sec:proof-framework-theorem}
In this section, we prove \rf{thm:frameworks}, which is the core result of this paper, assuming two lemmas that we state shortly.
The proof has four steps.
Given graphs~$G$ and $H$, we use \rf{lem:lemma-for-G-true-form} to obtain an approximation~$G'$ of $G$ together with a regular partition $\cV$ and a reduced graph $R$.
We then apply \cref{lem:lemma-for-Exceptional-Vertices} to embed {a part of~$H$ to} the vertices outside of $G'$ using only few vertices of every cluster of $\cV$.
Afterwards, we allocate what remains of $H$ to what remains of $G$ with \rf{lem:lemma-for-H}.
To finish, we convert this allocation into an embedding with the help of  \rf{lem:blow-up}.

Let us now give the formal statements of the aforementioned lemmas.
\begin{lemma}[Lemma for exceptional vertices]\label{lem:lemma-for-Exceptional-Vertices}
	Let $t,\ell \in \NATS$ and $1/n \ll \xi \ll \rho \ll 1/r \ll {\eps \ll d} \ll \mu, 1/\chi,1/\Delta$ with $w\leq \xi n$.
	{Set $\gamma = \eps^{1/4}$.}
	Let $G$ be a $\mu$-robust $(\chi,\rho;t,\ell)$-tiling framework on $n$ vertices.
	Suppose there is a subgraph $G' \subset G$ with a {balanced} partition $\cV=\{V_i\}_{i \in [r]}$ of $V(G')$ and a graph $R$ on vertex set $[r]$ such that
	\begin{enumerate}[\upshape (G1)]
		\item \label{itm:lemma-exc-cycle-approx} $G'$ is an $(\eps,10d)$-approximation of $G$ and
		\item \label{itm:lemma-exc-cycle-reg} $(G',\cV)$ is an $(\eps,d)$-regular $R$-partition.
	\end{enumerate}
	Let $H \in \cB(n,w,\Delta) \cap \cJ(\chi)$ with $v(H) \geq n/2$.
	Then there is a subtiling $H' \subset H$ of order at most $\gamma n$  and an embedding $\psi\colon V(H') \to V(G)$ such that $V(G) \sm V(G') \subset \psi(V(H'))$ and $|V_i \cap \psi(V(H'))| \leq \gamma |V_i|$ for all $i \in [r]$.
\end{lemma}

{For graphs $F$ and $R$, a homomorphism $F$-cover $R'$ of $R$ is an \emph{$F$-cover} if all involved homomorphisms are injective.
	In other words, $R' \subset R$ is a spanning subgraph and every vertex in $R'$ is on a copy of $F$.}

\begin{lemma}[Lemma for $H$]\label{lem:lemma-for-H}
	Let $1/n \ll \xi \ll \rho   \ll 1/r \ll \alpha \ll \mu ,1/\chi, 1/\Delta$ and $\rho \ll 1/t,1/\ell$ with  $w \leq \xi n$, $w'\leq \sqrt\xi n$.
	Suppose that $\chi+\mu/2 \leq {\lceil\chi\rceil}$ and that $G$ is a $\mu$-robust $(\chi+\mu,\rho;t,\ell)$-tiling framework on $n$ vertices.
	Let $R$ be a  $(\chi, \rho; t,\ell)$-tiling framework on vertex set $[r]$ containing a {$K_{{\lceil\chi\rceil}}$-cover} $R' \subset R$ with $\Delta(R') \leq \Delta$.
	Let $(G,\cV)$ be a  $2$-balanced $R$-partition.
	Let $H \in \cH(n,w,\Delta; \chi; t,\ell,w')$.
	Then there is a size-compatible (with $\cV$) vertex partition $\cX=\{X_i\}_{i \in [r]}$ of~$H$ and $\tcX= \{\tilde{X}_i\}_{i \in [r]}$  with $\tilde{X}_i\subset X_i$ for each $i\in [r]$ such that $(H,\cX)$ is an $R$-partition and {moreover} $\tcX$ is an $(\alpha,R')$-buffer for $(H,\cX)$.
\end{lemma}

The proofs of \cref{lem:lemma-for-Exceptional-Vertices,lem:lemma-for-H} can be found in \cref{sec:lemma-for-Exceptional-Vertices,sec:lemma-for-H}, respectively.
For now, we prove  \cref{thm:frameworks} assuming {the} truthfulness of these lemmas.

\begin{proof}[Proof of \cref{thm:frameworks}]
	{Without loss of generality, we may assume that $\mu \leq 1/4$.
		(By \cref{obs:framework-monotone}, the result only gets stronger for smaller $\mu$.)}
	Given $\chi$ and $\mu$ with $\mu$, let $k = \lceil \chi \rceil$.
	In the following, we assume that $\chi + \mu/2 \leq k$.
	The case where $\chi + \mu/2  > k$ is discussed at the end of this proof.
	Given this, we can assume that {$\chi + \mu/2 \geq (k-1) + a/b$} for (coprime) natural numbers {$a \leq b \leq 8/\mu$}.\footnote{{Let $4/\mu\leq  b:= 4 \lceil \mu^{-1}  \rceil \leq 8/\mu$ and choose $a \in \NATS$ maximal such that $k-1 + a/b \leq \chi +\mu/2$.}}
	We now choose $\alpha,\gamma,d,\eps,r_0, r_1,\rho,\xi,n$ such that
	\begin{align*}
		1/n \ll \xi \ll \rho  \ll 1/r_1 \ll 1/r_0 \ll {\eps \ll  d}   \ll \alpha  \ll \mu,1/k,1/a,1/b,1/\Delta.
	\end{align*}
	Set $t'= \min\{r_1^k,t\}$, $\ell'= \min\{r_1^k,\ell\}$ {and $\gamma = \eps^{1/4}$.}
	Now suppose we are given a $\mu$-robust $(\chi+\mu,\rho;t,\ell)$-tiling framework $G$ on $n$ vertices and $H \in \cH(n,w,\Delta; \chi; t,\ell,w')$ where $w\leq \xi n$ and $w'\leq \sqrt\xi n$.
	{We assume that $H$ has $n$ vertices, adding isolated vertices if necessary.}
	
	\medskip
	\noindent \emph{Step 1: Obtaining a regular partition.}
	By \cref{obs:framework-monotone}, $G$ is also a $\mu/2$-robust $(\chi+\mu/2,\rho;t,\ell)$-tiling framework.
	We apply \cref{lem:lemma-for-G-true-form} to obtain, for some $r_0 \leq r \leq r_1$, a subgraph $G' \subset G$ with a balanced partition $\cV=\{V_i\}_{i \in [r]}$ of $V(G')$, a graph $R$ on vertex set $[r]$ with a perfect $B_{{\chi + \mu/2}}$-tiling $R' \subset R$ such that
	\begin{enumerate}[\upshape (G1)]
		\item \label{itm:proof-lemma-for-G-cycle-approx} $G'$ is an $(\eps,10d)$-approximation of $G$,
		\item \label{itm:proof-lemma-for-G-cycle-reg} $(G',\cV)$ is an $(\eps,d)$-regular $R$-partition, $(\eps,d)$-super-regular on~$R'$ and
		\item \label{itm:proof-lemma-for-G-cycle-frame} $R$ is a $(\chi+\mu/2,\rho;t,\ell)$-tiling framework.
	\end{enumerate}
	Note that $R$ is also a $(\chi,\rho;t',\ell')$-tiling framework, since $K_k(R)$ has (trivially) at most $e(K_k(R)) \leq r_1^k$ tight or loose components.
	
	\medskip
	\noindent \emph{Step 2: Handling the exceptional vertices.}
	Since $H \in \cH(n,w,\Delta; \chi; t,\ell,w')$, we can select a subtiling $W \subset H$ of order at most $(t'+1)\sqrt{\xi} n$ that contains, disjointly, $t'$ proper $(k,kw)$-\fc{} subtilings and one topological $(\ell',w)$-\fc{} subtiling.
	Thus $H-V(W) \in \cB(n, w, \Delta)$ and $v(H-V(W)) \geq n/2$.
	Moreover, $\crit(H-V(W)) \leq \chi + {\mu/8}$ by the choice of $n,\xi,\mu$, and hence $H-V(W)\in \cJ(\chi+\mu/8)$.
	Next, we apply \cref{lem:lemma-for-Exceptional-Vertices} to obtain a subtiling $H' \subset H-V(W)$ of order at most $\gamma n$ and an embedding $\psi\colon V(H') \to V(G)$ such that $V(G) \sm V(G') \subset \psi(V(H'))$ and $|V_i \cap \psi(V(H'))| \leq \gamma|V_i|$ for all $i \in [r]$.
	It follows that $\cV''=\{V_i''\}_{i \in [r]}$ is $2$-balanced where  $V_i'':= V_i \sm \psi(V(H'))$.
	Let $G'' = G' \sm \psi(V(H'))$ and $n'' = v(G'')$.
	Note that $n'' \geq (1 - \gamma  - \eps) n$  by~\ref{itm:proof-lemma-for-G-cycle-approx}.
	So $G''$ is a $(\mu/4,\mu/4)$-approximation of $G$.
	Consequently,~$G''$ is a $\mu/2$-robust $(\chi + \mu,\rho;t,\ell)$-tiling framework.
	Since each cluster $V_i$ has lost a proportion of at most $\gamma$ vertices, it follows that $(G'',\cV'')$ is a $({\eps^{1/8}},d^2)$-regular $R$-partition that is $(\sqrt{\eps},d^2)$-super-regular on $R'$ by \cref{proposition:robust-regular}.
	Let $H''= H - V(H')$.
	We have $\crit(H'') \leq \chi + \mu/4$.
	Since $\chi + {\mu/4} \leq k$ and $W \subset H''$, it follows that $H''$ is also in $\cW(t',\ell',\chi+\mu/4,w,{w'})$.
	So together, we have $H'' \in \cH(n'',w,\Delta; \chi+\mu/4; t',\ell',w')$.

	\medskip
	\noindent \emph{Step 3: Allocating $H''$.}
	It remains to embed $H''$ into $G''$.
	To do this, we begin with an allocation of $H''$ into $R$.
	We apply \cref{lem:lemma-for-H} to $(G'',\cV'')$ with $n''$, $\chi+\mu/4$, $\mu/4$, $\max (\Delta,\Delta(R'))$, $t'$, $\ell'$ playing the role of $n$, $\chi$, $\mu$, $\Delta$, $t$, $\ell$ and the other parameters as they are.
	It follows that there is a size-compatible (with $\cV''$) vertex partition $\cX=\{X_i\}_{i \in [r]}$ of $H''$ and $\tcX= \{\tilde{X}_i\}_{i \in [r]}$  with $\tilde{X}_i\subset X_i$ for each $i\in [r]$ such that $(H'',\cX)$ is an $R$-partition and $\tcX$ is an $(\alpha,R')$-buffer for $(H'',\cX)$.
	
	\medskip
	\noindent \emph{Step 4: Embedding $H''$.}
	We finish by embedding $H''$ into $G''$ using the  \nameref{lem:blow-up}.
	More precisely, we apply \cref{lem:blow-up} to $(H'',\cX)$, $(G'',\cV'')$ with ${\eps^{1/8}},d^2,2, {\Delta(R')},v(G'')$ playing the role of $\eps,d, \kappa, \Delta_{R'},n$ and $\Delta,\alpha,r$ as themselves.
	This yields  an embedding $\psi'\colon V(H'')\to V(G'')$.
	As $H$ is the disjoint union of $H'$ and $H''$ and moreover $\psi(V(H'))$ and $V(G'')$ form a partition of $V(G)$,  we can combine $\psi$ and $\psi'$ to obtain an embedding of $H$ into $G$.
	This completes the proof.
	
	\medskip
	\noindent \emph{Case: $\chi \approx k$.}
	It remains to discuss the case when $\chi + {\mu}/2 > k$.
	Here, we replace $(\chi,{\mu})$ by $(\chi+{\mu}/2,{\mu}/2)$.
	Since {$\mu \leq 1/4$}, we have $\lceil \chi + {\mu}/2 \rceil = k+1$.
	{Moreover,  $(\chi + {\mu}/2) - {\mu}/4 \leq k+1$.}
	{Let $t'$ be defined} as above (with respect to the updated parameters).
	Now consider $H \in \cH(n,w,\Delta; \chi; t,\ell,w')$ as in the lemma input.
	{We can assume that $H$ has $n$ vertices, adding isolated vertices otherwise.}
	{Note that, trivially, $H \in \cJ(\chi +{\mu}/2)$.}
	
	{We claim that $H$ contains $t'$ proper $(k+1, (k+1) w)$-\fc{} subtilings.
		To see this, pick a subtiling $H^\ast \subset H$ on $\sqrt{\xi} n \leq n'   \leq \sqrt{\xi} n + w$ vertices.
		Since $w \leq \sqrt{\xi} n'$, we can apply \cref{lem:k+1_wildcard} to find that $H^\ast$ is properly $(k+1, n'/(k+1)^5)$-\fc{}.
		Since $(k+1)w \leq n'/(k+1)^5$, this implies that $H^\ast$ is properly $(k+1, (k+1) w)$-\fc{}.
		We repeat this process $t'-1$ further times using fresh subtilings of $H$.
		This shows that $H$ contains $t'$ proper $(k+1, (k+1) w)$-\fc{} subtilings.
		Moreover, since $H \in \cW(t,\ell,\chi,w,w')$ it follows that $H$ contains a {topological $(\ell,w)$-\fc{} subtiling} of order at most $w'$.
		(We can assume that all these subtilings are pairwise disjoint.)}
	Together, this gives $H \in \cH(n,w,\Delta; \chi+{\mu}/2; t',{\ell},w')$.
	We then continue with the rest of the proof as before.
\end{proof}

\section{Proof of \cref{lem:lemma-for-Exceptional-Vertices}} \label{sec:lemma-for-Exceptional-Vertices}

We devote this section to the proof of Lemma~\ref{lem:lemma-for-Exceptional-Vertices}, which is used for embedding the exceptional vertices in the proof of \cref{thm:frameworks}.

\begin{proof}[Proof of \cref{lem:lemma-for-Exceptional-Vertices}]
	We introduce $\mu' $ such that
	$$1/n \ll \xi \ll \rho \ll 1/r \ll {\eps \ll d}  \ll  \mu'\ll\mu, 1/\chi,1/\Delta.$$
	Given the statement, let $m=|V_1|$ be the common cluster size.
	Let $L = V(G) \sm V(G')$ be the set of leftover vertices, and let $\cK$ be the set of {vertex-ordered} $k$-cliques in $R$.
	{(So each element of $\cK$ is a clique whose vertices are ordered.)}
	
	\begin{claim}
		There is a partition $\cL = \{L_K\}_{K \in \cK}$ of $L$ with parts of size at most $\sqrt{\eps} m/r^{k-1}$ such that for each $K \in \cK$, every vertex in $L_K$ has at least $2d m$ neighbours in $V_i$ where $i$ ranges over the first $k-1$ vertices of $K$.
	\end{claim}
	\begin{proofclaim}
		Let $W$ denote the bipartite graph with colour classes $L$ and $\cK$.
		We add an edge between $v \in L$ and $K \in \cK$ if $v$ has at least $2d m$ neighbours in each $V_{i}$ where $i$ ranges over the first $k-1$ vertices of $K$.
		Recall that $G$ is a $\mu$-robust $(\chi,\rho;t,\ell)$-tiling framework.
		By \cref{prop:supersaturated-linkage}, every vertex in $G$ is linked to $\mu' n^k$ edges in $K_k(G)$.
		By \ref{itm:lemma-exc-cycle-approx} and $\eps \ll d \ll \mu'$, it follows that every vertex in $G$ is linked to $(\mu'/2) n^k$ edges in $K_k(G')$.
		{This together with~\ref{itm:lemma-exc-cycle-reg} implies} that each vertex in $L$ has at least $(\mu'/4) r^k$ neighbours in $\cK$ {in the bipartite auxiliary graph~$W$.}
		{Indeed, the clusters of every $k$-clique in $R$ can host at most $m^k$ $k$-cliques of $G'$. 
			Moreover, whenever there is a $k$-clique in $G'$, there is also one in $R$ in the corresponding clusters, since $(G',\cV)$ is an $R$-partition.
			So the claim follows from the pigeonhole principle together with the constant hierarchy.}
		We obtain a bipartite graph~$W'$ from $W$ by blowing up each vertex $K \in \cK$ to a set $K'$ of size $\sqrt{\eps} m/r^{k-1}$.
		Let $\cK' = \bigcup K'$.
		So vertices in $L$ have at least $(\mu'/4) \sqrt{\eps} mr \geq \eps n$ neighbours in $W'$.
		Since $|L| \leq \eps n$ by~\ref{itm:lemma-exc-cycle-approx}, we may greedily find a matching in $W'$ between $L$ and $\cK'$ covering $L$.
		We then define~$L_K$ as the vertices matched into $K'$ for each $K \in \cK$.
	\end{proofclaim}
	
	Now we show that one can individually cover the vertices of each part $L_K \in \cL$ by embedding a suitable subtiling $H_K \subset H$ into the clusters  corresponding to $K$.
	To this end, fix $L_K \in \cL$.
	Without loss of generality, let $V(K)=[k]$.
	Let $\cW = \{W_i\}_{i \in [k]}$ where $W_i=V_i$.
	We apply \cref{pro:super-regularising-R'} with $G'[\cW]$ and $k$ playing the role of $G$ and $\DeltaRp$.
	This gives a family of disjoint sets $\cW' = \{W_i'\}_{i \in [k]}$ with balanced clusters $W_i' \subset W_i$ of size $m' =\lfloor (1 - {\sqrt{\eps}})m \rfloor$ such that $(G'[\cW'],\cW')$ is a $({2\eps,d/2})$-super-regular $K$-partition.
	{Let  $\cW'' = \{W_i''\}_{i \in [k]}$ be obtained by deleting $|L_K|$ arbitrary vertices from $W'_k$ and then adding the vertices of $L_K$.}
	Let $G''$ be obtained from $G'[\cW'']$ by adding the edges of $G$ between $L_K$ and the clusters $W_1'',\dots,W_{k-1}''$.
	Note that, $(G'',\cW'')$ is  a balanced $(\sqrt{\eps} , d/4)$-{super}-regular $K$-partition by \cref{proposition:robust-regular}.
	
	Next, we select a subtiling $H_K \subset H$ of order $ q/2  \leq v(H_K) \leq q$ where  $q := (\eps^{1/4}/(k r^{k-1} )) m$.
	(Note that $w\leq \xi n \ll q$.)
	By assumption, $H_K$ has a $k$-colouring, whose colour classes we denote by $C_1,\dots,C_k$.
	Since $\eps  \ll 1/\Delta,1/k$, we can assume that $|C_k| \geq \Delta^2 \sqrt{\eps} m/r^{k-1}$.
	So  we can greedily find a set $C_k' \subset C_k$ of size $|C_k'| \geq \sqrt{\eps} m/r^{k-1}$ such that vertices in $C_k'$ are at distance $3$ in $H_K$ (meaning that their neighbourhoods do not intersect).
	We greedily embed the vertices of $C_k$ into $W''_k$ covering the vertices of $L_k$ with $C_k'$.
	Let $\psi$ denote this partial embedding.
	It remains to extend $\psi$ to a full embedding of~$H_K$, where $C_1,\dots,C_{k-1}$ are embedded into $W_1'',\dots,W_{k-1}''$, respectively.
	{This can be done by applying \cref{lem:simple-blow-up} with $\zeta = \Delta \sqrt{\eps}/r^{k-1}$ and $c=d/4-\sqrt{\eps}$ to $(G'',\cW'')$.
	The image restricting sets are defined as follows. 
	For each for every $y \in C_k'$, $x \in N_{H}(y)$ and $i \in [k]$, we set $I_x = N(\phi(y))\cap W''_i$.
	Note that as $\Delta(H) \leq \Delta$ the number of restricted vertices is at most $\zeta m$ for each cluster, while each image restricting set $I_x$ has size at least $cm$ due to super-regularity and the fact that the neighbourhoods of the vertices of $C_k'$ do not intersect.
	Hence the conditions of \cref{lem:simple-blow-up} are satisfied and the desired extension of $\psi$ exists.}
	
	Finally, we argue that the $H_K$'s can be chosen to be pairwise vertex-disjoint.
	This is indeed possible by restricting the sets $\cW$ to the yet unused vertices, where we use that each $L_K$ requires at most {$(\eps^{1/4}/(kr^{k-1})) m  = \gamma m / (kr^{k-1})$} vertices of any particular cluster $V_i$ and each $i \in V(R)$ is on at most $kr^{k-1}$ ordered $k$-cliques of $R$.
	It follows that $H' = \bigcup_{K\in \cK} H_K$ has the desired properties.
	Since $\gamma m r\leq \gamma n$ this finishes the proof.
\end{proof}

\section{Allocations}\label{sec:allocations}
In the proof of \cref{lem:lemma-for-H}, we aim to embed a graph $H$ into a (mostly) balanced blow-up~$R^\ast$ of a graph $R$.
We can express this with a homomorphism $\eta\colon H\rightarrow R$ that sends the correct number of vertices of $H$ to every vertex of $R$.
In the following, we introduce some further vocabulary to formalise this.
For a family of graphs $\cF$, we write $\Hom(\cF;R)=\bigcup_{F\in\cF}\Hom(F;R)$.
Recall that $\cF=\cF(k,w)$ is the family of properly $k$-colourable graphs of order at most $w$.
Suppose that $H$ is an $\cF$-tiling and let $R' \subset R$.
Hence we can express a homomorphism $\eta\colon H\rightarrow R'$ by a vector $\vecb u\in \NATSZ^{\Hom(\cF;R)}$ where $\vecb u (\phi)$ counts the number of tiles $F$ of $H$ for whom $\eta $ restricted to $F$ is $\phi \in \Hom(F;R')$.
In this situation, we say that $\vecb u$ is an \emph{allocation of $H$ to $R'$} and the vector $\vecb u$ \emph{encodes} the homomorphism $\eta$.\footnote{Strictly speaking, the encoding is only up automorphisms that permute isomorphic tiles of $H$, but this is irrelevant for our argument.}

It will be useful to track how many vertices of a  tile $F \subset H$  are mapped to a particular vertex in $R$ by a particular homomorphism $\theta \in \Hom(F;R)$.
To this end, we define $A=A_{\cF,R}\in \NATSZ^{V(R) \times \Hom(\cF;R)}$ to be the \emph{{homomorphism} incidence matrix} by setting $A(v,\theta)=|\theta^{-1}(v)|$ for each $v\in V(R)$ and $ \theta\in \Hom(\cF;R)$.
Hence for a blow-up $R^\ast$ of $R$ with cluster sizes described by $\vecb x \in \NATSZ^{V(R)}$, an allocation $\vecb u$ \emph{encodes an embedding} of $H$ into $R^\ast$ provided that $A\vecb u \leq \vecb x$.

Recall that for sets $T \subset S$ and a vector $\vecb w \in \REALS^{S}$, we obtain the restricted vector $\restr{\vecb w}{T}$ by setting all entries that do not correspond to an element of $T$ to be zero.
Suppose $k\geq 2$ and let $J \subset K_k(R)$.
We say that the allocation $\vecb u$ is \emph{$(s,J)$-surjective} if we can write $\vecb u$ as the sum  of (non-negative) ${\vecb v^\ast,} \vecb v_{e,v} \in {\NATSZ}^{\Hom(\cF;R)}$ ranging over all $e \in E(J)$ and {all choices} $v \in e$ such that $A  \vecb v_{e,v} (v) \geq s$ and $A  \vecb v_{e,v} (w) = 0$ for all $w \notin e$.
{(The additional summand $\vecb v^\ast$ allows for some more flexibility in the use of this definition.)}
In practice, a surjective allocation will act as a reservoir of tiles that can be used for local adjustments.

To state the following results, we need the concept of reachability, which is inspired by the work of Lo and Markström~\cite{LM15}.
Let us recall the definition of tight components.
Recall from \cref{sec:host-graphs} that For a $k$-graph~$J$, the vertices of the graph  $T(J)$ are the edges of $J$ where $e,f \in E(J)$ are adjacent whenever $|e\cap f| = k-1$.
The {tight components} of~$J$ are the subgraphs $T \subset J$ whose edges form the connectivity components of $T(J)$.
Let $J$ be a $k$-graph and $C \subset J$ be a tight component.
We define $\cR(C)$ to be the graph on the same vertex set as $C$ where two vertices $x,y$ are adjacent if there are edges $e,f\in E(C)$ such that $x\in e \sm f$, $y\in f\sm e$ and $|f\cap e|=k-1$.
The \emph{reachability components} of $C$ are the components of $\cR(C)$.

\begin{proposition}\label{pro:reachabilitity}
	A tight component in a $k$-graph has at most $k$ reachability components.
	Moreover, each of its edges contains a vertex of each reachability component.
\end{proposition}
\begin{proof}
	Without loss of generality, we can assume that $J$ is a tight component, meaning that~$T(J)$ is connected.
	We proceed by induction on the number of edges of $J$.
	Let $J'$ be obtained by deleting the edge corresponding to a leaf $e$ of a spanning tree of $T(J)$.
	So $J'$ is still a tight component.
	If $J'$ is empty we are done as $e$ only has $k$ vertices, so assume otherwise.
	By induction, $J'$ has at most $k$ reachability components.
	Moreover, each edge of $J'$ contains a vertex of each reachability component of $J'$.
	By choice of $e$, there is an edge $f$ in $J'$ with $|e \cap f| = k-1$.
	It follows that the vertex of $e \sm f$ is in the same reachability component as the vertex of $f \sm e$.
	Since $f \sm e \subset V(J')$, it is clear that the vertex of $e\sm f$ belongs to one of the (at most $k$) existing tight components, and $e$ contains a vertex of each reachability component of~$T$.
\end{proof}

We also need the following simple fact.
For a graph $G$,
we denote by $\vec E(G)$ the set of directed edges of $G$, that is, $\{(u,v)\colon uv\in E(G)\}$.
A function $f \colon   \vec E(G) \to \INTS$ is a \emph{flow} if $f(u,v)=-f(v,u)$ for all $uv\in E(G)$ where we write $f(x,y)$ instead of $f((x,y))$.

\begin{proposition}\label{pro:weights_to_flows}
	Suppose $G$ is a connected graph and $\vecb w \in \INTS^{V(G)}$ satisfies $\sum_{v\in V(G)}\vecb w(v)=0$ and $\onenorm{\vecb w}\leq s$.
	Then there exists a flow $f \colon   \vec E(G) \to \INTS$ on $G$ satisfying $|f(u,v)|\leq s$ for all $uv\in \vec E(G)$ and $\sum_{u\in N(v)}f(u,v)=\vecb w(v)$ for all $v\in V(G)$.
\end{proposition}
\begin{proof}
	We construct the flow as follows.
	Consider a set of $\onenorm{\vecb w}$ positive and negative units assigned to $V(G)$.
	The vertex $x\in V(G)$ receives $|\vecb w(x)|$ units whose sign is determined by the sign of $-\vecb w(x)$.
	We match each negative unit to a positive unit, say at $x$ and $x'$ respectively, and then put a flow of one unit on each edge of some  (directed) path from $x$ to $x'$.
	This is possible because $G$ is connected.
	Next, we take the union of all these paths; that is, for each such path~$P$ and each (directed) edge $(y,y')$ in $P$, add $1$ to $(y,y')$ and subtract $1$ from $(y',y)$.
	Taking the sum of the units assigned at each edge gives rise to a flow $f\colon \vec E(G) \to \INTS$  such that $\sum_{u\in N(v)}f(u,v)=\vecb w(v)$ and $|f(u,v)|\leq s$ for all $uv\in E(G)$.
\end{proof}

In the following we prove for a graph $R$ and a tight component $T \in K_k(R)$ that surjective allocations to $R[V(T)]$ can be altered within a reachability component $C$ of~$T$.

\begin{proposition}\label{prop:flow}
	{Let $s,s',r,w,k \in \NATS$ with $k \geq 2$ and $r (rs+w) \leq s'/2$}.
	Let $R$ be a graph on $r$ vertices.
	Let $T$ be a tight component of $K_k(R)$, and let $C$ be a reachability component of $T$.
	Let $\vecb b \in  \mathbb{Z}^{V(R)}$ satisfy $\vecb b (v) = 0$ for all $v \in V(R) \sm V(C)$ as well as $\sum_{v\in V(C)}\vecb b(v)=0$ and $\maxnorm{\vecb b}  \leq s$.
	Let $\cF= {\cF{(k,w)}}$, $A = A_{\cF,R}$ and $H$ be an $\cF$-tiling.
	Let $\vecb u \in\NATSZ^{\Hom(\cF;R)}$ be an $(s',T)$-surjective allocation of $H$ to $R[V(T)]$.
	Then there exists an $(s'/2,T)$-surjective allocation $\vecb w \in \NATSZ^{\Hom(\cF;R)}$ of $H$ to $R[V(T)]$ with $A (\vecb u -\vecb w) =\vecb b$.
\end{proposition}
\begin{proof}
	Let us give a brief overview of the proof.
	We start by viewing $\vecb b$ as a description of sinks and sources, before translating it into an integer flow along the edges of $C \subset \cR(T)$.
	This gives us a local representation of $\vecb b$ and allows us to treat the flow at each edge of $C$ essentially independently.
	For each $xx'\in E(C)$, we find a small set of graphs, part of whom is allocated to $x$ or $x'$ (depending on the flow direction at $xx'$) by $\vecb u$.
	We alter the allocation of these graphs at $x$ and $x'$, `transferring' vertices between $x$ and $x'$ in order to mimic the flow at $xx'$.
	Repeating this for each edge, we mimic the entire flow and so the altered allocations; this gives the desired~$\vecb w \in \NATSZ^{\Hom(\cF;R)}$.

	As $C$ is connected we can apply \cref{pro:weights_to_flows} to find a flow  $f\colon \vec{E}(C) \rightarrow \INTS$ on $C$  satisfying ${\sum_{x' \in N_C(x)} f(x',x)} = \vecb{b}(x)$ and $|f(x,x')|\leq \maxnorm{\vecb b}v(R) \leq rs$ for all  $xx'\in E(C)$.
	So we now have a local representation of our vector $\vecb b$.

	Fix an edge $xx'\in E(C)$.
	We aim to find a set of tiles $S_{xx'}$ of $H$ that will be allocated differently in $\vecb u$ and $\vecb w$ so that $A (\vecb u -\vecb w)$ mimics $f(x,x')$.
	Suppose, without loss of generality, that $f(x,x') = c \ge 0$ (recall that $f(x,x') \leq rs$).
	By the definition of reachability components, there are $k$-cliques $K,K' \subset R$ such that $V(K),V(K') \in E(T)$ with $V(K)\sm V(K')=\{x\}$ and $V(K')\sm V(K) = \{x'\}$.
	Given any $F\in \cF$ and $\theta\in \Hom(F;K)$, we can find $\theta'\in \Hom(F; K\cup K')$ such that $\theta^{-1}(v)=\theta'^{-1}(v)$ for all $v\in V(K)\cap V(K')$ with $|\theta'^{-1}(x)|=|\theta^{-1}(x)|-p$ and ${|\theta'^{-1}(x')|=p}$ for each $0\leq p \leq |\theta^{-1}(x)|$.
	Indeed, $K\cup K'$ is a $(k+1)$-clique minus the edge $xx'$ and as $\theta^{-1}(x)$ is an independent set, it is clear that such a $\theta'$ exists.
	We choose $S_{xx'}$ to be a set of tiles $F$ in $H$ that are mapped to $R[V(K)]$ by
	some $\theta_F\in \Hom(\cF;K)$ and
	such that $c \leq \sum_{F\in S_{xx'}}|\theta_F^{-1}(x)| \leq c+ w$.
	This is possible as the allocation {$\vecb u = {\vecb u^\ast +} \sum_{e \in E(T),v \in e} \vecb u_{e,v}$} of $H$ to $R$ is ${(s',T)}$-surjective and each tile $F$ of $H$ has order at most $w$.
		{So in particular, we choose $F$ as a subset of the tiles allocated by $\vecb u_{V(K),x}$.}
	We choose $\theta_F' \in \Hom(F;K\cup K')$ for each $F\in S_{xx'}$ such that
	$\theta_F^{-1}(v)=\theta_F'^{-1}(v)$ for each $v\in V(K)\cap V(K')$,
	while $\sum_{ F \in  S_{xx'}}(| \theta_F^{-1}(x)|-|\theta'^{-1}_F(x)|)=c$ and
	$\sum_{F\in S_{xx'}}|\theta_F'^{-1}(x')|=c$.

	We perform this procedure for each edge $xx'\in E(C)$ in such a way that for all distinct $xx',yy'\in E(C)$, we have $S_{xx'}\cap S_{yy'}=\emptyset$.
	This is possible by ${(s',T)}$-surjectivity of $\vecb u$ and since {$s'\geq 2r (rs+w)$.}
	We then define our new allocation $\vecb w = \vecb u^\ast + { \vecb w^\ast} + \sum_{e \in E(R), v \in e} \vecb w_{e,v}$ of~$H$ to~$R$ as described above
	{where $\vecb w^\ast$ represents the allocation of all tiles that we changed and $\vecb w_{e,v}$ represents the allocation of all tiles represented by $\vecb u_{e,v}$ that are not changed}.
	The resulting $\vecb{w}$ satisfies
	\[
		\left(A (\vecb u-\vecb w)\right)(x) = \sum_{x'\in N_C(x)}f(x',x) = \vecb b(x)
	\]
	for all $x\in V(C)$, by construction.
	For each $e \in E(T)$ and $v \in e$, 
	the vector $\vecb w_{e,v}$ allocates at most $r({rs}+w) \leq s'/2$ fewer vertices to $v$ than $\vecb u_{e,v}$.
	Thus $\vecb w$ is ${(s'/2,T)}$-surjective.
\end{proof}

In \cref{prop:flow}, we require $\vecb b$ to be zero outside of the vertex set of the reachability component $C$ of the tight component $T$.
If $H$ contains a proper \fc{} subtiling, we can  weaken this assumption to being zero outside of $V(T)$.

For the following definition we extend the notion of topological colourings to hypergraphs. That is, a \emph{topological $k$-colouring of hypergraph $\cJ$} is a colouring  $\phi:V(\cJ)\rightarrow [k]$ of the vertices of $\cJ$ with the property that all edges of $\cJ$ are monochromatic.
It is worth noting that in a  topological colouring all loose components are monochromatic.
Let $\cF=\cF(k,w)$, $W$ be an $\cF$-tiling and $R, R'$ be graphs with $R' \subset R$.
Suppose $\vecb u \in  \NATSZ^{\Hom( \cF;R)}$ is an allocation of~$W$ to~$R'$.
We say that $\vecb u$ \emph{encodes a proper colouring} $\phi$ of $W$ if $\vecb u$ encodes a (restricted) homomorphism $\eta\in \Hom(W;R')$ such that there exists a proper colouring $\phi'$ of $R'$ with $\phi' \circ \eta = \phi$.

\begin{proposition}\label{lem:allocate-tight}
	{Let $s,s',r,w,k \in \NATS$ with $k\geq2 $ and $r (r^2s+w) \leq s'/2^k$}.
	Let $R$ be a graph on $r$ vertices, $T$ a tight component of $K_k(R)$ and $e_0\in E(T)$.
	Let $\cF = \cF(k,w)$ and $A = A_{\cF,R}$.
	Let $H$ be an $\cF$-tiling, and let $\vecb u,\vecb u_1,\vecb u_2\in \NATSZ^{\Hom( \cF;R)}$ be as follows. The vector $\vecb u=\vecb u_1+\vecb u_2$ is an allocation of $H$ to $R[V(T)]$.
	The vector $\vecb u_1$ is an $(s',T)$-surjective allocation of a subtiling $H'\subset H$ to $R[V(T)]$ and $\vecb u_2$ is an allocation of a proper $(k,{rs})$-\fc{} subtiling  $W=H - V(H')$ to $R[e_0]$ that encodes a proper central $k$-colouring of $W$.

	Suppose $\vecb c \in \INTS^{V(R)}$  satisfies $\vecb c (v) = 0$ for all $v \in V(R) \sm V(T)$ as well as $\maxnorm{\vecb c}\leq s$ and 
	$\sum_{v\in {V(R[V(T)])}}\vecb c(v)=0$. Then there exists an allocation $\vecb w\in \NATSZ^{\Hom(\cF;R)}$ of $H$ to $R[V(T)]$ with $A(\vecb u-\vecb w)=\vecb c$.
\end{proposition}
\begin{proof}
	The proof traces the following outline.
	We begin by restricting the problem to each of the reachability components in turn.
	This allows us to apply \cref{prop:flow} separately to each of them and find an allocation that is as desired everywhere but at $e_0$.
	In some sense, we transfer all the imbalance, encoded by $\vecb c$, to $e_0$.
	Once we have all the imbalance at $e_0$, we use the property of the proper \fc{} subtiling at $e_0$ to correct it, and thus obtain an allocation that is as desired everywhere.

	\newcommand{\kl}{k'}
	We divide the problem between the reachability components as follows.
	Denote the reachability components of $T$ by $C_1,\dots,C_{\kl}$, where $\kl \leq k$ by \cref{pro:reachabilitity}.
	We deal with each reachability component sequentially.
	Let $e_0=\{v_1,\dots,v_k\}$ such that $v_i\in V(C_i)$ for all $i\in [\kl]$.
	Note that this is possible  by \cref{pro:reachabilitity}.
	To this end, for each $j\in [{\kl}]$, we define $\vecb b_j\in \INTS^{V(R)}$ by setting
	\begin{align*}
		\vecb b_j(v)=
		\begin{cases}
			\vecb   c (v) - \sum_{u\in V(C_j)} \vecb c(u) & \text{if $v=v_j$},                     \\
			\vecb   c (v)                                 & \text{if $v\in V(C_j)\sm \{v_j\}$ and} \\
			0                                             & \text{otherwise}.
		\end{cases}
	\end{align*}
	Thus, for each reachability component $C_j$, we can think of $\vecb b_j$ as encoding `transfers' of the value $\vecb c(v)$ to $v_j\in e_0$ from each vertex in $v\in V(C_j)\sm \{v_j\}$.
	In particular, for all $j\in [{\kl}]$, we have $\sum_{v\in V(C_j)}\vecb b_j(v)=0$, $\vecb b_j (v) = 0$ for all $v \in V(R) \sm V(C_j)$; in addition, we have $\sum_{j \in [\kl]}\maxnorm{\vecb b_j}\leq r  \maxnorm{\vecb c} \leq rs$.
	We now have ${\kl}$ vectors $\vecb b_1,\dots,\vecb b_{{\kl}}$, one for each reachability component, each satisfying the conditions of \cref{prop:flow}.

	Our next task is to find a set of altered allocations that, when combined, implement the `transfers' encoded by $\vecb b_j$ on each reachability component.
	Starting with $C_1$ this means we want to change the allocation of some tiles such that the resulting change after we apply the map $A$ is $\vecb b_1$.
	Let $\vecb w_0= \vecb u_1$.
	We apply \cref{prop:flow} in ${\kl}$ rounds with input
	$H',C_{i},\vecb b_{i},\vecb w_{i-1},rs,s'/2^{i-1}$ playing the role of $H,C,\vecb b,\vecb u,s,s'$ in the $i\tth$ round and the remaining parameters and objects keeping their names, which is possible as $r (r^2s+w) \leq s'/2^k$.

	After the $i\tth$ round, for $i\in [{\kl}]$, we obtain a $(k,s'/2^{i})$-surjective allocation $\vecb w_{i} \in \NATSZ^{\Hom(\cF;R)}$ of $H'$ to $R[V(T)]$ with $A(\vecb w_{i-1}-\vecb w_{i}) = \vecb b_{i}$.
	Thus in $\vecb w_{{\kl}}$ we have implemented all the transfers $\vecb b_1.\dots, \vecb b_{{\kl}}$ and $A(\vecb u_1 - \vecb w_{{\kl}})=\sum_{j \in [\kl]}\vecb b_j$.
	Note that here we also use the fact that $\vecb b_j(v)$ is non-zero for at most one $j \in [{\kl}]$ for all $v \in V(R)$.
	At this point we have found an allocation $\vecb w_{{\kl}}$ of $H'$ such that $\vecb w_{{\kl}}+\vecb u_2$ is the desired allocation everywhere with the exception of the vertices of $v_1,\dots,v_{{\kl}}$.
	The final step is to use the \fc{} property to find a new allocation of $W$ to replace $\vecb u_2$ and so fix the only remaining imbalances, which are  at $e_0$.
	Define $\vecb c' :=\vecb c - A(\vecb u_1 - \vecb w_{{\kl}})=\vecb c - \sum_{j \in [\kl]}\vecb b_j$ and note that $\vecb c'(v) \neq 0$ implies $v \in \{v_1,\dots,v_{{\kl}}\}$.
	By assumption $\sum_{v\in V(R)}\vecb c(v)=0$ and so
	\[\sum_{v\in V(R)}\vecb c'(v)=\sum_{j \in [\kl]}\left( \vecb c(v_j) - \vecb b_j(v_j)\right)=\sum_{j \in [\kl]} \sum_{v\in V(C_j)} \vecb c(v) =0.
	\]
	Observe that $\maxnorm{\vecb c'}\leq r\cdot \maxnorm{\vecb c} \leq rs$ and that $\vecb u_2$ encodes a central $k$-colouring, say $\phi$, of the proper $(k,rs)$-\fc{} subtiling $W$ to $R[e_0]$.
	If we now consider $\restr{\vecb c'}{e_0}$ and view it as living in $\INTS^k$, then by the definition of $(k,rs)$-\fc{} graphs, we can find a colouring $\phi'$ of $W$ with $\ord(\phi)-\ord(\phi')=\restr{\vecb c'}{e_0}$.
	This colouring $\phi'$ is encoded by an allocation $\vecb u'_2\in \NATSZ^{\Hom(\cF;R)}$ of $W$ to $R[e_0]$ that satisfies $A(\vecb u_2-\vecb u'_2)=\vecb c'$.
	We combine the new allocation $\vecb w_{{\kl}}$ of $H'$ and $\vecb u_2'$ of $W$ into an allocation $\vecb w \in \NATSZ^{\Hom(\cF;R)}$ of $H$ to $R$ defined by $\vecb w  = \vecb w_{{\kl}} +\vecb u'_2$.
	Since
	\begin{equation*}
		A(\vecb u - \vecb w)=A((\vecb u_1 -\vecb w_{{\kl}}) + (\vecb u_2 - \vecb u'_2))=\sum_{j \in [\kl]}\vecb b_j +\vecb c' =\vecb c,
	\end{equation*}
	this is the desired allocation $\vecb w$ of $H$.
\end{proof}

\section{Properties of \fc{} graphs}\label{sec:wildcards}
In the following, we establish a few further properties of \fc{} graphs.
We start by extending the original definition (\cref{def:wildcardS}) slightly.
A graph $W$ is $p$-\emph{approximately $(k,s)$-\fc{}} ($(k,s,p)$-\emph{\fc{}} for short) with central colouring $\phi$ if the following holds.
For all $\vecb w\in\INTS^k$ with $\sum_{i \in [k]}\vecb w(i)=0$ and $\maxnorm{\vecb w}\leq s$, there exists a second $k$-colouring~$\phi'$ of $W$ with $\maxnorm{\ord(\phi)-\ord(\phi')-\vecb w} \leq p$.
If the colourings are proper, then $W$ is properly \fc{}, and if they are topological, then $W$ is topologically \fc{}.
In particular a proper $(k,s,0)$-\fc{}  graph is equivalent to, and is still called, a proper $(k,s)$-\fc{} graph (as defined in \cref{def:wildcardS}).

We begin by proving the following result, which holds for both proper and topological colourings and which implies \cref{prop:build_wildcards-simple}.
\begin{proposition}\label{prop:build_wildcards}
	Let $k,w \in \NATS$ and $p \in \NATSZ$ with $p < w$.
	Let $H_1,\dots,H_{k^2}$ be graphs such that $H_i$ has a proper (topological) $k$-colouring $\theta_i$ for each $i\in[k^2]$ satisfying the following property.
	For all $y \in [w]$, there exists a second proper (topological) $k$-colouring $\theta'_i$ of $H_i$ such that $\ord(\theta'_i)(j)=\ord(\theta_i)(j)$ for $j\in[k-2]$ and $\ord(\theta'_i)(k)=\ord(\theta_i)(k)+y \pm p$.
	Then $\bigcup_{i \in [k^2]}H_i$ is properly (topologically) $(k,w/k,kp)$-\fc{}.
\end{proposition}

\begin{proof}
	We relabel the input so that we have a graph $H_{i,j}$ with a $k$-colouring $\theta_{i,j}$ for each $1 \leq  i < j \leq k$ (the remaining $k^2 - \binom{k}{2}$ graphs of $H$ are not used).
	Let $H=\bigcup_{i \in [k^2]}{ H_{i,j}}$.
	In order to show that $H$ is \fc{} we have to provide {a} central colouring $\phi$.
	We construct the (candidate for our) central colouring $\phi$ by defining its restrictions $\phi_{i,j}=\restr{\phi}{H_{i,j}}$ as follows.
	For $1 \leq i<j \leq k$ and  $F \in H_{i,j}$, we obtain $\phi_{i,j}$ from $\theta_{i,j}$ by switching colour class $k-1$ with $i$ and colour class $k$ with $j$.
	For all other tiles of $H$, $\phi$ and $\theta$ coincide.

	Let us now show that $\phi$ has the properties of a central colouring of a  $(k, w/k, kp)$-\fc{} graph.
	Suppose we are given some $\vecb w\in \INTS^k$, with $\sum_{i \in [k]}\vecb w(i)=0$ and $\maxnorm{\vecb w} \leq w/k$.
	Our goal is to construct a second $k$-colouring $\phi'$ of $H$ that satisfies $\maxnorm{\ord(\phi)-\ord(\phi')-\vecb w} \leq kp$.

	Let us first reformulate $\vecb w$ as follows.
	We apply \cref{pro:weights_to_flows} to the complete graph with vertex set $[k]$ and weights $\vecb w(1),\dots, \vecb w(k)$.
	This allows us to write  $\vecb w = \sum_{1\leq i < j \leq k}  \vecb w_{i,j}$ as the sum of vectors $\vecb w_{i,j}\in \INTS^k$  with  {$\vecb w_{i,j}(j)=-\vecb w_{i,j}(i)$}, $\maxnorm{\vecb w_{i,j}}\leq w$ and $\vecb w_{i,j}(\ell)=0$ for all $\ell \in  [k]\sm \{i,j\}$.
	(The factor $k$ is due to the change of norm.)

	Next, we define $\phi'$ via its restriction to the $H_{i,j}$'s using the weights $\vecb w_{i,j}$.
	Now let $i,j \in [k]$ be distinct.
	Suppose {first} that $0\leq y=\vecb w_{i,j}(j)  = - \vecb w_{i,j}(i) \leq w$.
	It remains to define $\phi'_{i,j}$.
	By assumption there is a colouring $\theta'_{i,j}$ of $H_{i,j}$, satisfying $\ord(\theta'_{i,j})(k)=\ord(\theta_{i,j}){(k)}+y\pm p$, {while the colour classes $1\ldots, k-2$ stay the same.}
	We obtain $\phi'_{i,j}$ from $\theta'_{i,j}$ by permuting the colour classes taking $k-1$ to {$j$} and $k$ to {$i$}.
	Thus $\ord(\phi'_{i,j})({i}) = \ord(\phi_{i,j})({i}) + y \pm p$ and {$\ord(\phi'_{i,j})(j) = \ord(\phi_{i,j})(j) -y \pm p$} and all colour class orders other than $i,j$ {are} as before.
	Finally, let $\phi' = \bigcup \phi_{i,j}'$.
	{If $\vecb w_{i,j}(j)<0$, we can proceed similarly but permuting $k-1$ to $i$ and $k$ to $j$.} 

	To verify that $\phi'$ has the desired property, we compute that for all $i\in [k]$
	\begin{align*}
		| & \ord(\phi)(i)-\ord(\phi')(i)-\vecb w(i) |\leq    \sum_{\substack{j \in [k]}}\Maxnorm{\ord(\phi_{i,j}) - \ord(\phi'_{i,j})-\vecb w_{i,j}}
		\leq kp,
	\end{align*}
	which completes the proof.
\end{proof}

We obtain \cref{lem:k+1_wildcard} as a corollary.

\begin{proof}[Proof of \cref{lem:k+1_wildcard}]
	We choose $(k+1)^2$ subtilings $S_i$ for $i\in [(k+1)^2]$ each of order at least  $n/(k+1)^2 - 2w$.
	We can do this greedily as each tile has order at most $w$.
	For each $S_i$, we choose a $k$-colouring with largest colour class $k$ and empty colour class $k+1$.
	As we can recolour vertices from colour $k$ to $k+1$ freely, and as the colour class $k$ has order at least $n/(k+1)^4$, we can apply \cref{prop:build_wildcards} with $k+1,n/(k+1)^4,0$ as $k,w,p$.
	It follows that $H$ is properly $(k+1, n/(k+1)^5 , 0)$-\fc{}.
\end{proof}

The following is a simple but pertinent property of \fc{} graphs.

\begin{proposition}\label{prop:wild_sum}
	Suppose $W_1$ is properly (topologically) $(k,s_1,p_1)$-\fc{} and $W_2$ is {properly (topologically)} $(k,s_2,p_2)$-\fc{} with central colourings $\phi_1$ and $\phi_2$, respectively.
	If $s_2 \geq p_1$, then $W=W_1\cup W_2$ is properly (topologically) $(k,s_1,p_2)$-\fc{}. Further, $\phi_1 \cup \phi_2$ is a central colouring of $W$.
	\begin{proof}
		Let $\phi = \phi_1 \cup \phi_2$, and fix a $\vecb w\in\INTS^k$ with
		$\sum_{i \in [k]}\vecb w(i)=0$ and
		$\maxnorm{\vecb w}\leq s_1$.
		We seek a $k$-colouring $\phi'$ of $W_1 \cup W_2$ with $\maxnorm{ \ord(\phi)-\ord(\phi')-\vecb w } \leq p_2$.
		We obtain this by first finding $\phi'_1:=\restr{\phi'}{W_1}$ and then using $\phi'_1$ to find $\phi'_2:=\restr{\phi'}{W_2}$ such that $\phi'$ has the desired properties.

		As $W_1$ is $(k,s_1,p_1)$-\fc{}, with  central $k$-colouring  $\phi_1$, we can find a $k$-colouring $\phi'_1$ of $W_1$ such that $\sum_{i \in [k]}\vecb w'(i)=0$ and $\maxnorm{\vecb w'}\leq p_1$ where $\vecb w' :=-(\ord(\phi_1)-\ord(\phi'_1)-\vecb w) \in \INTS^k$.
		By assumption $p_1\leq s_2$ and so we can find a colouring $\phi'_2$ of $W_2$ such that $\maxnorm{\ord(\phi_2)-\ord(\phi'_2)-\vecb w'}\leq p_2$.
		We see that $\phi'$ is the desired colouring as $\ord(\phi)=\ord(\phi_1)+\ord(\phi_2)$, $\ord(\phi')=\ord(\phi'_1)+\ord(\phi'_2)$ and
		\begin{align*}
			\maxnorm{\ord(\phi)-\ord(\phi')-\vecb w}
			\leq &
			\maxnorm{\left(\ord(\phi_2)-\ord(\phi'_2)\right) + \left(\ord(\phi_1)-\ord(\phi'_1) - \vecb w\right)}
			\\
			\leq &
			\maxnorm{\ord(\phi_2)-\ord(\phi'_2) -\vecb w'}\leq p_2.\qedhere
		\end{align*}
	\end{proof}
\end{proposition}

We now prove that sufficiently large $\cF(k,w)$-tilings are properly $kw$-approximately \fc{}.

\begin{proposition}\label{lem:Lat_gap}
	Let $1/n\leq \xi  \ll \tau \ll 1/\chi$ {with $k = \kk$ and $\chi < k-\tau$.}
	Let $w\leq \xi n$ and $H$ be an $ \cF(k,w)$-tiling  on $n$ vertices with $\crit(H) < k-\tau$. Then $H$ is properly $(k,\tau n/(2k^5), kw)$-\fc{}.
\end{proposition}

\urldef\myurl\url{https://www.wolframalpha.com/input/?i=k-x%3D%28k-1%29+%2B+%28k-1%29a%2F%28n-a%29+solve+for+a}
\begin{proof}
	Our strategy for the proof is as follows.
	We first find a colouring $\theta$ of $H$ in which the size of the largest and smallest colour classes differ by at least $ \frac{\tau}{2k^2} n $.
	We then use this colouring to partition $H$ into  subtilings that satisfy the conditions of \cref{prop:build_wildcards} and conclude.

	Observe that we can assume $\crit(H) \geq k-1$ as we can simply {select $k-1$ vertices in $H$ and add edges so that these $k-1$ vertices induce a $(k-1)$-clique (call this new graph again~$H$)}.
	We may thus write $\crit (H) = k-x$ for some $\tau < x < 1$.
	By definition, there is a $k$-colouring~$\theta$ of $H$ such that $\crit(H) = (k-1) + \frac{(k-1)a}{n - a}$ where $a=|\theta^{-1}({k})|$.
	Solving for $a$,\COMMENT{\myurl} this gives $|\theta^{-1}({k})| = \frac{1-x}{1-x/k} \frac{n}{k}  \leq (1-\frac{x}{k})\frac nk  <  (1- \frac{\tau}{k}) \frac{n}{k}$, where we used that $k\geq 2$ in the first inequality.\COMMENT{$1-x \leq 1 - 2x/k + (x/k)^2= (1-x/k)^2$}
	Moreover, by averaging, there is also some colour, say ${k-1}$, such that $|\theta^{-1}({k-1})| \geq \frac{n}{k} $ and so $|\theta^{-1}({k-1})|-|\theta^{-1}({k})|>{\frac{\tau}{k}  \frac{n}{k}} $.
	For each tile~$F$ of $H$, we denote by $\theta_F= \restr{\theta}{V(F)}$ the colouring of~$F$ induced by $\theta$.
	Summing over the tiles of $H$, this gives
	\begin{equation*}
		\sum_{F}|\theta_F^{-1}({k-1})|-|\theta_F^{-1}({k})| =  |\theta^{-1}({k-1})|-|\theta^{-1}({k})| \geq  \frac{\tau}{k} \frac{n}{k}.
	\end{equation*}

	We greedily partition the tiles of $H$ into $k^2$ subtilings $H_i$ for $i \in[k^2]$ such that for each such~$H_i$, we have
	\begin{equation*}
		  \sum_{F\in H_i}|\theta_F^{-1}({k-1})| - |\theta_F^{-1}({k})| \geq \frac{1}{k^2}\frac{\tau}{k} \frac{n}{k}-w.
	\end{equation*}
	This is possible because $v(F)\leq w$ {for each tile $F$}.
	Let $\theta_i = \bigcup_{F \subset H_i} \theta_F$.

	We intend to apply \cref{prop:build_wildcards} to the $H_i$'s with $k,\frac{\tau n}{2k^4},w$ playing the role of $k,w,p$.
	To verify the {assumptions}, fix $H_i$ and consider an arbitrary $y \in [{\frac{\tau n}{2k^4}}]$.
	We have to show that there exists a second $k$-colouring $\theta'_i$ of $H_i$ such that $\ord(\theta'_i)(j)=\ord(\theta_i)(j)$ for $j\in[k-2]$ and $\ord(\theta'_i)(k)=\ord(\theta_i)(k)+y \pm w$.
	It is not hard to see that one can obtain such a colouring by flipping the colours $k-1$ and $k$ of tiles $F \in H_i$ which {have} more vertices of colour $k$ than colour $k-1$ under $\theta_F$.
	Since $v(F) \leq w$, we overshoot by at most $w$ when this process ends.
	Hence \cref{prop:build_wildcards}  applies and it follows that $H$ is $(k,\tau n/(2k^5), kw)$-\fc{}.
\end{proof}

\section{Proof of \nameref{lem:lemma-for-H}}\label{sec:lemma-for-H}

In the following, we show \rf{lem:lemma-for-H}, which states that given a certain $R$-partition of $G$, we can find a size-compatible vertex partition of the given graph $H$.

\begin{proof}[Proof of \cref{lem:lemma-for-H}]
	Let $k={\lceil\chi\rceil}$.
	Given $\mu$ with $\chi+\mu/2 \leq \lceil\chi\rceil$, we can assume that $\chi + \mu/2 = (k-1) + a/b$ for (coprime) natural numbers {$a \leq b \leq 8/\mu$}.\footnote{{Let $4{\mu^{-1}}\leq  b:= 4 \lceil \mu^{-1}  \rceil \leq 8/\mu$ and choose $a \in \NATS$ maximal such that $k-1 + a/b \leq \chi +\mu/2$.}}
	{(As in the proof of \cref{thm:frameworks}, we used that by \cref{obs:framework-monotone} the result only gets stronger for smaller $\mu$.)}
	We choose $\rho',\beta>0$ such that
	\begin{align*}
		1/n \leq \xi \ll  \rho \ll \rho' \ll \beta \ll 1/r  \ll \alpha \ll \mu, 1/\Delta,1/\chi.
	\end{align*}
	Let $w \leq \xi n$.
	Suppose that $G$ is a $\mu$-robust $(\chi+\mu,\rho;t,\ell)$-tiling framework on $n$ vertices.
	Let $R$ be a  $({\chi}, \rho; t,\ell)$-tiling framework on vertex set $[r]$
	{with a $K_{k}$-cover $R' \subset R$}
	with $\Delta(R') \leq \Delta$.
	For $\cV=\{V_i\}_{i \in [r]}$, let $(G,\cV)$ be a  {$2$-balanced} $R$-partition,  where $m\leq|V_i|\leq 2m$ for $i\in [r]$.
	So in particular $m \geq n/(2r)$.
	Let $H \in \cH(n,w,\Delta; \chi; t,\ell,w')$.
	{Let $J = K_k(R)$.
	Since $R$ is a $({\chi}, \rho; t,\ell)$-tiling framework, we know that $J$ is $(t,\ell)$-connected.
	Hence there is a (vertex) spanning subgraph $J' \subset J$ with tight components $T_1,\dots, T_{t'} \subset J'$ with $t' \leq t$ and loose components $L_1,\dots,L_{\ell'}$ with $\ell' \leq \ell$.
	Observe that $H \in \cH(n,w,\Delta; \chi; t',\ell',w')$.
	In the following, we hence may assume that $t'=t$ and $\ell'=\ell$ as otherwise we simply redefine $t$ and $\ell$.
	Thus $t\leq r^k$ and $\ell\leq t$ as $e(J)\leq r^k$.}

	Our objective is now to find a size-compatible (with $\cV$) vertex partition $\cX=\{X_i\}_{i \in [r]}$ of $H$ and $\tcX= \{\tilde{X}_i\}_{i \in [r]}$  with $\tilde{X}_i\subset X_i$ for each $i\in [r]$ such that $(H,\cX)$ is an $R$-partition and $\tcX$ is an $(\alpha,R')$-buffer for $(H,\cX)$.
	We reformulate this problem in vector notation.
	Let $\vecb b \in \mathbb{N}^{V(R)}$ with $\vecb{b}(i)=|V_i|$.
	Set $\cF = \cF(k,w)$ and $A=A_{\cF,R}\in \NATSZ^{V(R) \times \Hom(\cF;R)}$.
	Our goal is to find an allocation $\vecb v^\ast \in \NATSZ^{\Hom(\cF;R)}$ of $H$ to $R$ {(as defined in \cref{sec:allocations})} such that
	\begin{enumerate}[(1)]
		\item  \label{itm:lem-H-matrix} $A \vecb v^\ast = \vecb b$ and
		\item \label{itm:lem-H-buffer} $A ( \restr{ \vecb v^\ast}{\Hom(\cF;R')}) {\geq {2}\alpha m}$.
	\end{enumerate}
	Given such vectors we generate the desired $R$-partition $(H,\cX)$ as follows.
	Let $\eta$ be a homomorphism from $H$ to $R$ encoded by $\vecb v^\ast$ {(as defined in \cref{sec:allocations})}.
	We set $X_i  =\eta^{-1}(i)$ for each $i\in [r]$, and let
	$\cX=\{X_i\}_{i \in [r]}$.
	It follows that $(\cX,R)$ is an $R$-partition.
	Moreover, $\cX$ is size-compatible with $\cV$ by~\ref{itm:lem-H-matrix}.
	For the buffer,
	observe that the restriction $\restr{\vecb v^\ast}{\Hom(\cF;R')}$ encodes homomorphisms from some tiles of $H$ to $R'$.
	Thus $\restr{\vecb v^\ast}{\Hom(\cF;R')}$ is an allocation of some subtiling~$F$ of $H$ to $R'$ and encodes a homomorphism $\eta'$ from $F$ to $R'$.
	We define $\tilde{X}_i=\eta'^{-1}(i)$.
	When combined with~\ref{itm:lem-H-buffer}, this states that $|\tilde{X}_i|\geq {\alpha} |X_i|$ for all $i\in [r]$.

	Now we give an overview of the proof.
	We first ensure that~\ref{itm:lem-H-buffer} is satisfied by finding a $\vecb v_2^\ast$ that allocates a subtiling of $H$ (\cref{cla:buffer}).
	The remainder of $H$ is denoted by $H'$.
	We then find a $\vecb v'$, which allocates $H'$ to $R$ such that $A (\vecb v' + \vecb v_2^\ast) \approx \vecb b$  (\cref{clm:allocation_vectors}).
	In addition to this,~$\vecb v'$ also matches the loose and tight components of the tiling framework $R$  with the topological and proper \fc{} subtilings of $H'$.
	We finish by applying \cref{lem:allocate-tight} separately to each tight component (\cref{clm:split_between_tight}) to obtain $\vecb v_1^\ast$ from $\vecb v'$ such that $A (\vecb v_1^\ast + \vecb v_2^\ast) = \vecb b$.
	Now come the details.

	The following claim establishes~\ref{itm:lem-H-buffer}.
	\begin{claim} \label{cla:buffer}
		There is a {subtiling} $H' \subset H$ and an allocation $\vecb v_2^\ast\in \NATSZ^{\Hom(\cF;R)}$ of $H - V(H')$ to $R'$ with
		\begin{enumerate}[\upshape (1)]\addtocounter{enumi}{2}
			\item $H' \in \cH(n,w,\Delta; \chi+\sqrt{ \alpha}; t,\ell,w')$ and
			\item \label{itm:claim-buffer-bounded} ${2}\alpha m\leq A ( \restr{ \vecb v_2^\ast}{\Hom(\cF;R')} ) \leq \sqrt{\alpha }m$.
		\end{enumerate}
	\end{claim}
	We assume the truth of the claim for the time being and prove it later.
	Let $H'$ and $\vecb v_2^\ast$ be as in \cref{cla:buffer}.
	Set $\vecb b' = \vecb b - {A}\vecb v_2^\ast$; that is, $\vecb b'(i)$ equals the number of vertices of 	$V_i$ that are not yet used by $H - V(H')$ for the allocation.
	We now seek an allocation $\vecb v_1^\ast\in \NATSZ^{\Hom(\cF;R)}$ of $H'$ to~$R$ with $A \vecb v_1^\ast = \vecb b'$.
	We fix a  $k$-edge $e_j \in E(T_j)$ for each $j\in[t]$ and a $k$-edge $f_q \in E(L_q)$ for each $q\in [\ell]$.

	The next claim {provides} an allocation {$\vecb v'\in \NATSZ^{\Hom(\cF;R)}$} of $H'$ to $R$, which almost satisfies $\vecb b'$ and prepares the \fc{} subtilings for their upcoming uses.
	Note that the proper \fc{} subtilings are somewhat stronger than those ensured by $H \in \cH(n,w,\Delta; \chi; t,\ell,w')$.
	\begin{claim} \label{clm:allocation_vectors}
		There are disjoint subtilings $W,W_1,\dots,W_t$ of $H$ such that $W$ is topologically $(\ell,w)$-\fc{} and $W_j$ is properly $(k, {\sqrt{\rho'}r n})$-\fc{} for each {$j \in [t]$}.
		Moreover, there exist $\vecb w, \vecb w_{j},\vecb u_{j}, \vecb r \in \NATSZ^{\Hom(\cF;R)}$ for each {$j \in [t]$} with $\vecb v' := \vecb r  +\vecb w + \sum_{j\in [t]} \vecb u_j +\vecb w_j$ and  $\vecb c ': = A\vecb v' - \vecb b' \in \NATSZ^{V(R)}$ such that {for each $j\in [t]$}
		\begin{enumerate}[\upshape (1)] \addtocounter{enumi}{4}
			\item \label{itm:allocation_vectors-loose-wildcard} $\vecb w$ is an allocation of $W$ to $R$ that encodes a {topological} central $\ell$-colouring of $W$ such that each $R[f_1],\dots,R[f_\ell]$ receives one colour class.\footnote{{That is, $\vecb w$ assigns to each $R[f_{\ell'}]$ for $\ell'\in[\ell]$ a set of components of $W$ and as $R[f_{\ell'}]\in E(K_k(R))$ is a $k$-clique in $R$, the assignment inside this $k$-clique can be done in such a way that each colour class of a proper colouring of this set of components is assigned to one unique vertex of this $k$-clique.}}
			\item $\vecb w_j$ is an allocation of  $W_j$ to $R[e_j]$ that encodes a {proper} central $k$-colouring of $W_j$,
			\item \label{itm:allocation_vectors-u}  $\vecb u_j$ is a $(\beta n,T_j)$-surjective allocation of a subtiling of $H'$ to $R[V(T_j)]$,
			\item $\vecb v'$ is an allocation of $H'$ to $R$ and
			\item \label{itm:allocation_vectors-wildcard-distance} ${\sum_{i\in V(L_q)}|\vecb c'(i)|}\leq w$ for each $q\in [\ell]$ {and $\sum_{q\in [\ell]} \sum_{i\in V(L_q)}\vecb c'(i)=0$.}
		\end{enumerate}
	\end{claim}
	We also  defer the proof of this claim and continue with the main argument.
	Fix $\vecb w_1,\dots,\vecb w_t$, $\vecb u_1,\dots,\vecb u_t$, $\vecb r  ,\vecb v', \vecb c'$ and $W, W_1,\dots,W_t$ as in \cref{clm:allocation_vectors}.

	Next, we use the definition of topologically $(\ell,w)$-\fc{} graphs together with~\ref{itm:allocation_vectors-loose-wildcard} to reallocate $W$.
	By~\ref{itm:allocation_vectors-loose-wildcard}, there is a central topological $\ell$-colouring $\phi$ of $W$ encoded by $\vecb w$.
	By \cref{def:wildcardS} and \ref{itm:allocation_vectors-wildcard-distance}, there is a second topological $\ell$-colouring $\phi'$ of $W$ with $(\ord(\phi))(q)-(\ord(\phi'))(q)=\sum_{i\in V(L_q)}\vecb c'(i)$.
	Let $\vecb w'\in \NATSZ^{\Hom(\cF;R)}$ be an allocation of $W$ to $R$ that encodes $\phi'$.
	(Such an allocation exists as $W$ is properly $k$-colourable, {as it is a subtiling of $H$}, and each $R[f_q]$ induces a copy of $K_k$ to which we can allocate $W[\phi'^{-1}(q)]$.)
	Define
	\begin{align}\label{equ:def-v''-c}
		\text{$\vecb v''=\vecb r +  \vecb w' + \sum_{j \in [t]} (\vecb u_j +\vecb w_j)$ and $\vecb c=A\vecb v'' -  \vecb b'$.}
	\end{align}
	By construction, we have
	\begin{align}\label{equ:vecc-sums}
		\text{$\onenorm{\vecb c}\leq r{\rho'} n$ and   $\sum_{i\in V(L_q)}\vecb c(i) = 0$ for each $q\in [\ell]$.}
	\end{align}

	We may think of $\vecb c$ as the imbalance that is still present in our current allocation and~\eqref{equ:vecc-sums} says that there is no imbalance across loose components.
	The following claim states that the imbalance can be divided into $t$ parts such that there is no imbalance across tight components any more.
	The imbalance inside each tight component can then be handled by a single \fc~subtiling via \cref{lem:allocate-tight}.

	\begin{claim}\label{clm:split_between_tight}
		There exist $\vecb c_1,\dots,\vecb c_{t}\in \INTS^{V(R)}$ such that
		\begin{enumerate}[\upshape (1)] \addtocounter{enumi}{9}
			\item  \label{itm:balanced} $\sum_{j \in [t]}\vecb c_j=\vecb c$,
			\item  \label{itm:clm-split-simple} {$\maxnorm{\vecb c_j}  \leq \sqrt {\rho'} n$ for all $j \in [t]$,}
			\item $\sum_{i \in [r]}\vecb c_j(i) = 0$   for all $j\in [t]$ and
			\item  $i \in V(T_j)$ whenever $ \vecb c_j(i) \neq 0$ for all $j\in [t]$.
		\end{enumerate}
	\end{claim}

	As before, we assume the truth of the claim for now and prove it later.
	Let $H_j'$ be the subtiling of $H'$ allocated by $\vecb w_j +\vecb u_j$ for each $j \in [t]$.
	For each tight component $T_j$, we apply \cref{lem:allocate-tight} with
	\begin{center}
		\begin{tabular}{ c  *{10}{|c} }
			object/parameter    & $W_j$ & {$R$} & $H_j'$ & $T_j$ & $e_j$ & $\vecb u_j$ & $\vecb w_j$ & $\vecb c_j$ & $\sqrt{\rho'} n$ & $\beta n$ \\
			\hline
			playing the role of & $W$   & $R$         & $H$    & $T$   & $e_0$ & $\vecb u_1$ & $\vecb u_2$ & $\vecb c$   & $s$              & $s'$
		\end{tabular}
	\end{center}
	to obtain a second allocation $\vecb v_j$ of $H_j'$ such that
	\begin{align}\label{equ:lemma-H-Auwv=c}
		A( \vecb u_j + \vecb w_j-\vecb v_j) = \vecb c_j.
	\end{align}

	We claim that these allocations combine to give the result.
	Note that as $\vecb v''=\vecb r +  \vecb w'  +\sum_{j \in [t]} (\vecb u_j +\vecb w_j) $ is an allocation of $H'$ to $R$, so is
	$\vecb v_2^\ast :=\vecb r + \vecb w' + \sum_{j \in [t]} \vecb v_j $.
	Indeed we change neither $\vecb r $ nor $\vecb w'$, and $\vecb v_j$ allocates the same subtiling $H_j'\subset H'$ as $\vecb u_j + \vecb w_j$ {does}.
	Further,
	\begin{align*}
		A\vecb v_2^\ast - \vecb b'
		                                     & = A(\vecb v_2^\ast -  \vecb v '' ) + (A \vecb v '' - \vecb b')                                                                                          \\
		\overset{\eqref{equ:def-v''-c}}      & {=} A\left(\vecb r + \vecb w'+ \sum_{j \in [t]} \vecb v_j  - \left( \vecb r + \vecb w'+ \sum_{j \in [t]}  (\vecb u_j + \vecb w_j)\right)\right) + \vecb c \\
		\overset{\eqref{equ:lemma-H-Auwv=c}} & {=}  - \sum_{j \in [t]} \vecb c_j + \vecb c 
		= 0,
	\end{align*}
	{where we used \ref{itm:balanced} for the last equality.}
	Thus $\vecb v_2^\ast$ is the desired allocation of $H'$, and so $\vecb v^\ast = \vecb v_1^\ast + \vecb v_2^\ast$ satisfies both~\ref{itm:lem-H-matrix} and \ref{itm:lem-H-buffer}.
	To finish the proof, it remains to show the three claims.
	
	{We use the following easy observation. Given a graph $\tilde{H}$ and we delete at most a $\gamma\leq 1/2$ fraction of its vertex set to obtain $\tilde{H}'$,
	then $\crit(\tilde{H}')\leq \crit(\tilde{H})+1/(1-\gamma)\leq \crit(\tilde{H})+2\gamma$.
	This follows easily from \eqref{equ:crit}.}

	\begin{proof}[Proof of \cref{cla:buffer}]
		Consider any $i \in [r]$.
		Choose a subtiling $H_i \subset H$ of order ${2}k\alpha m \leq v(H_i) \leq  {2}k\alpha m +w \leq {3}k\alpha m $.
		(This is possible, since the tiles of $H$ have size at most $w \leq \xi n \leq \alpha m$.)
		We do so such that each subtiling is disjoint and does not intersect the \fc{} subtilings of $H$ (which in total span at most $(t+1)w'\leq(t+1)\sqrt\xi n$ vertices).
		Consider a (proper) $k$-colouring of $H_i$ and note that its largest colour class has at least $\alpha m$ vertices.
		Since $R'$ is a $K_k$-cover, there is some $k$-clique $K_i \subset R'$ which contains {(vertex)} $i$.
		Let $\vecb q_i$ be the allocation of $H_i$ to the clique $K_i$ such that the largest colour class is mapped to $i\in V(K_i)$.
		Note that  at most~${3}k\alpha m$ vertices are allocated to any vertex $j$ of $K_i$.
		
		{Altogether}, since $\Delta(R') \leq \Delta$, this allocates in total at most ${3}k\Delta\alpha m$ vertices to any cluster.
		{In addition,} $\sum_{i\in [r]}v(H_i) \leq {3k\alpha mr} \leq {\alpha^{2/3}}n$.
		Consequently, we may assume the $H_i$'s are pairwise disjoint.
		It follows that $\vecb v_2^\ast := \sum_{i\in [r]} \vecb q_i$ satisfies  ${\alpha^{2/3}}m \geq {3}k\Delta\alpha m \geq  A  ( \restr{ \vecb v_2^\ast}{\Hom(\cF;R')} )\geq {2}\alpha m$.
		Let $H'= H - \bigcup V(H_i)$.
		
		{By construction,} $H'$ (like $H$) contains $t$ disjoint proper $(k,kw)$-\fc{} subtilings and one topological $(\ell,w)$-\fc{} subtiling each of order at most $w'\leq\sqrt\xi n$.
		{Moreover,} $\crit(H') \leq \crit(H) + \sqrt\alpha$.
		Thus $H' \in \cH(n,w,\Delta; \chi+\sqrt{ \alpha}; t,\ell,w')$.
	\end{proof}

	\begin{proof}[Proof of \cref{clm:allocation_vectors}]
		Let us first sketch the proof of the claim.
		We begin by finding and allocating the \fc{} subtilings $W,W_1,\dots,W_t$ with $\vecb w, \vecb w_1,\dots,\vecb w_t$.
			{Note that the subtilings $W_1,\dots,W_t$ are proper $(k, {\sqrt{\rho'}r n})$-\fc{}, whereas $H' \in \cH(n,w,\Delta; \chi+\sqrt{ \alpha}; t,\ell,w')$ only guarantees proper $(k, kw)$-\fc{} subtilings.}
		We achieve this by combining the \fc{} subtilings given by the assumptions with approximate \fc{} subtilings generated by \cref{lem:Lat_gap}.

		Recall that we aim to find an allocation of $H'$ to $R$.
		Next, we find ivive allocations $\vecb u_1,\dots,\vecb u_t$ as in~\ref{itm:allocation_vectors-u} by a simple greedy argument.
		We then delete vertices from each cluster of $\cV$ according to the allocation so far.
		This yields an approximation $G' \subset G$ which is still a tiling framework.
		To finish, {we} use \cref{thm:mixed-Komlos-for-frameworks} to find an almost perfect tiling of the yet--to--be--allocated vertices {in} $G'$.
		From this, we derive an allocation of almost all of the vertices of $H'$ to $R$.
		Arbitrarily allocating the remaining vertices gives $\vecb r$ and thus $\vecb v'$.

		Let us now turn to the details.
		We begin by finding the \fc{} subtilings.
		Write $H'=H^{(1)}$ for sake of notational symmetry.
		Since $H^{(1)} \in \cH(n,w,\Delta; \chi+\sqrt{ \alpha}; t,\ell,w')$, there are pairwise disjoint subgraphs ${W},W'_1,\dots, W'_t \subset H^{(1)}$ {each of order at most $w'$} such that $W'_1,\dots, W'_t$ are properly $(k,kw,0)$-\fc{} and  ${W}$ is topologically $(\ell,w,0)$-\fc{}.
		Let  $H^{(2)} = H^{(1)} - V({W}\cup \bigcup_{i \in [t]} W'_i)$.
		As each {$W_j'$} has order at most $w'\leq \sqrt\xi n \leq \beta n$, it follows that $\crit(H^{(2)})\leq \crit(H^{(1)})+ \sqrt \beta$, {again by our observation above}.
		Now choose disjoint subtilings $W_1'',\dots,W_t''\subset H^{(2)}$, each on $\beta n +w \pm w$ vertices such that $W_j''$ has $\crit(W_j'') {<} k - \mu/4$ for $j\in[t]$.
		{This is possible because $\beta t \leq 1/20$, say, and in particular $\chi \leq k - \mu/2$.}
		{We apply \cref{lem:Lat_gap} with $W_j''$, $\sqrt\xi $ and {$3k^5\rho'^{1/3}\beta^{-1}$}  playing the role of $H$, $\xi$ and $\tau$.}
		It follows that $W_j''$ is properly $(k,\rho'^{1/3} n, kw)$-\fc{}.
		Let $W_j=W'_j\cup W_j''$ for $j\in [t]$.
		By \cref{prop:wild_sum}, we see that each $W_j$ is properly $(k,{\rho'}^{1/3} n,0)$-\fc{}.
		Let  $H^{(3)} = H^{(2)} - \bigcup_{j \in [t]} V(W_j'')$ and observe that $\crit(H^{(3)})\leq \crit(H^{(2)}) + 4t\beta \leq \crit(H) + \mu/4$.
		For each $j\in[t]$, consider a central proper $k$-colouring $\phi_j$ of $W_j$, and let $\vecb w_j$ be an allocation of~$W_j$ to $R[e_j]$ that encodes $\phi_j$.
		Similarly, let $\vecb w$ be an allocation of $W$ to $R$ that encodes a central topological $\ell$-colouring of $W$ via $J'$, so $\vecb w$ places one colour class into each $R[f_q]$ for $q\in[\ell]$.
		(This is possible, since $W$ is properly $k$-colourable and $R[f_q]$ is a $k$-clique.)
		Since $W,W_1,\dots,W_t$ have together at most {$ 2t\beta n \leq \alpha m$} vertices, it follows that {$A  (\vecb w + \sum_{j \in [t]} \vecb w_j)  \leq \alpha m$.}

		Next we find $(k,\beta n)$-surjective allocations $\vecb u_1,\dots,\vecb u_t$ as in~\ref{itm:allocation_vectors-u}.
		The argument is the same as in \cref{cla:buffer}.
		Fix an arbitrary $j \in [t]$.
		Consider any $k$-clique $K \subset R$ with $V(K) \in E(T_j)$ and a vertex $i \in V(K)$.
		Choose a subtiling $H_{K,i} \subset H^{(3)}$ of order $k\beta n \leq v(H_i) \leq 2k\beta n$;
		consider a (proper) $k$-colouring of $H_{K,i}$ and note that its largest colour class has at least $\beta n$ vertices.
		Let $\vecb q_{K,i}\in \NATSZ^{\Hom(\cF;R)}$ be the allocation of $H_{K,i}$ to the clique $K$ such that the largest colour class is mapped to $i\in V(K)$.
		Note that this allocates at most $2k\beta n$ vertices to any vertex $j$ of $K_i$.
		Since the number of $k$-cliques in $R$ is at most $r^k$, this allocates in total at most $2kr^k \beta n \leq \sqrt{\beta} m$ vertices to any cluster.
		Similarly, the union of the $H_{K,i}$'s has order at most $2kr^k\beta nt \leq \sqrt{\beta} m$ where {we used that $t\leq r^k$.}
		Consequently, we may assume the $H_{K,i}$'s are pairwise disjoint.
		It follows that $\vecb u_j = \sum_{V(K) \in E(T_j)}{\sum_{i\in V(K)}} \vecb q_{K,i}$ is a $(\beta n,T_j)$-surjective allocation of a subtiling of $H^{(1)}$ to $R$.
		We remark that {$A  (\sum_{j \in [t]}   \vecb u_j ) \leq  t\sqrt{\beta} m \leq \alpha m$.}
		Let $H^{(4)}= H^{(3)} - \bigcup V( H_{K,i})$ and note that $\crit(H^{(4)})\leq \crit(H^{(3)}) + {2\sqrt{\beta}t} \leq \crit(H) + \mu/2$.

		All that remains in the proof of \cref{clm:allocation_vectors} is to find $\vecb r$.
		Recall that $G$ is a $\mu$-robust $(\chi+\mu,\rho;t,\ell)$-tiling framework on $n$ vertices and $(G,\cV)$ is a $2$-balanced $R$-partition.
		Note that {$\vecb a := A (\vecb v_2^\ast + \vecb w +\sum_{j \in [t]}  ( \vecb u_j + \vecb w_j))$} records the overall number of vertices {allocated to each $i\in V(R)$ so far.}
		Moreover, by the above choices and~\ref{itm:claim-buffer-bounded}, we have {$\vecb a \leq 2{\alpha }m + \sqrt{\alpha }m \leq 2{\sqrt{\alpha }m}$.}
		We obtain a subgraph $G' \subset G$ by deleting for each $i \in [r]$, exactly $\vecb a(i)$ vertices of $V_i$.
		This is possible as $|V_i| \geq m$.
		Note that $G'$ is a $(\mu/8,\mu/8)$-approximation of $G$, {since 2$\sqrt{\alpha}m$ is much smaller than $\mu n/8$.}
		{Together with \cref{obs:framework-monotone}, we derive} that $G'$ is a $\mu/2$-robust $(\chi+\mu/2,\rho;t,\ell)$-tiling framework.
		Choose an arbitrary subtiling $H^*\subset H^{(4)}$ such that $(1-{\rho'/2})v(G')-w \leq v(H^*) \leq (1-{\rho'/2})v(G')$ and $\crit(H^*)\leq \crit(H) + \mu/2$.
		Note that $H^* \in \cB(v(H^\ast), w,\Delta)  \cap \cJ(\chi + \mu/2)$.
		Hence, we can apply \cref{thm:mixed-Komlos-for-frameworks}  with
		\begin{center}
			\begin{tabular}{ c  *{11}{|c} }
				object/parameter    & $G'$ & $H^*$ & $v(G')$ & $2 \xi$ & $\chi+\mu/2$ & $a$ & $b$ & $\rho$ & $\rho'/2$ & $\mu/2$ & $\Delta$ \\
				\hline
				playing the role of & $G$  & $H$   & $n$     & $\xi$   & $\chi$       & $a$ & $b$ & $\rho$ & $\rho' $  & $\mu $ & $\Delta$
			\end{tabular}
		\end{center}
		to find an embedding of $H^\ast$ to $G'$.
		We encode this embedding by an allocation $\vecb r'\in \NATSZ^{\Hom(\cF;R)}$ of $H^*$ to $R$.
		Finally, we allocate the tiles in {$H^{(4)}- V(H^*)$} arbitrarily to obtain an allocation $\vecb r\in \NATSZ^{\Hom(\cF;R)}$ of $H^{(2)}$ to $R$.
		As ${v(H^{(4)})}-v(H^*)\leq \rho' n/2$ and by the construction of $\vecb r$, we have that $\vecb v' := \vecb r + \vecb w + \sum_{j \in [t]}  (\vecb u_j + \vecb w_j)$ satisfies
		$\onenorm{A\vecb v'- \vecb b'} \leq \rho' n$.

		Finally, let $\vecb c' = A\vecb v' - \vecb b' \in \NATSZ^{V(R)}$.
		By the last paragraph, we have ${|\sum_{i\in V(L_q)}\vecb c'(i)|} \leq  {\sum_{i\in V(L_q)}|\vecb c'(i)|}\leq {\rho'} n$ for each $q\in [\ell]$.
		In fact, we can assume that ${|\sum_{i\in V(L_q)}\vecb c'(i)|}\leq w$ for each $q\in [\ell]$.
		This is indeed possible as $\sum_{i \in [r]}\vecb c'(i)= 0 $ which holds as $\vecb v'$ is an allocation of all of $H$ to~$R$.
		So repeatedly reallocating a tile of $H$ (of order at most $w$) from the induced graph $R[V(L_q)]$ that maximises the sum of $\vecb c'$ on its vertices to the  $R[V(L_{q'})]$ that minimises it, eventually gives the desired result.
	\end{proof}

	\begin{proof}[Proof of \cref{clm:split_between_tight}]
		The idea is to construct the $\vecb c_j$'s inductively with the additional condition that $\maxnorm{\vecb c_j} \leq (2r)^t  {\rho'} n$.
		Note that this gives~\ref{itm:clm-split-simple}, since $(2r)^t  {\rho'}  n \leq \sqrt {\rho'} n$.
		Recall $T_1,\dots, T_t$ are the tight components of $J'$ and that $L(J')$ is the graph with vertex set $\{T_1,\dots, T_t\}$ and $T_iT_j$ is an edge if $V(T_i)\cap V(T_j) \neq \emptyset $ and $i\neq j$.
		Moreover, we call the union of a set of $T_j$'s that form a connected component in $L(J')$, a loose component, and denote these by $L_1,\dots,L_\ell\subset J'$.
		By \eqref{equ:vecc-sums}, we have ${\maxnorm{\vecb c} \leq} \onenorm{\vecb c}\leq r{\rho'} n$ and   $\sum_{i\in V(L_q)}\vecb c(i) = 0$ for all $q\in[\ell]$.
		{In addition, we have $\sum_{i \in [r]}\vecb c(i) = 0$.}
		For simplicity, let us first assume that $\ell = 1$.

		We prove the claim by induction on $t$.
		The base case is trivial as we then set $\vecb c_1 =\vecb c$.
		Suppose now that the claim holds for $t' < t$.
		Consider a leaf {$T_t \in V(L(J'))$} of a spanning tree in $L(J')$, which is  adjacent (in the spanning tree) to, say, {$T_{t-1}\in V(L(J'))$}.
		By definition of $L(J')$, there is some $i^\ast \in V(T_t) \cap V(T_{t-1})$.
		We then define $\vecb c_t \in \INTS^{V(R)}$ by setting {$\vecb c_t (i^\ast) = \vecb c(i^\ast) - \sum_{i \in V(T_t)} \vecb c(i)$, $\vecb c_t (j)= \vecb c(j)$ for all $j \in V(T_t)\sm \{i^\ast\}$ and $\vecb c_t (j)=0$ otherwise.}
		Clearly, {$\sum_{i \in [r]}\vecb c_t(i) = 0$.}
		Moreover, $\maxnorm{\vecb c_t} \leq v(T_t)\maxnorm{\vecb c}  \leq r \maxnorm{\vecb c}$.

		Now let ${\tilde J}$ be obtained from $J'$ by deleting the edges of $T_t$ and the vertices contained in no other component than $T_t$.
		Since $T_t$ was a leaf in a spanning tree of  $L(J')$, it follows that $L({\tilde J})$ is still connected.
		Let $\vecb {\tilde c} = \vecb c - \vecb c_t$, and note that {$\maxnorm{\vecb {\tilde c}} \leq 2v(T_t)\maxnorm{\vecb c} \leq 2r\maxnorm{\vecb c}$.}
		Hence we may use induction where $\vecb {\tilde c}$ plays the role of $\vecb c$ to obtain $\vecb c_1,\dots,\vecb c_{{t-1}}\in \INTS^{V(R)}$ satisfying the statement (for $t'=t-1$).
		In particular, we have $\maxnorm{\vecb c_j} \leq (2r)^{t-1}{\maxnorm{\vecb {\tilde c}}}  \leq (2r)^{t}\maxnorm{\vecb c}$ for all $j \in [t-1]$.
		Thus  $\vecb c_1,\dots,\vecb c_{t}$ present the desired solution.
		This concludes the analysis for $\ell = 1$.

		When $\ell > 1$ we can use the same arguments separately on each of the loose components, which is possible since $\sum_{i\in V(L_q)}\vecb c(i) = 0$ for all $q\in[\ell]$ and the loose components are vertex-disjoint.
	\end{proof}
	This completes the proof of \cref{lem:lemma-for-H}.
\end{proof}

\section{Conclusion}\label{sec:conclusion}

In this paper, we conduct a systematic investigation of the question {of} when a dense host graph~$G$ contains a (potentially spanning) guest graph $H$ where $H$ has bounded maximum degree and where each component of $H$ may be huge, but of smaller magnitude than the order of $G$.
In particular, we study this topic beyond the usual setting where one seeks for the optimal minimum degree bound that guarantees the existence of a copy of $H$ in~$G$.
This has several advantages as for example it covers many other natural conditions a host graph may satisfy apart from (optimal) minimum degree conditions
and it makes our proof more modular.
Nevertheless, our paper is by no means the end of the story and there are a few interesting questions that remain.

We applied our framework to degree conditions and degree sequences.
	A related {direction}, introduced by Ore~\cite{Ore60}, imposes a lower bound on the sum of the degrees of any two non-adjacent vertices.
	Such Ore-type conditions have been investigated for perfect matchings~\cite{KK08} and $F$-tilings~\cite{KOT09}.
	Our framework generally applies to such problems after minor workarounds (see~\cite{LS23}).
	Hence, one can extend the results of Kühn, Osthus and Treglown~\cite{KOT09} to mixed tilings using already established structural insights similar as in the proof of \cref{thm:deg-seq-komlos}.
	On the other hand, in its current state our framework is unable to reproduce these results for $F$-tilings with fixed $F$.
	This is because the Ore-type `threshold' that guarantees that every vertex of the host graph lies in a copy of $F$ can be significantly lower than the `threshold' for the linkage property, in the cases when $F$ is not a bottle.
	To overcome this limitation, we can define \emph{$F$-frameworks} by replacing the linkage condition with the property that every vertex is on a copy of $F$.
	It is then not hard to see that the proof of \cref{thm:frameworks} can be modified to show that robust $F$-frameworks lead to perfect $F$-tilings.
	Given this, the results of Kühn, Osthus and Treglown~\cite{KOT09} can be fully replicated {extending it to $F$ with at most $o(n)$ vertices}.

	As mentioned before, the original K\"uhn--Osthus Tiling Theorem is more precise in comparison to our work as it is stated in terms of an (additive) constant error.
	It is plausible that our results can potentially be refined further via a stability analysis, which could lead to sharper error bounds in certain situations and, in particular, replicate the results of K\"uhn and Osthus for minimum degree conditions.

Turning to more general questions,
	the literature on spanning substructures has so far identified three natural barriers to finding perfect tilings, which correspond to obstacles in space, divisibility and local conditions.
	Very recently, Freschi and Treglown~\cite{FT22} asked whether (in an approximate sense) it can be proved that these are the only obstructions.
	Our work solves this `meta problem' for mixed tilings, since the space, divisibility and linkage conditions of \cref{def:tiling-framework} certify that each of the above obstacles can be (individually) overcome.

In a similar vein, there has also been recently quite some research activity on graph tilings in related settings such as directed graphs~\cite{CDM+18,Tre16}, ordered graphs~\cite{BLT20,FT22}, multipartite graphs~\cite{MS17a}, random (hyper)graphs~\cite{FK20} and in the discrepancy setting~\cite{BCP+21}.
We expect that many ideas of our paper also apply in these cases and think it is an interesting and fruitful task to set up respective frameworks.

While tilings of graphs have been an integral and well-studied question for graphs, not so much is known for uniform hypergraphs~\cite{KM15,LM13a,Myc16}; even an analogous Hajnal--Szemer\'edi result seems currently out of reach.
It would be very interesting, but potentially also very challenging, to approach hypergraph tilings from a framework viewpoint.
The idea would be to identify similar, rather abstract, conditions as in \cref{def:tiling-framework} and prove that these are sufficient to find a mixed tiling.
Once such a result is obtained, it only remains to verify that these abstract conditions in fact are satisfied given the host hypergraph admits certain minimum degree properties;
however, this might turn out to be the most difficult part.
{For recent progress in this direction see the work of Lang~\cite{Lan23}.}

\section*{Acknowledgments}
We would like to thank Andrew Treglown for {stimulating discussions and} bringing the work of Ishigami~\cite{Ish02,Ish02b} to our attention.
{Finally, we express our gratitude to the anonymous referees who provided a large number of helpful remarks that improved the presentation and correctness of the paper.}

\bibliographystyle{amsplain}
\bibliography{bibliography}

\appendix

\section{Reformulating Koml\'os' conjecture}\label{sec:reformulating-komlos-conjecture}

In this section, we give an alternative formulation of Koml\'os' conjecture from \cref{sec:introduction}.
More precisely, we derive inequality~\eqref{eq:Kom_conj2} from inequality~\eqref{eq:Kom_conj}.
Recall that in this context $H$ is a graph on $n-f(w)$ vertices with components $F_1,\ldots, F_r$ where $f(\cdot)$ is a function of the maximum component order~$w$.
Suppose that $\chi(F_i) = \lceil \crit(H) \rceil$ for all $i\in [r]$.
Let $v= v(H)$, $v_i=v(F_i)$, 
$\alpha_i=\alpha(F_i)$ and recall that
$\crit(F_i)=\frac{k-1}{1-\alpha_i}$.
Then 
\begin{align*}
	\sum_{i\in [r]}v(F_i)\left(1- \frac{1}{\crit(F_i)}\right)
	&=\sum_{i \in [r]} v_i \left(1-\frac{1-\alpha_i}{k-1}\right) \\
	&=v - \frac{1}{k-1}\sum_{i \in [r]}  v_i(1-\alpha_i) \\
	&= v - v\frac{1-\alpha(H)}{k-1} \\
	&=\left(1-  \frac{1}{\crit(H)} \right)v.
\end{align*}
So from \eqref{eq:Kom_conj} and $v = n-f(w)$, we obtain that 
\begin{align*}
	\MT(n,H) &\leq \left(1-  \frac{1}{\crit(H)} \right)\NEW{v} + f(w)  
	= \left(1-  \frac{1}{\crit(H)} \right) n + \frac{f(w)}{\crit(H)}.
\end{align*}

\section{Universally dense graphs}\label{sec:uni-dense-explanation}

In this section, we sketch the proofs of \cref{cor:dense-many-well-connected-cliques,prop:dense-inseperable-matching}, which we restate here for convenience of the reader.

{\cordensemanywellconnectedcliques*}

{\propdenseinseperablematching*}

The formal arguments can be found in the work of Lang and Sanhueza-Matamala~\cite[Corollary 5.27 and Proposition 5.28]{LS23}.
For the first result, a straightforward calculation shows that $G[S]$ is $(\sqrt{\rho},d)$-dense.
\cref{cor:dense-many-well-connected-cliques} thus follows from the following lemma due Ebsen, Maesaka, Reiher, Schacht and Schülke.

\begin{lemma}[{\cite[Lemma 2.1]{EMR+20}}] \label{lemma:EMRSScliquecounting}
	Let $k \ge 1$, $d \in [0,1]$ and $\rho > 0$.
	Every $(\rho, d)$-dense graph $G$ on $n$ vertices contains at least $( d^{\binom{k}{2}} - (k-1)k\rho) n^k$ ordered $k$-cliques.
\end{lemma}

We now turn to the proof of \cref{prop:dense-inseperable-matching}.
By \cref{cor:dense-many-well-connected-cliques}, a largest $k$-clique-tiling covers already all but $\eta^{1/(2k)} n$ vertices.
From this, one can easily construct a perfect fractional $k$-clique-tiling using an absorption argument.
The idea is to first reserve a small set $A$ of pairwise disjoint $k$-cliques such that every vertex of $G$ forms a $(k+1)$-clique with at least $\eta^{1/(2k)} n$ of the elements of $A$.
This can be done using a standard probabilistic argument and the fact that every vertex has many $k$-cliques in its neighbourhood, which is guaranteed by \cref{cor:dense-many-well-connected-cliques} and the minimum degree of $G$.
We then pick an almost perfect $k$-clique-tiling $H$ in $G-A$.
To finish, the remaining vertices are absorbed into $A$ by matching each vertex to a (private) $k$-clique and picking a perfect fractional $k$-clique-tiling in the resulting $(k+1)$-clique.

\section{Robustness and regularity}\label{sec:app-regularity}

In this section, we show \cref{prop:supersaturated-linkage,prop:inherit-framework,prop:dense-tight-component}, which we restate in each case for convenience of the reader.
The proofs of the first two results are direct adaptations of former work~\cite{LS23}.
For \cref{prop:supersaturated-linkage}, we shall use the following lemma, which is a particular case of the Graph Removal Lemma~\cite{EFR86}.

\begin{lemma} \label{lemma:graphremoval}
	For every $\eps >0$ and $k \in \NATS$, there exists $\delta >0$ and $n_0 \in \NATS$ such that for all graphs $G$ on $n\geq n_0$ vertices the following holds:
	if one needs to delete more than $\eps n^2$ edges to eliminate all $k$-cliques of $G$, then $G$ contains at least $\delta n^k$ $k$-cliques.
\end{lemma}

{\propsupsatlink*}

\begin{proof}[Proof of \cref{prop:supersaturated-linkage}] 
	Let $v \in V(G)$ be an arbitrary vertex.
	Let $G_v \subseteq G -  v$ be the subgraph formed by all the edges which intersect $N_G(v)$.
	Note that each edge in $K_k(G_v)$ intersects $N_G(v)$ in at least $k-1$ vertices and does not contain $v$,
	and thus corresponds to a linked edge of $v$ in $K_k(G)$.
	It is therefore enough to show that $G_v$ contains at least $\mu' n^k$ many $k$-cliques.
	
	We want to show that, after the removal of any $(\mu^2/8) n^2$ edges from $G_v$, at least one $k$-clique remains in $G_v$.
	Indeed, suppose we remove $(\mu^2/8) n^2$ edges $Y$ from $G_v$.
	A double-counting argument reveals that at most $(\mu/2) n$ vertices had more than $(\mu/2) n$ edges removed from them, let $X$ be those vertices.
	Note that, in particular since $v \notin V(G_v)$, we have $v \notin X$.
	In $G$, remove the edges corresponding to $Y$ and also remove the vertices corresponding to $X$,
	to obtain a subgraph $G' \subseteq G$.
	Note that $v \in V(G')$.
	Clearly, $G'$ is a $\mu$-approximation of $G$.
	Since $G$ is $\mu$-robust $(\chi,\rho;t,\ell)$-tiling framework,
	we deduce that $G'$ is a $(\chi,\rho;t,\ell)$-tiling framework.
	It follows that $v$ is linked to an edge in $K_k(G')$.
	By construction, this corresponds to a $k$-clique in $G_v - Y$, as required.
	
	Assuming $n$ is sufficiently large, what we have shown implies that $G_v$ satisfies the assumptions of \cref{lemma:graphremoval}.
	Thus there exists $\mu' > 0$ such that the number of $k$-cliques in $G_v$ is at least $\mu' n^k$, as desired.
\end{proof}

{\propinheritframework*}

\begin{proof}[Proof of \cref{prop:inherit-framework}]
	{Introduce $\mu'$ with $1/r\ll \mu' \ll 1/\chi, \mu$.}
	Let $R' \subset R$ be an arbitrary $\mu$-approximation of $R$.
	Let $\cV'$ be obtained from $\cV$ by deleting the clusters whose indices are not in $V(R')$.
	We obtain $G' \subset G$ by deleting all vertices of clusters $V_i$ with $i \notin V(R')$ and all edges between $V_i$ and $V_j$ whenever $ij \notin E(R')$.
	Let $m$ denote the common cluster size of $\cV$.
	Since $n=mr$ and $1/n \ll 1/r \ll \mu$, it follows that $G'$ is a  $(2\mu)$-approximation of $G$.
	Since $G$ is {$(4\mu)$-robust} $(\chi,\rho;t,\ell)$-tiling framework, $G'$ is a {$(2\mu)$-robust} $(\chi,\rho;t,\ell)$-tiling framework.
	
	We have to show that  $R'$ is a $(\chi,\rho;t,\ell)$-tiling framework.
	Observe that the divisibility (as detailed in \cref{def:tiling-framework}) follow immediately from the respective property of $G'$ and the definition of $R$-partitions.
	
	For the linkage property, note that by \cref{prop:supersaturated-linkage}, every vertex of~$G'$ is linked to at least $\mu' n^k$ edges in $K_k(G')$.
	There are at most $r^{k-1}m^{k} = (1/r) n^k$ edges in $K_k(G')$, which have two or more vertices in the same cluster of $\cV'$.
	Similarly, there are at most $r^{k-1} m^{k} = (1/r) n^{k}$ edges in $K_k(G')$, which intersect a particular cluster.
	Since $\mu' > 2/r$, it follows that every vertex~$v$ of $K_{k}(G')$ is linked to an edge, which has at most one vertex in each cluster and does not intersect with the cluster of $v$.
	It follows that every vertex of $R'$ is linked in $K_{k}(R')$.
	
	For the space property, consider a $(1-\rho)$-perfect fractional $B_\chi$-tiling $\vecb w\colon  \Hom(B_\chi;G') \to [0,1]$.
	We then construct a $(1-\rho)$-perfect fractional $B_\chi$-tiling $\vecb w'\colon  \Hom(B_\chi;R') \to [0,1]$  as follows.
	For each $\phi' \in \Hom(B_\chi;R')$, we set $\vecb w' ( \phi' )$ to be the sum of the weights $\vecb w(\phi )/m$ over all  $\phi \in \Hom(B_\chi;G')$ that send, for each $i \in V(R)$, the same number of vertices of $B_\chi$ to $V_i$ as $\phi'$ sends to $i$.  
\end{proof}

{\propdensetightcomponent*}

The proof of \cref{prop:dense-tight-component} has also been adapted from former work~\cite[Proposition 3.3]{LS20}.

\begin{proof}[Proof of \cref{prop:dense-tight-component}]
	Let $T_1,\dots,T_s$ be the tight components of $G$.
	Fix $1 \leq q \leq s$ which maximises $e(T_q) / e_{}(\partial T_q)$.
	By the definition of tight components, we have $\sum_{i=1}^s e(T_i) = e(G)$ and $\sum_{i=1}^s e(\partial T_i) = e_{}(\partial G)$.
	It follows that
	\begin{align*}
		e(G) = \sum_{i \in [s]} e(T_i) \leq \frac{e(T_q)}{e_{}(\partial T_q)} \sum_{i \in [s]}e_{}(\partial T_i) = \frac{e(T_q)}{e_{}(\partial T_q)}  e(\partial G).
	\end{align*}
	
	Now let $T=T_q$.
	By the above, we have $e(T) \geq e(\partial T) e(G) / \binom{n}{k-1} \geq e(\partial T) (1+\eps )\mu \binom{n}{k} / \binom{n}{k-1}$.
	Let $x \in \REALS$ satisfy $e(T) = \binom{xn}{k}$.
	It follows by the \cref{thm:KK} that $e(\partial T) \geq \binom{xn}{k-1}$.
	Hence $\binom{xn}{k}  \geq \binom{xn}{k-1} (1+\eps ) \mu \binom{n}{k} / \binom{n}{k-1} $.
	Using that $\binom{yn}{\ell} \approx y^\ell \binom{n}{\ell}$ for $\ell \leq k$ and our choice of $\eps$, this solves to $x\geq \mu$.
	Hence the result follows.
\end{proof}

\section{Construction}\label{sec:constructions}

In this section we give constructions related to \cref{thm:extended-KO}.
The \emph{square of a cycle} is obtained from a cycle by connecting every vertex $v$ with the vertices of distance $2$ to $v$ (in the cyclic order).

\subsection{Large tiles}\label{sec:sizes-tiles}

In the following, we give a construction, which shows that \cref{thm:extended-KO} fails
if tiles of order $10\sqrt{n}$ are permitted, even in the case when all components are isomorphic.

Choose (large) $w \in \NATS$ divisible by $72$ and $n \in  \NATS$ divisible by $3 w$ such that $10\sqrt{n}  \leq  w$.
Observe that there exists a $3$-chromatic graph~$F$ of order $w$ with {$\Delta(F) \leq 12$} and a unique proper $3$-colouring (up to relabelling colours) such that the colour classes have size $w/4-1, w/4, w/2+1$.
{So $F\in  \cFcr$ and $\crit(F) \leq 2 / (1-1/4)=8/3$.}
(Indeed, consider the square of a cycle on $2w/3$ vertices and carefully blow up some vertices by $2$ or $3$.)

Let $G$ be a complete tripartite graph on $n$ vertices with parts of order $n/3$, $n/3+w/8$ and $n/3-w/8$.
{Then $\delta(G) \geq (2/3-o(1))n$, which is much larger than $5n/8 \geq (1-1/\crit(F))n$.}

A perfect $F$-tiling $H$ of $G$ would contain fewer than $\sqrt{n}/10 \leq w/100$ copies of $F$, and
so in any proper $3$-colouring of $H$ and any colour $i$,
the number of vertices coloured with $i$ satisfies $0 \pm w/100 \bmod w/4$.
This shows that $H$ is not a subgraph of $G$.

\subsection{Many tight components with divisibility obstruction}
Let $m \in \NATS$ be even.
Let $F'$ be a square of a cycle of length $6$.
Note that there is exactly one (proper) $3$-colouring of $F'$ (up to relabelling) with each colour class having size $2$.
Let $F$ be obtained by blowing up each vertex of two colour classes of~$F'$ by $m$.
We have $\Delta(F) = 4m$.
Moreover, it follows that $\alpha(F) = \frac{2}{4m+2} = \frac{1}{2m+1}$ and hence $1-\frac{1}{\crit(F)} = \frac{m+1}{2m+1}$.

For a vertex $x$, let $G'$ be the union of triangles $T_1,\dots,T_{m/2}$ such that each triangle contains~$x$ and any two distinct triangles are vertex-disjoint apart from $x$.
Denote the vertices of $T_i$ by $\{x,v_i,v_i'\}$.
For $d \in \NATS$ divisible by $v(F)=4m+2$, let $G$ be a graph obtained from $G'$ by 
\begin{itemize}
	\item blowing up  $x$ by $d(m-1)$,
	\item blowing up  $v_i$ by $d$ for each $1 \leq i \leq m/2$,
	\item blowing up $v_i'$ by $d-1$ for each $1 \leq i \leq m/4$,
	\item blowing up $v_j'$ by $d+1$ for each $m/4+1 \leq j \leq m/2$.
\end{itemize}
So $G$ has order $n:= d(2m-1)$.
For $d$ large enough, it follows that $$\delta(G) = dm-1 = \frac{m-1/d}{2m-1} n > \frac{m+1}{2m+1} n = \left(1- \frac{1}{\crit(F)}\right)n.$$

Note that the blown-up triangles present the tight components of $G$.
Moreover, each copy of~$F$ in $G$ must be contained in one of the blown-up triangles.
On the other hand, each blow-up of a triangle has one part of odd order (as $d$ is even), which makes it impossible to cover its vertices with (vertex-disjoint) copies of $F$.
Hence $G$ contains $m/2$ tight components, which each come with their own divisibility obstruction to a perfect $F$-tiling of $G$.
Lastly, we note that in this construction the number of tight components grows with $\Delta(F)$.

\end{document}